\documentclass[11pt,a4paper,reqno]{amsart} 

\usepackage [english]{babel}  

\DeclareFontFamily{U}{mathb}{\hyphenchar\font45}
\DeclareFontShape{U}{mathb}{m}{n}{
      <5> <6> <7> <8> <9> <10>
      <10.95> <12> <14.4> <17.28> <20.74> <24.88>
      mathb10
      }{}
\DeclareSymbolFont{mathb}{U}{mathb}{m}{n}

\DeclareMathSymbol{\sqbullet}{1}{mathb}{"0D}

\allowdisplaybreaks 

\usepackage{amsmath,mathtools}
\usepackage{amssymb}
\usepackage{euscript}
\usepackage{bm}
\usepackage{graphicx}
\usepackage[all]{xy}
\usepackage[dvipsnames]{xcolor}
\usepackage[colorlinks=true,linktocpage=true,pagebackref=false, citecolor=black,linkcolor=black]{hyperref}
\usepackage{pifont}
\usepackage{tikz,pgfplots}
\pgfplotsset{compat=1.10}
\usepgfplotslibrary{fillbetween}
\usetikzlibrary{cd,arrows,decorations.pathmorphing,backgrounds,automata,positioning,fit,matrix}
\usepackage{caption,subcaption}
\usepackage{comment}
\usepackage{xcolor}

\usetikzlibrary{matrix}
\usetikzlibrary{decorations.pathreplacing, calc, fit, backgrounds}

\pgfplotsset{soldot/.style={color=black,only marks,mark=*}} \pgfplotsset{holdot/.style={color=black,fill=white,only marks,mark=*}}

\newtheorem{thm}{Theorem}[section]

\newtheorem{lem}[thm]{Lemma}
\newtheorem{prop}[thm]{Proposition}

\newtheorem{cor}[thm]{Corollary}

\newtheorem{ass}[thm]{Assumption}
\newtheorem{quest2}[thm]{Question}

\theoremstyle{definition}
\newtheorem{defn}[thm]{Definition}

\newtheorem{remark}[thm]{Remark}

\newtheorem{example}[thm]{Example}

\theoremstyle{remark}

\numberwithin{equation}{section}
\numberwithin{figure}{section}

 \newcommand{\R}{{\mathbb R}}
\newcommand{\Q}{{\mathbb Q}} \newcommand{\C}{{\mathbb C}}
 
\newcommand{\HH}{{\mathbb H}} 
\newcommand{\sph}{{\mathbb S}} 
 
 \renewcommand{\O}{{\mathbb O}}

 \newcommand{\Cont}{{\mathcal C}}

\newcommand{\Ff}{{\EuScript F}}

\newcommand{\Bb}{{\EuScript B}}

\newcommand{\Qq}{{\EuScript Q}}

\newcommand{\ord}{\operatorname{ord}}

\renewcommand{\Re}{\operatorname{Re}}
\renewcommand{\Im}{\operatorname{Im}}

\newcommand{\x}{{\tt x}}  
\newcommand{\z}{{\tt z}} \renewcommand{\t}{{\tt t}}
 
 \newcommand{\e}{{\tt e}}
  \newcommand{\q}{{\tt q}}

\usepackage{scalerel}

\begin{document}

\title[Division algebras of slice-Nash functions]{Division algebras of slice-Nash functions}
\begin{abstract}
The purpose of this paper is to introduce the notion of Nash functions in the context of slice regular functions of one quaternionic or octonionic variable. We begin with a detailed analysis of the possible definitions of Nash slice regular functions which leads us to the definition of \textit{slice-Nash} function proposed in this paper (and which we strongly believe to be the natural generalisation of the classical real and complex Nash functions to this context). Once the `correct' definition of slice-Nash functions has been established, we study their properties with particular focus on their finiteness properties. These finiteness properties position this new class of slice-Nash functions as an intermediate class between the class of slice regular functions and the class of slice polynomials function, in analogy with the classical real and complex case. We also introduce semiregular slice-Nash functions, in analogy with meromorphic Nash functions, and study their finiteness properties. 
\end{abstract}

\date{17/10/2025}
\author{Cinzia Bisi}
\address{Cinzia Bisi, Dipartimento di matematica e informatica,
Via Machiavelli 30, Università di Ferrara, 44121 Ferrara (ITALY). {\tt ORCID: 0000-0002-4973-1053}}
\email{cinzia.bisi@unife.it}
\thanks{The first author was partially supported by GNSAGA of INdAM and by PRIN \textit{Variet\'a reali e complesse: geometria, topologia e analisi armonica.}}

\author{Antonio Carbone}
\address{Antonio Carbone, Dipartimento di Scienze dell'Ambiente e della Prevenzione, Palazzo Turchi di Bagno, C.so Ercole I D'Este, 32, Università di Ferrara, 44121 Ferrara (ITALY)}
\email{antonio.carbone@unife.it}

\maketitle

\tableofcontents

\section{Introduction}

The purpose of this paper is to introduce the notion of Nash functions in the context of slice regular functions of one quaternionic or octonionic variable.

\subsection{Hypercomplex analysis and geometry}

Since the beginning of the last century, there have been several attempts to determine a suitable class of functions of one quaternionic variable that would had played the same role as the holomorphic functions of one complex variable. One of the first successful notions of `regular' quaternionic functions is due to Fueter \cite{fu}, who defined regular quaternionic functions as those functions $f$ that satisfy the so-called Cauchy-Riemann-Fueter equation:
$$
\frac{\partial f}{\partial \x_0}+i\frac{\partial f}{\partial \x_1}+j\frac{\partial f}{\partial \x_2}+k\frac{\partial f}{\partial \x_3}=0.
$$
The theory of Fueter regular functions has been widely developed, and we refer the reader to \cite{su} for further information and details. Unfortunately, the class of Fueter regular functions does not contain the identity map $q\mapsto q$, and therefore does not include polynomials and power series. 

Following some ideas of Cullen \cite{cu}, Gentili and Struppa \cite{gs, gs2} introduced a notion of regularity for quaternionic functions, called \textit{slice regularity} (see \cite[Def.1.2]{gs}). The class of slice regular functions includes (suitable) polynomials and convergent power series. After the work of Gentili and Struppa, the theory of slice regular functions has been widely developed and studied in the last 20 years. This theory has been extended to octonions \cite{gs3} and to Clifford algebras \cite{csast}. Using a different approach to slice regularity, based on the concept of stem functions, Ghiloni and Perotti \cite{gp0,gp} extended the theory to any real alternative $*$-algebra of finite dimension. This approach also allows one to characterise slice regularity using a system of global differential equations \cite{gp5} and to extend the theory of slice regular functions in one variable to functions of several variables \cite{gp6}.

Slice regular functions share many properties with holomorphic functions - they satisfy the identity principle, Cauchy formulas, and admit (locally) suitable expansions as convergent power series. Several classical results of complex analysis have their quaternionic counterpart. For instance, the open mapping theorem \cite{gst3,gst2}, the Weierstrass factorisation theorem \cite{gv}, the Schwarz-Pick lemma \cite{bs}, the Landau theorem \cite{bs2}, the Picard theorem \cite{bw}, the Runge approximation theorem \cite{bw2}, the Brolin theorem \cite{bd}, the Jensen formula \cite{alb,bw3}, the classification of the automorphisms of the algebra of functions \cite{bw6}, and many others. We refer the reader to the monograph \cite{gsst} and the references therein for further information and details. 

While the `analytic properties' of slice regular functions are nowadays quite well understood (at least for functions in one variable), the situation changes radically when we talk about their `geometrical properties'. Except for some positive results (for instance the reader may check the articles \cite{ap, ap2, ab, bg, gss, gss2,  gsv} and the references contained therein), very little is known on how to define and study fundamental geometrical objects - such as manifolds, algebraic sets or varieties, analytic sets or spaces, etc. - in this context. There are two main reasons, in our opinion, that explain these difficulties:
\begin{itemize}
\item Besides the foundational work of Ghiloni and Perotti \cite{gp6}, the theory of slice functions in several variables (unlike the theory of one variable) is still not well developed nor understood.
\item The non-commutative (and often non-associative) nature of the involved algebra makes it intrinsically `difficult' to `do geometry' (whatever that means) in this context.
\end{itemize}

\subsection{Nash geometry}\label{nashgeometry} Nash functions were introduced by John Nash in his groundbreaking paper \cite{nas} published in 1952, where he showed that a compact $d$-dimensional smooth manifold $M$ is diffeomorphic to a union of nonsingular connected components of a real algebraic subset of $\R^{2d+1}$. A (real) Nash function is an analytic function that is algebraic over the ring of polynomials (see Definition \ref{defNash} for the precise definition). In 1973, Tognoli \cite{to} improved Nash's result showing that $M$ is actually diffeomorphic to a whole nonsingular algebraic subset of $\R^{2d+1}$. We refer the reader to \cite[\S 14.1]{bcr} for further information and details. 

Nash functions are the smallest subring of the ring of analytic functions that contains the polynomials and satisfy the implicit function theorem. In particular, there is a well established theory of (affine) Nash manifolds. We refer the reader to \cite[\S8]{bcr} and \cite{sh} for a detailed introduction to Nash functions and Nash manifolds. Nash manifolds and maps form a category. The Nash category enjoys many advantages over other categories: the analytic category of real analytic manifolds and analytic maps, and the algebraic category of affine nonsingular real algebraic sets and regular maps. \textit{The Nash category lies in between these two categories.} The analytic category has too many objects because any $\Cont^{\infty}$ manifold admits a unique analytic manifold structure, and the algebraic category has too few morphisms to construct a desirable theory. On the other hand, the Nash category has suitably many objects because any affine Nash manifold is Nash diffeomorphic to an affine nonsingular real algebraic set. Moreover, it has enough but not too many morphisms to construct a good theory. It is also worthwhile to notice that the Nash category shares many of the \textit{`finiteness properties'} of the algebraic category. 

As an example of result that enlighten this kind of `intermediate' behaviour of the Nash category, we cite the following theorem of Coste, Ruiz and Shiota. In \cite{crs1}, the authors show a global version of Artin’s approximation theorem \cite{ar}, which implies the following approximation result.

\begin{thm}\label{alg}
Let $X\subset \R^n$ be a compact algebraic set (resp. a compact Nash manifold) and $Y\subset \R^n$ be any algebraic set (resp. Nash set). Then every real analytic map $f:X\to Y$ can be uniformly approximated by real Nash maps $g:X\to Y$.
\end{thm}

The previous theorem does not hold true if one tries to approximate analytic maps by regular maps, because, if the target space $Y$ is `generic' in a suitable sense, then regular maps from a compact nonsingular real algebraic set $X$ to $Y$ are `very few', see \cite{gh1,gh2}. It is worth noting that compact real algebraic sets with at least two points do not have tubular neighbourhoods with regular retractions \cite[Thm.2]{gh}.

Approximation of analytic objects by algebraic objects has been extensively studied also in the complex case by many mathematicians. For instance, the argument of Nash was adapted to the complex case (with more restricting hypotheses) by Stout \cite{st}, while Lempert showed in \cite{le} the counterpart of Theorem \ref{alg} for the complex setting. In the complex case too, the Nash category (see \cite[\S1]{t} for an introduction to complex Nash functions) lies in between the analytic category and the algebraic category and it is the `suitable' category to study algebraic approximation. 

\subsection{Slice-Nash functions}

As already mentioned, the purpose of this work is to introduce the class of Nash functions in the context of slice regular functions of one quaternionic or octonionic variable and to study their properties, with particular focus on their `finiteness properties'. In a forthcoming paper, we will introduce Nash functions in several quaternionic or octonionic variables. The very first question we asked ourselves before starting this work was the following:

\begin{quest2}\label{q1}
Are Nash functions `interesting' (whatever that means) in the context of hypercomplex analysis and geometry?
\end{quest2}

As explained in \S\ref{nashgeometry}, the Nash category lies in between the analytic category and the algebraic category and helps to study (geometrical) properties of both these categories, as well as their relationships. In fact, as the classical real and complex case show, having a subclass of the class of analytic functions that shares finiteness and tame properties with algebraic functions often helps to enlighten properties of both the algebraic and analytic objects. Thus, our hope is that defining `the correct' class of Nash functions in this context will help to better understand the geometrical properties of slice regular functions and slice polynomials (especially given that the theory will also be developed in several variables in the near future). From another perspective, Savi, in the recent work \cite{sa}, studied quaternionic and octonionic algebras from a model theoretic point of view. Slice regular functions which are definable in the theories he introduced are precisely those slice regular functions whose real components are real Nash functions. 

In light of these considerations, we strongly believe that Question \ref{q1} has an affirmative answer. Once established that, a natural question arises.

\begin{quest2}\label{q2}
What is the `correct' definition of Nash functions in the quaternionic (or octonionic) setting?
\end{quest2}

A real (resp. complex) Nash function is a real analytic (resp. holomorphic) function which is algebraic over the ring of polynomials with real (resp. complex) coefficients (see Definition \ref{defNash} below). In analogy with the real and complex cases, one would like to define quaternionic (or octonionic) Nash functions as those slice regular functions which are, in addition, algebraic over the ring of polynomials with quaternionic (or octonionic) coefficients. Being $\HH$ and $\O$ not commutative and $\O$ not even associative, it is not clear what should be the meaning of: \textit{`being algebraic over the polynomials'}. Moreover, there are several kinds of polynomials one may consider - slice polynomials, slice polynomials endowed with the multiplicative structure induced by the slice product, and mixed polynomials. In \S\ref{buonadef}, we analyse all these possibilities and explain why quaternionic (or octonionic) Nash functions cannot be defined as those regular functions which are `algebraic over polynomials' if we want them to have all the good properties discussed above. That is why, in Definition \ref{defslicenash}, we define \textit{slice-Nash functions} as those slice regular functions induced by stem functions whose complex components (with respect to any base, Proposition \ref{indipendente}) are complex Nash functions. In order to study the properties of this new class of functions, in Theorems \ref{char} and \ref{char2}, we characterise slice-Nash functions in terms of their components given by the splitting lemma and in terms of their real components. Putting these results together, we obtain the following characterisation for slice-Nash functions (we refer the reader to \S\ref{sectionchar} for the precise definitions and notation). Let $A$ be either the algebra of quaternions $\HH$ or the algebra of octonions $\O$ and $\Omega\subset A$ a (non-empty) open circular set.

\begin{thm}
Let $f:\Omega\to A$ be a slice function. The following statements are equivalent:
\begin{itemize}
\item[{\rm(i)}] $f$ is a slice-Nash function on $\Omega$.
\item[{\rm(ii})] There exist $I\in\sph_A$ and a splitting base $\Bb_I:=\{I_0:=1,I, I_1,II_1,\ldots,I_{u_A},II_{u_A}\}$ of $A$ associated to $I$ such that the functions  $f_1^{I,I_k},f^{I,I_k}_2:\Omega_I\to \C_I$ are $\C_I$-Nash functions on $\Omega_I$ for each $k=0,\ldots,u_A$. 
\item[{\rm(iii)}] For each $I\in\sph_A$ and each splitting base $\Bb_I:=\{I_0:=1,I, I_1,II_1,\ldots,I_{u_A},II_{u_A}\}$ of $A$ associated to $I$ the functions $f_1^{I,I_k},f^{I,I_k}_2:\Omega_I\to \C_I$ are $\C_I$-Nash functions on $\Omega_I$ for each $k=0,\ldots,u_A$. 
\item[{\rm(iv})] $f$ is slice regular and there exists a base $\Bb$ of $A$ as a $\R$-vector space such that the real components $f_\ell^{\Bb}:\Omega_{\Bb}\to \R$ of $f$ with respect to the base $\Bb$ are real Nash functions for each $\ell=0,\ldots,d_A$.
\item[{\rm(v)}] $f$ is slice regular and for each base $\Bb$ of $A$ as a $\R$-vector space the real components $f_\ell^{\Bb}:\Omega_{\Bb}\to \R$ of $f$ with respect to the base $\Bb$ are real Nash functions for each $\ell=0,\ldots,d_A$.
\end{itemize}
\end{thm}

Slice-Nash functions share many properties with classical real and complex Nash functions. For instance, the slice derivatives of a slice-Nash function are themselves slice-Nash (Proposition \ref{slicederProp}) and slice-Nash functions are closed under taking finite sums and slice products (Proposition \ref{ringNash}). In Theorem \ref{*Nash}, we obtain the following result about the algebraic structure for slice-Nash functions (compare this result with Theorem \ref{sliceregularstructure}). We refer the reader to \S\ref{preliminaries} and Definition \ref{defNash} for the precise definitions and notations. 

\begin{thm}\label{introduDivAlg}
The set $\mathcal{SN}_A(\Omega)$ of all slice-Nash functions on $\Omega$ is an alternative $*$-subalgebra of the alternative $*$-algebra $\mathcal{SR}_A(\Omega)$ of slice regular functions on $\Omega$ endowed with $+$, $\cdot$, $\cdot^c$. If $A=\HH$, then $\mathcal{SN}_\HH(\Omega)$ is associative. Moreover,
\begin{itemize}
\item[\rm{(i)}] if $\Omega$ is a symmetric slice domain, then $\mathcal{SN}_A(\Omega)$ is a division algebra,
\item[\rm{(ii)}] if $\Omega$ is a product domain, then $\mathcal{SN}_A(\Omega)$ includes some element $f\not\equiv 0$ with $N(f)\equiv 0$. However, every element $f$ with $N(f)\not\equiv 0$ admits a multiplicative inverse in the algebra $\mathcal{SN}_A(\Omega\setminus V(N(f)))$.
\end{itemize}
\end{thm}

In \S\ref{SliceNashProp}, we study finiteness properties of slice-Nash functions enlightening their `algebraic nature'. These finiteness properties place the class of slice-Nash functions as an intermediate class that lies in between the class of slice-regular functions and the class of slice polynomial functions, in analogy with the classical real and complex case. In order to show the results of \S\ref{SliceNashProp}, we make use of some (finiteness) properties of complex Nash functions of one complex variable that we include in \S\ref{compNashSec}. Most of the results of \S\ref{compNashSec} are probably well known; we included them there with full detailed proofs as we could not find precise references in the literature. We summarise here the main results of \S\ref{SliceNashProp}:
\begin{itemize}
\item Global slice-Nash functions are slice polynomial functions (Theorem \ref{global}).
\item Slice-Nash functions defined on symmetric slice domains have only finitely many isolated zeros and isolated spherical zeros (Corollary \ref{slicedomainzero}).
\item Slice-Nash functions are polynomially bounded at infinity (Theorem \ref{slicepolbound}).
\end{itemize} 
In \S\ref{semiregularSec}, we also introduce \textit{semiregular slice-Nash functions}, as the quaternionic (or octonionic) counterpart of the complex meromorphic Nash functions. In analogy with Theorem \ref{introduDivAlg}, in Theorem \ref{semiregularAlgThm}, we obtain the following result for semiregular slice-Nash functions (compare this result with Theorem \ref{*semiregular}). 

\begin{thm}
The set $\mathcal{SEN}_A(\Omega)$ of all semiregular slice-Nash functions on $\Omega$ is an alternative $*$-subalgebra of the alternative $*$-algebra $\mathcal{SEM}_A(\Omega)$ of semiregular slice functions on $\Omega$ endowed with $+$, $\cdot$, $\cdot^c$. If $A=\HH$, then $\mathcal{SEN}_\HH(\Omega)$ is associative. Moreover,
\begin{itemize}
\item if $\Omega$ is a symmetric slice domain, then $\mathcal{SEN}_A(\Omega)$ is a division algebra,
\item if $\Omega$ is a product domain, then $\mathcal{SEN}_A(\Omega)$ includes some element $f\not\equiv 0$ with $N(f)\equiv 0$. However, every element $f$ with $N(f)\not\equiv 0$ admits a multiplicative inverse within the algebra.
\end{itemize}
\end{thm}

Also semiregular slice-Nash functions enjoy finiteness properties, for instance global semiregular slice-Nash functions are `slice rational' (Theorem \ref{globalsemiregular}).

\begin{remark}
Our definition of slice-Nash functions extends verbatim to slice regular functions defined on open circular subsets of the quadratic cone $\Qq_A$ of a finite dimensional real alternative $*$-algebra $A$. Also most of our results (with possibly some modifications to the statements) extend to this general context. A case of particular interest is the Clifford algebra $\R_3$, because, in this case explicit calculation is possible \cite{bd2}. We have decided to focus on $\HH$ and $\O$ in order to offer a clearer theory and exposition. The interested reader may try to find themselves the adequate hypothesis and statements to generalise our results to general alternative $*$-algebras. 
\end{remark}

\subsection*{Structure of the article}
We summarise here the organisation of the article. In order to keep this work as self-contained as possible, in \S\ref{preliminaries}, we collect the notions and results that we need in the remaining sections. In \S\ref{compNashSec}, we study finiteness properties for complex Nash functions (of one complex variable) and meromorphic Nash functions that we will use in \S\ref{SliceNashProp} to derive their slice counterparts. In \S\ref{buonadef}, we answer the important Question \ref{q2}. In particular, we discuss why defining slice-Nash functions as the class of slice regular functions that are ‘algebraic over the polynomials’ is not the ‘correct’ approach. In \S\ref{Maindef}, we finally introduce slice-Nash functions and characterise them in terms of their components given by the splitting lemma and in terms of their real components. Moreover, still in \S\ref{Maindef}, we show that slice-Nash functions are closed under taking slice derivatives and we establish their algebraic structure. In \S\ref{SliceNashProp}, we study finiteness properties of slice-Nash functions, we introduce semiregular slice Nash functions (in analogy with meromorphic Nash functions), and we establish finiteness properties also for this new class of functions.  

\section{Preliminaries}\label{preliminaries}

In this section, we collect some preliminary concepts and results that will be used freely along this article. We include them for the sake of completeness and to ease the reading of the article. 

\subsection{Real quaternions and octonions}

Consider the field $\R$ of real numbers and the field $\C:=\{\alpha+\beta i: \alpha,\beta\in\R\}$ of complex numbers endowed with the \textit{complex conjugation} defined as $\overline{\alpha+\beta i}:=\alpha-\beta i$ for each $\alpha,\beta\in \R$. The field $\C$ can be obtained by $\R$ by means of the so-called Cayley-Dickson construction.

$\bullet$ $\C:=\R\oplus i \R$, with product defined as $(\alpha+i\beta)(\gamma+i\delta)=(\alpha\gamma-\beta\delta)+i(\alpha\delta+\beta\gamma)$ for each $\alpha,\beta,\gamma,\delta\in \R$ and the complex conjugation $\overline{\alpha+i\beta}:=\alpha-i\beta$ for each $\alpha,\beta\in \R$ as $*$-involution.

Let $\HH$ and $\O$ denote the $*$-algebra of quaternions and octonions respectively. They can be built by iterating the Cayley-Dickson construction (see for instance \cite{ba,w,cosm}):

$\bullet$ $\HH:=\C\oplus j\C $, with product defined as  $(\alpha+j\beta)(\gamma+j\delta):=(\alpha\gamma-\delta\overline{\beta})+j(\overline{\alpha}\delta+\gamma\beta)$ for each $\alpha,\beta,\gamma,\delta\in \C$ and the $*$-involution defined as $(\alpha+j\beta )^c:=\overline{\alpha}-j\beta $ for each $\alpha,\beta\in \C$.

$\bullet$ $\O:=\HH\oplus\ell \HH$, with product defined as  $(\alpha+\ell\beta)(\gamma+\ell\delta):=(\alpha\gamma-\delta\beta^c)+\ell(\alpha^c\delta+\gamma\beta)$ for each $\alpha,\beta,\gamma,\delta\in \HH$ and the $*$-involution defined as $(\alpha+\ell \beta)^c:=\alpha^c-\ell\beta$ for each $\alpha,\beta\in \HH$.

We briefly recall some of the properties of $\HH$ and $\O$, referring the reader to \cite{ba,w,cosm} for further information and details.

The Cayley-Dickson construction endows $\HH$ and $\O$ with a structure of (real) $*$-algebra of dimension 4 and 8 respectively. By construction, each of them has a multiplicative neutral element $1$ and we identify $\R$ with the subalgebra generated by $1$. It is well known that $\HH$ is not commutative but associative. While $\O$ is not commutative nor associative but it is \textit{alternative}, i.e. the associator of three elements $(x,y,z):=(xy)z-x(yz)$ vanishes whenever two of them coincide. In particular, Artin's theorem holds for $\O$.

\begin{thm}[Artin's theorem {\cite[\S37]{ku}}]
For each $x,y\in\O$ the subalgebra generated by $x$ and $y$ is associative.
\end{thm}

Let $A$ denote either the algebra $\HH$ of quaternions or the algebra $\O$ of octonions. The $*$-involution is a real linear map $\cdot^c:A\to A,\, x\mapsto x^c$ that satisfies the following properties:
\begin{itemize}
\item[{\rm(i)}] $(x^c)^c=x$ for each $x\in A$,
\item[{\rm(ii)}] $(xy)^c=y^cx^c$ for each $x,y\in A$,
\item[{\rm(iii)}] $x^c=x$ for each $x\in \R$.
\end{itemize}
The \textit{trace} and the \textit{(squared) norm} functions, defined as
$$
t(x):=x+x^c \quad \text{and} \quad n(x):=xx^c
$$
are real-valued. In $\HH\simeq \R^4$ and $\O\simeq \R^8$, the expression $\tfrac{1}{2}t(xy^c)$ coincides with the standard scalar product $\langle x,y\rangle$ and $n(x)$ with the squared Euclidean norm $\|x\|^2$. In particular, $n(x)=0$ implies $x=0$. Moreover, $(x+y)^c=x-y$ for each $x\in\R$ and each $y$ in the Euclidean orthogonal complement of $\R$. In particular, $xx^c=x^cx$ for each $x\in A$.

Every non-zero element $x$ of $\HH$ or $\O$ has a multiplicative inverse, namely 
$$
x^{-1}=n(x)^{-1}x^c=x^cn(x)^{-1}.
$$
Moreover, for each $x,y\neq 0$ it holds $(xy)^{-1}=y^{-1}x^{-1}$. As a consequence, the algebras $\HH$ and $\O$ are \textit{division algebras with no zero divisors}. A famous result of Frobenius states that $\R$, $\C$, $\HH$ and $\O$ are the only finite dimensional real division algebras.

For $A=\HH$ or $\O$ we consider the  \textit{sphere of imaginary units} 
$$
\sph_A:=\{x\in A : t(x)=0, n(x)=1\}=\{I\in A : I^2=-1\}. 
$$ 
Observe that $\sph_A$ lies in the orthogonal complement of $\R$. As in what follows we will consider also the complexified algebra $A\otimes_\R\C$, along this paper we make the following assumption:

\begin{ass}\label{assum}
The complex field $\C$ is defined as $\C:=\{\alpha+\iota \beta : \alpha,\beta\in\R\}$ endowed with the standard multiplication and complex conjugation, where $\iota$ satisfies $\iota^2=-1$ and $\iota\not\in \sph_A$. In particular, $\C$ is not a subalgebra of $A$.
\end{ass}

The $*$-subalgebra generated by any imaginary unit $I\in\sph_A$, i.e. the complex plane $\C_I:=\langle 1,I\rangle$, is $*$-isomorphic to $\C$ through the $*$-isomorphism 
\begin{equation}\label{*iso}
\phi_I: \C\to \C_I, \quad \alpha+\iota \beta\mapsto \alpha+\beta I.
\end{equation} 
The complex plane $\C_I=\phi_I(\C)$ is called \textit{(complex) slice}. This name is justified by the following equality:
\begin{equation}\label{slicedec}
A=\bigcup_{I\in\sph_A}\C_I.
\end{equation}
Moreover, it holds $\C_I\cap \C_J=\R$ for each $I,J\in\sph_A$ such that $J\neq \pm I$. In particular, \eqref{slicedec} gives a \textit{book decomposition of the algebra $A$ into complex slices}. As a consequence, every element $x\in A\setminus \R$ can be written uniquely as $x=\alpha+\beta I$, where $\alpha,\beta\in \R$, $\beta>0$ and $I\in\sph_A$. If $x\in \R$, then $x=\alpha$, $\beta=0$ and $I$ can be chosen arbitrarily in $\sph_A$. In particular, we may define the \textit{real part} $\Re(x)$ and the \textit{imaginary part} $\Im(x)$ of $x:=\alpha+\beta I\in A$ by setting
$$
\Re(x):=\frac{x+x^c}{2} =\alpha \quad \text{and} \quad \Im(x):=x-\Re(x)=\frac{x-x^c}{2}=\beta I.
$$

On the $\R$-vector space $A$, we consider the standard Euclidean topology and analytic structure. The relative topology on each slice $\C_I$, where $I\in\sph_A$, agrees with the topology induced by the identification of $\C_I$ with $\C$. In particular, the $*$-isomorphism $\phi_I$ is a homeomorphism. 

\subsection{Slice functions}

Let $A$ be either the algebra $\HH$ of quaternions or the algebra $\O$ of octonions. We endow the complexified algebra $A\otimes_{\R}\C$ with the product defined as
$$
(x+\iota y)\cdot(x'+\iota y'):=xx'-yy'+\iota(xy'+yx')
$$
for each $x,x',y,y'\in A$. The algebra $A\otimes_{\R}\C$ has two commutating $*$-involution - the \textit{complex conjugation} $w:=x+\iota y\mapsto \overline{w}:=x-\iota y$ and the \textit{complex $*$-involution} $w:=x+\iota y\mapsto w^c:=x^c+\iota y^c$. It follows, by Assumption \ref{assum}, that the imaginary unit $\iota$ commutes with every element of $A\otimes_\R \C$. In particular, the centre of the algebra $A\otimes_{\R}\C$ is the $*$-subalgebra $\R\otimes_\R \C=\C$. The $*$-algebra $\HH\otimes_{\R}\C$ is associative, while the $*$-algebra $\O\otimes_{\R}\C$ is not associative, but alternative. For each $I\in\sph_A$ the $*$-isomorphism $\phi_I:\C\to \C_I$ introduced in \eqref{*iso} extends to 
$$
\widetilde{\phi}_I:A\otimes_\R\C\to A,\quad x+\iota y\mapsto x+yI,
$$ 
which is still a surjective $*$-algebra morphism, but no longer injective.

For a (non-empty) subset $D\subset \C$ we define the \textit{circularised} $\Omega_D$ \textit{of} $D$ as:
$$
\Omega_D:=\{x=\alpha+\beta I\in A : \alpha,\beta\in\R, \alpha+\iota \beta \in D, I\in\sph_A\}=\bigcup_{I\in \sph_A}\phi_I(D).
$$
A subset $\Omega\subset A$ is called a \textit{circular subset} if there exists $D\subset \C$ such that $\Omega=\Omega_D$. Observe that a circular subset $\Omega\subset A$ is always axially symmetric with respect to the real axes. A circular subset $\Omega\subset A$ is called a \textit{circular domain} if it is open and connected.

For each subset $D\subset \C$ we denote by $\overline{D}$ the \textit{conjugate of} $D$, that is the subset of $\C$ defined as $\overline{D}:=\{z\in \C : \overline{z}\in D\}$. We have $D=\overline{D}$ if and only if $D$ is symmetric with respect to the real axis. Observe that  $D$ is symmetric with respect to real axes if and only if $\phi_I(D)=\Omega_D\cap \C_I$ for each $I\in\sph_A$. In particular, $D\cup \overline{D}$ is symmetric with respect to the real axes and $D\cup \overline{D}=\overline{\overline{D}\cup D}$. It is clear that $\Omega_D=\Omega_{D\cup \overline{D}}$. Thus, given a circular subset $\Omega=\Omega_D\subset A$, up to substitute $D$ with $D\cup \overline{D}$, we may always assume that $D$ is symmetric with respect to the real axes. 

Let $D\subset \C$ be an open subset symmetric with respect to the real axes and $\Omega:=\Omega_D$. For each $I\in\sph_A$ denote $\Omega_I:=\Omega\cap \C_I$. If $\Omega=\Omega_D$ is a circular domain, then we only have the following two possibilities:
\begin{itemize}
\item[{\rm(i)}] For each $I\in\sph_A$ the set $\Omega_I$ is connected. In particular, $\Omega\cap \R\neq \varnothing$. In this case $\Omega$ is called a \textit{symmetric slice domain}.
\item[{\rm(ii)}] For each $I\in\sph_A$ the set $\Omega_I$ has two connected components obtained one by the other by the complex conjugation of $\C_I$. In particular, $\Omega\cap \R=\varnothing$. In this case $\Omega$ is called a \textit{product domain}. 
\end{itemize}
Let $I\in\sph_A$ be an imaginary unit and $\phi_I:\C\to \C_I$ the $*$-isomorphism introduced in \eqref{*iso}. As $\Omega_I=\phi_I(D)$, then in case (i), $D$ is connected and $D\cap \R\neq \varnothing$. While in case (ii), $D$ has two connected components obtained one by the other by the complex conjugation of $\C$, so in particular $D\cap \R=\varnothing$. 

We start by recalling the definition of stem functions.

\begin{defn}[Stem function]
Let $D\subset \C$ be a (non-empty) open subset. A function $F:D\to A\otimes_\R\C$ is called a \textit{stem function on} $D$ if it is \textit{complex intrinsic}, i.e. it satisfies 
\begin{equation}\label{compint}
\overline{F(z)}=F(\overline{z})
\end{equation}
for each $z\in D$ such that $\overline{z}\in D$. \hfill$\sqbullet$
\end{defn}

If $D$ is symmetric with respect to the real axes, then \eqref{compint} holds for each $z\in D$. In what follows, even if not explicitly mentioned, we make the following assumption:

\begin{ass}\label{domainD}
$D$ is non-empty and symmetric with respect to the real axes, that is $D=\overline{D}$.
\end{ass}

Observe that $F=F_1+\iota F_2:D\to A\otimes_\R\C$ is a stem function if and only if the $A$-valued components $F_1$ and $F_2$ form an \textit{odd-even pair} with respect to the imaginary part of $z$, that is, if and only if $F_1(\overline{z})=F_1(z)$ and $F_2(\overline{z})=-F_2(z)$ for each $z\in D$.

We define the \textit{conjugated of} $F$, that we denote by $F^c$, as the function defined as $F^c(z):=(F(z))^c$ for each $z\in D$. It is worthwhile to remark that if $F=F_1+\iota F_2$, then $F^c=F_1^c+\iota F_2^c$. Thus, $F$ is a stem function if and only if $F^c$ is a stem function. 

The $*$-algebra structure on $A\otimes_\R\C$ induces a $*$-algebra structure on the set of all stem functions on $D$. Namely, under Assumption \ref{domainD}, we have the following:

\begin{prop}\label{stemalg*}
The set of all stem functions $D\to A\otimes_{\R}\C$ form an alternative $*$-algebra over $\R$ with pointwise addition $(F+G)(z):=F(z)+G(z)$, multiplication $(FG)(z):=F(z)G(z)$ and conjugation $F^c(z):=(F(z))^c$. Moreover, if $A=\HH$, then the product is associative. 
\end{prop}

Next, we recall the definition of slice functions.

\begin{defn}[Slice function]
Let $D\subset \C$ be a (non-empty) open subset and $\Omega_D$ the circularised of $D$. A function $f:\Omega_D\to A$ is a (\textit{left}) \textit{slice function} if there exists a stem function $F:D\to A\otimes_\R\C$ such that the diagram
$$
\xymatrix{
D\ar[d]_{\phi_I} \ar[r]^{F}&A\otimes_\R\C \ar[d]^{\widetilde{\phi}_I}\\
\Omega_D\ar[r]^{f}& A
}
$$
commutes for each $I\in\sph_A$. In this case, we say that the slice function $f$ is \textit{induced by the stem function} $F$ and we write $f=\mathcal{I}(F)$. We denote by $\mathcal{S}_A(\Omega_D)$ the set of all slice functions on $\Omega_D$. \hfill$\sqbullet$
\end{defn}

If $F=F_1+\iota F_2:D\to A\otimes_\R\C$ is a stem function such that $f=\mathcal{I}(F)$, then 
$$
f(x)=F_1(z)+IF_2(z)
$$
for each $I\in\sph_A$ and each $x:=\alpha+\beta I\in \Omega_I=\Omega_D\cap \C_I$, where $z:=\phi_I^{-1}(x)=\alpha+\iota\beta\in D$. Observe that as $F_1$ and $F_2$ form an odd-even pair, then the slice function $f$ is well defined. In fact,
$$
f(\alpha+(-\beta)(-I))=F_1(\overline{z})+(-I)F_2(\overline{z})=F_1(z)+IF_2(z)=f(\alpha+\beta I).
$$

The following \textit{representation formula} shows that slice functions are completely determined by their values on a single (complex) slice.

\begin{prop}[Representation formula {\cite[Prop.6]{gp}}]
Let $\Omega\subset A$ be an open circular set and $f:\Omega\to A$ a slice function. Then for each $I\in \sph_A$ we have
\begin{equation}\label{repform}
f(x)=\frac{1}{2}(f(\alpha+\beta I)+f(\alpha-\beta I))-\frac{J}{2}(I(f(\alpha+\beta I)-f(\alpha-\beta I)))
\end{equation}
for each $J\in \sph_A$ and each $x:=\alpha+\beta J\in\Omega_J:=\Omega\cap \C_J$.
\end{prop}

Let $\Omega\subset A$ be an open circular set and $D\subset \C$ an open subset such that $\Omega=\Omega_D$. Under Assumption \ref{domainD}, the open set $D$ is the unique maximal open subset of $\C$ (with respect to the inclusion) such that $\Omega=\Omega_D$. In particular, as 
$$
F_1(z)=\frac{1}{2}(f(x)+f(x^c))\quad \text{and} \quad F_2(z)=-\frac{I}{2}(f(x)-f(x^c))
$$
for each $x\in \Omega_I:=\Omega\cap \C_I$ and $z=\phi^{-1}_I(x)\in D$, then under Assumption \ref{domainD}, there exists a unique stem function $F:D\to A\otimes_\R\C$ such that $f=\mathcal{I}(F)$.

\subsection{Slice regular functions} Let $A$ be either the algebra $\HH$ of quaternions or the algebra $\O$ of octonions. Let $D\subset \C$ be an open subset and $F=F_1+\iota F_2:D\to A\otimes_\R\C$  a stem function. We denote by $\Cont^1(D)$ the set of all stem functions of class $\Cont^1$ on $D$ (with respect to the differentiable structure induced by the Euclidean one). Assume that $F\in\Cont^1(D)$. Then, the partial derivatives defined as 
$$
\frac{\partial F}{\partial\z}(z):=\frac{1}{2}\Big(\frac{\partial F}{\partial\bm{\alpha}}(z)-\iota \frac{\partial F}{\partial\bm{\beta}}(z)\Big) \quad \text{and} \quad \frac{\partial F}{\partial\overline{\z}}(z):=\frac{1}{2}\Big(\frac{\partial F}{\partial\bm{\alpha}}(z)+\iota \frac{\partial F}{\partial\bm{\beta}}(z)\Big)
$$
for each $z=\alpha+\iota \beta\in D$ are continuous stem functions on $D$. Let $\Omega:=\Omega_D$ and let $f:=\mathcal{I}(F)$	. The \textit{slice derivative of $f$} is defined as
$$
\frac{\partial f}{\partial\x}:=\mathcal{I}\Big(\frac{\partial F}{\partial\z}\Big) 
$$
In particular, as $F$ is $\Cont^1$, then the slice derivative of $f$ is a continuous slice function on $\Omega$ \cite[Prop.7(i)]{gp}. 

The left multiplication by $\iota$ defines a complex structure on $A\otimes_\R\C$. With respect to this structure, a stem function $F=F_1+\iota F_2:D\to A\otimes_\R\C$ is a \textit{holomorphic stem function on $D$} if it is of class $\Cont^1(D)$ and satisfies the Cauchy-Riemann equation:
$$
\frac{\partial F}{\partial\overline{\z}}=\frac{1}{2}\Big(\frac{\partial F}{\partial\bm{\alpha}}+\iota \frac{\partial F}{\partial\bm{\beta}}\Big)=0.
$$
The Cauchy-Riemann equation is equivalent to the equations
$$
\frac{\partial F_1}{\partial\bm{\alpha}}= \frac{\partial F_2}{\partial\bm {\beta}} \quad \text{and} \quad \frac{\partial F_1}{\partial\bm{\beta}}=- \frac{\partial F_2}{\partial\bm{\alpha}}.
$$
Slice regular functions are defined as those slice functions induced by holomorphic stem functions.

\begin{defn}[Slice regular function]
Let $D\subset \C$ be a (non-empty) open subset and $\Omega:=\Omega_D$. Let $F:D\to A\otimes_\R\C$ be a stem function. We say that the slice function $f=\mathcal{I}(F):\Omega\to A$ is \textit{slice regular on} $\Omega$ if the stem function $F$ is holomorphic on $D$. We denote by $\mathcal{SR}_A(\Omega)$ the set of all slice regular functions on $\Omega$. \hfill$\sqbullet$
\end{defn}

Polynomials and power series are examples of slice regular functions. In fact, it holds:

\begin{prop}[{\cite[Thm.2.1]{gs2}\cite[Thm.2.1]{gs3}}]Let $A$ be either the algebra of quaternions $\HH$ or the algebra of octonions $\O$. 
\begin{itemize}
\item Every polynomial of the form $\sum_{s=0}^nx^s\alpha_s$ with coefficients $\alpha_0,\ldots,\alpha_n\in A$ defines a slice regular function on $A$. 
\item Every power series $\sum_{n\geq 0} x^n\alpha_n$ with coefficients $\{\alpha_n\}_{n\geq 0}\subset A$ converges in an open ball $B(0,R):=\{x\in A : \|x\|<R\}$. If $R>0$, then the sum of the series defines a slice regular function on $B(0,R)$.
\end{itemize}
\end{prop}

Let $F$ be a stem function of class $\Cont^1$ on $D$ and $\Omega:=\Omega_D$. Then, the slice function $f:=\mathcal{I}(F)$ is slice regular on $\Omega$ if and only if
$$
\frac{\partial f}{\partial\x^c}:=\mathcal{I}\Big(\frac{\partial F}{\partial\overline{\z}}\Big)
$$
vanishes identically on $\Omega$. Moreover, if $f$ is slice regular on $\Omega$, then also its slice derivative $\tfrac{\partial f}{\partial\x}$ is slice regular on $\Omega$. 

It is worthwhile to remark that if $\Omega\subset A$ is a symmetric slice domain, then a $f$ is a slice regular function if and only if it is \textit{Cullen regular} in the sense introduced by Gentili and Struppa in \cite{gs,gs2} for the quaternionic case and in \cite{gssv,gs3} for the octonionic case.

\subsection{Splitting lemma} Let $A$ be either the algebra of quaternions $\HH$ or the algebra of octonions $\O$. Define the integers  
\begin{equation}\label{dim}
d_A:=\dim_\R(A)-1=
\begin{cases}
3, \text{ if } A=\HH,\\
7, \text{ if } A=\O,
\end{cases}
\quad
u_A:=\dim_\C(A)-1=
\begin{cases}
1, \text{ if } A=\HH,\\
3, \text{ if } A=\O.
\end{cases}
\end{equation}
In particular, $d_A=2u_A+1$. The left multiplication by an element of $\sph_A$ induces a complex structure on $A$. More precisely, for a fixed imaginary unit $I\in\sph_A$, the sum of $A$ together with the complex scalar multiplication $\C_I\times A\to A$ sending the pair $(\gamma,a)\in \C_I\times A$ into the product $\gamma a$ in $A$ defines a structure of $\C_I$-vector space on $A$ (in the case $A=\O$, it follows by Artin's theorem). If $\{1,I_1,\ldots,I_{u_A}\}$ is a base of $A$ as a $\C_I$-vector space, then $\Bb_I:=\{1,I, I_1,II_1,\ldots,I_{u_A},II_{u_A}\}$ is a base of $A$ as a $\R$-vector space \cite[Lem.2.3]{gp2}. A real vector base of $A$ of this form is called a \textit{splitting base of $A$ associated to $I$}. 

\begin{example}\label{splittingex}
(i) A splitting base of $\HH$ associated to $I\in\sph_\HH$ is any real basis of the form $\{1,I,J,IJ\}$ where $J\in\sph_{\HH}$ is an imaginary unit such that $J\neq \pm I$.

(ii) There exist imaginary units $I,J,L\in\sph_\O$ such that $\{1,I,J,IJ,L,IL,JL,I(JL)\}$ is a splitting base of $\O$ associated to $I$. For instance, one can take $I=i, J=j,$ and $L=\ell$, where $i,j,\ell$ are the imaginary unit obtained by iterating the Cayley-Dickson construction. \hfill$\sqbullet$
\end{example}

Let $\Omega\subset A$ be an open circular set and $f:\Omega\to A$ a slice function. Let $I\in\sph_A$ be an imaginary unit and $\{I_0:=1,I, I_1,II_1,\ldots,I_{u_A},II_{u_A}\}$ a splitting base of $A$ associated to $I$. Let $\Omega_I:=\Omega\cap \C_I$ and $f_I:=f|_{\Omega_I}:\Omega_I\to A$. Then, there exist unique real valued functions $f_{s,k}^{I,I_k}:\Omega_I\to \R$ for $s=1,2$ and $k=0,\ldots,u_A$ such that
$$
f_I(z_I)=\sum_{k=0}^{u_A}(f_{1,k}^{I,I_k}(z_I)+f_{2,k}^{I,I_k}(z_I)I)I_k
$$
for each $z_I\in \Omega_I$. For each $k=0,\ldots,u_A$ define the (complex) function 
$$
f_k^{I,I_k}:=f_{1,k}^{I,I_k}+f_{2,k}^{I,I_k}I:\Omega_I\to \C_I.
$$

The relation between slice regularity and complex holomorphy is made explicit in
the following result, called the \textit{splitting lemma}. 

\begin{lem}[Splitting lemma {\cite[Lem.2.4]{gp2}}]
The following are equivalent:
\begin{itemize}
\item[{\rm (i)}] $f$ is slice regular.
\item[{\rm (ii)}] There exist $I\in \sph_A$ and a splitting base $\{I_0:=1,I, I_1,II_1,\ldots,I_{u_A},II_{u_A}\}$ of $A$ associated to $I$ such that the functions $f_k^{I,I_k}:\Omega_I\to \C_I$ are holomorphic for each $k=0,\ldots,u_A$.
\item[{\rm (iii)}] For each imaginary unit $I\in \sph_A$ and each splitting base $\{I_0:=1,I, I_1,II_1,\ldots,I_{u_A},II_{u_A}\}$ of $A$ associated to $I$, the functions $f_k^{I,I_k}:\Omega_I\to \C_I$ are holomorphic functions for each $k=0,\ldots,u_A$.
\end{itemize}
\end{lem}

In order to lighten the exposition, we will often use the notation from the following example:

\begin{example}
(i) Let $\Omega\subset \HH$ be an open circular set and $f:\Omega\to \HH$ a slice regular function. Let $I,J\in\sph_{\HH}$ be imaginary units such that $J\neq \pm I$. Then $\Bb_I:=\{1,I,J,IJ\}$ is a splitting base of $\HH$ associated to $I$. By the splitting lemma there exist unique holomorphic functions $f_1^{I,J},f_2^{I,J}:\Omega_I\to \C_I$ such that
$$
f_I(z_I)=f_1^{I,J}(z_I)+f^{I,J}_2(z_I)J
$$
for each $z_I\in\Omega_I$. When the situation is clear by the context, we will often simply write $f_1,f_2$ instead of $f_1^{I,J},f_2^{I,J}$.

(ii) Let $\Omega\subset \O$ be an open circular set and $f:\Omega\to \O$ a slice regular function. Let $I,J,L\in\sph_{\O}$ be imaginary units such that $\{1,I,J,IJ,L,IL,JL,I(JL)\}$ is a splitting base of $\O$ associated to $I$. By the splitting lemma, there exist unique holomorphic functions $f_s^{I,J,L}:\Omega_I\to \C_I$ for $s=0,1,2,3$ such that
$$
f_I(z_I)=f_0^{I,J,L}(z_I)+f^{I,J,L}_1(z_I)J+f_2^{I,J,L}(z_I)L+f_3^{I,J,L}(z_I)(JL)
$$
for each $z_I\in\Omega_I$. When the situation is clear by the context, we will often simply write $f_0,f_1,f_2,f_3$ instead of $f_0^{I,J,L},f_1^{I,J,L},f_2^{I,J,L},f_3^{I,J,L}$ \hfill$\sqbullet$
\end{example}

\subsection{Slice product} Let $A$ be either the algebra of quaternions $\HH$ or the algebra of octonions $\O$. In general, the pointwise product $x\mapsto f(x)g(x)$ of two slice functions $f,g:\Omega\to A$ is not a slice function. For instance, if $f(x)=xI$, then the function $f(x)f(x)=(xI)(xI)$ is not a slice function. However, the product of stem functions induces a natural product on slice functions. Let $D\subset \C$ be an open subset and $\Omega:=\Omega_D$.

\begin{defn}[Slice product]
Let $F,G:D\to A\otimes_\R\C$ be stem functions and $f=\mathcal{I}(F)$ and $g=\mathcal{I}(G)$ the slice functions induced by $F$ and $G$ respectively. The \textit{slice product of $f$ and $g$} is the slice function defined as $f\cdot g:=\mathcal{I}(F\cdot G)\in\mathcal{S}_A(\Omega).$ \hfill$\sqbullet$
\end{defn}

By \cite[Prop.11]{gp}, we have that, if $f$ and $g$ are slice regular on $\Omega$, then $f\cdot g$ is slice regular on $\Omega$. Observe that in general $(f\cdot g)(x)\neq f(x)g(x)$. Let $F=F_1+\iota F_2$ and $G=G_1+\iota G_2$ be stem functions such that $f=\mathcal{I}(F)$ and $g=\mathcal{I}(G)$. Let $I\in\sph_A$, $z_I=\alpha+\beta I\in \Omega_I:=\Omega\cap \C_I$ and $z=\phi_I^{-1}(z_I)=\alpha+\iota\beta\in D$. Then,
\begin{equation}\label{intrsliceprod}
(f\cdot g)(z_I)=F_1(z)G_1(z)-F_2(z)G_2(z)+I(F_1(z)G_2(z)+F_2(z)G_1(z)),
\end{equation}
while
\begin{equation}\label{intrprod}
f(z_I)g(z_I)=F_1(z)G_1(z)+(IF_2(z))(IG_2(z))+F_1(z)(IG_2(z))+(IF_2(z))G_1(z).
\end{equation}

It is worthwhile to remark that if $\Omega\subset \HH$ is a symmetric slice domain, then the slice product on $\mathcal{SR}_{\HH}(\Omega)$ coincides with the $*$-product on $\mathcal{SR}_{\HH}(\Omega)$ as defined in \cite[Def.2.2]{gst2}.

Let $F:D\to A\otimes_\R\C$ be a stem function. We define the \textit{conjugate of} $f:=\mathcal{I}(F)$ as $f^c:=\mathcal{I}(F^c)$. It is clear that $f$ is slice regular if and only if $f^c$ is slice regular. Let $D\subset \C$ be an open subset (that satisfies Assumption \ref{domainD}) and $\Omega:=\Omega_D$. By Proposition \ref{stemalg*}, we deduce the following:

\begin{prop}\label{algslice1}
The set $\mathcal{S}_A(\Omega)$ of all slice functions on $\Omega$ form an alternative $*$-algebra over $\R$ (which is associative if $A=\HH$) with the pointwise addition $(f+g)(x):=f(x)+g(x)$, the slice product $x\mapsto (f\cdot g)(x)$ and the conjugation $f\mapsto f^c$. The mapping $\mathcal{I}$ is a $*$-algebra isomorphism from the $*$-algebra of stem functions on $D$ onto $\mathcal{S}_A(\Omega)$.
\end{prop}

Observe that the restriction of $\mathcal{I}$ to the $*$-algebra of holomorphic stem functions on $D$ is a $*$-algebra isomorphism onto $\mathcal{SR}_A(\Omega)$. We end this section with the following remark:

\begin{remark}\label{sliceprodoct}
As constant functions are a particular kind of slice functions, we may ask whether the product and the slice product of a slice function $f:\Omega\to A$ and a constant function coincide or not. If $A=\HH$, then for any $q_0\in \HH$ we have
$
f(q)q_0=(f\cdot q_0)(q)
$
for each $q\in \HH$. 

While if $A=\O$, then the product and the slice product with a constant function may be different. Let $I,L\in \sph_{\O}$ be imaginary units such that $IL=-LI$. Consider the slice polynomial function $f(x)=xI$. The slice function
$
(f\cdot L)(x)=(xI)\cdot L=x(IL)
$
is different from the function $f(x)L=(xI)L$. Observe that as $(f\cdot L)(x)=f(x)L$ on the slice $\C_I$ (by Artin's theorem), we deduce, by the representation formula, that the function $f(x)L$ is not a slice function. \hfill $\sqbullet$
\end{remark}

\subsection{Slice polynomials} \label{prelipol}

Let $A$ be either the algebra of quaternions $\HH$ or the algebra of octonions $\O$. We denote by $A[\x]$ the set of \textit{slice polynomials}, that is, the set of all polynomials of the form $P(\x)=\x^n\alpha_n+\ldots+\x\alpha_1+\alpha_0$, where $\alpha_0,\ldots,\alpha_n\in A$. To each slice polynomial $P\in A[\x]$ we can associate a polynomial function $A\to A, x\mapsto P(x)$, which is slice regular, that we will still denote by $P$ (if $A=\O$, it follows by Artin's theorem). With the same operations, the slice polynomial functions form a $*$-subalgebra of $\mathcal{SR}_A(A)$. Observe that the slice product of two ordered polynomials coincide with the usual product of polynomials, where $\x$ is considered as a commutating variable (as defined, for instance, in \cite{lam}). 

Analogously, we define the set of \textit{slice polynomials} $A[\x_1,\x_2]$ in two variables as the set of formal finite sums of ordered monomials of the form $\x_1^{s}(\x_2^{r}\alpha)$, where $\alpha\in A$. To each ordered polynomial $P\in A[\x_1,\x_2]$ we can associate a polynomial function $A^2\to A, (x_1,x_2)\mapsto P(x_1,x_2)$, which is slice regular in the sense introduced in \cite{gp6} for functions of several variables (see \cite[Prop.3.14]{gp6}). We endow $A[\x_1,\x_2]$ with the usual sum of polynomial and the product of polynomials introduced in  \cite{gp6}, which is defined on the monomials as $(\x_1^{s_1}(\x_2^{r_1}\alpha_1))\cdot(\x_1^{s_2}(\x_2^{r_2}\alpha_2)):=\x_1^{s_1+s_2}(\x_2^{r_1+r_2}(\alpha_1\alpha_2))$.

\subsection{Normal function} 

Let $A$ be either the algebra of quaternions $\HH$ or the algebra of octonions $\O$. Led $D\subset \C$ be an open subset and $\Omega:=\Omega_D$. Let $F=F_1+\iota F_2:D\to A\otimes_\R\C$ be a stem function and $f=\mathcal{I}(F)$. The slice function $f$ is called \textit{slice preserving} if $f(\Omega\cap \C_I)\subset \C_I$ for each $I\in\sph_A$. Equivalently by \cite[Prop.10]{gp}, $f$ is slice preserving if and only if the $A$-components $F_1$ and $F_2$ of the stem function $F$ are real valued functions. Observe that if $f$ is slice preserving, by \eqref{intrsliceprod} and \eqref{intrprod}, we deduce 
$
(f\cdot g)(x)=f(x)g(x)
$
for each $x\in \Omega$. 

Recall that $F^c:=F_1^c+\iota F_2^c$. The function
$
F\cdot F^c=n(F_1)-n(F_2)+\iota t(F_1F_2^c)
$
is a stem function on $D$. Hence we may introduce the following definition:

\begin{defn}[Normal function]
Let $D\subset \C$ be an open subset and $\Omega:=\Omega_D$ the circularised of $D$. Let $F:D\to A\otimes_\R\C$ be a stem function and $f=\mathcal{I}(F):\Omega\to A$. The \textit{normal function of $f$} is the slice function $
N(f):=f\cdot f^c=\mathcal{I}(F\cdot F^c)\in\mathcal{S}_A(\Omega)$. \hfill$\sqbullet$
\end{defn}

Normal functions satisfy the following properties:
\begin{enumerate}
\item $N(f)=N(f)^c$,
\item $N(f)=f\cdot f^c=f^c\cdot f=N(f^c)$,
\item $N(f)$ is slice preserving,
\item If $f\in\mathcal{SR}_A(\Omega)$, then $N(f)\in\mathcal{SR}_A(\Omega)$.
\item If $f,g\in\mathcal{S}_A(\Omega)$, then $N(f\cdot g)=N(f)N(g)=N(g)N(f)=N(g\cdot f)$.
\end{enumerate}

In general, it is not guaranteed that if $f\in\mathcal{SR}_A(\Omega)$ is not identically zero, then $N(f)$ is not identically zero. We point out this phenomenon for $\O$ in the following example, but it can be easily adapted to $\HH$.

\begin{example}[{\cite[Rmk.12]{gp}, \cite[Ex.2.10]{gps2}}]\label{fettazero}
Let $J\in\sph_\O$. Consider the function $f:\O\setminus \R\to \O$ defined as
$$
f(x):=1+\frac{\Im(x)}{\|\Im(x)\|}J
$$
for each $x\in\O\setminus\R$. The zero set of $f$ is $\C_J^+:=\{\alpha+\beta J : \alpha,\beta\in \R, \beta>0\}$. The function $f$ is slice regular, as it is induced by the holomorphic stem function
$$
F:\C\setminus\R\to \O\otimes_\R\C, \quad z\mapsto \begin{cases}1+\iota J:=1\otimes 1+J\otimes \iota, \, \text{ if } \Im(z)>0,\\
1-\iota J:=1\otimes 1-J\otimes \iota, \, \text{ if } \Im(z)<0.\end{cases}
$$
As 
$$
F^c:\C\setminus\R\to \O\otimes_\R\C, \quad z\mapsto \begin{cases}1-\iota J, \, \text{ if } \Im(z)>0,\\
1+\iota J, \, \text{ if } \Im(z)<0.\end{cases}
$$
and
$$
(1+\iota J)(1-\iota J)=(1-\iota J)(1+\iota J)=1+J^2+\iota(-J+J)=0,
$$
we have $N(f)=\mathcal{I}(F\cdot F^c)\equiv 0$. \hfill$\sqbullet$
\end{example}

Nevertheless, it holds the following characterisation:

\begin{prop}[{\cite[Prop.3.7]{gps}}]\label{nonzeroint}
Let $\Omega\subset A$ be an open circular set and $f\in\mathcal{SR}_A(\Omega)$. The equality $N(f)\equiv 0$ implies $f\equiv 0$ if and only if $\Omega$ is a union of symmetric slice domains.
\end{prop}

\subsection{Zeros of slice functions}\label{sectionzero} Zero sets of slice functions have been extensively studied in several contexts - for quaternionic slice regular functions \cite{gst, gs2, gst2}, for octonionic power series \cite{gs3,gp3} and more generally for slice functions defined on finite dimensional real alternative $*$-algebras \cite{gps3,gps}.  

Let $A$ be either the algebra of quaternions $\HH$ or the algebra of octonions $\O$. Let $\Omega\subset A$ be a circular domain and $f\in\mathcal{S}_A(\Omega)$. The \textit{zero set of} $f$ is the set 
$
V(f):=\{x\in \Omega : f(x)=0\}.
$
Let $I\in\sph_A$ be an imaginary unit and $x=\alpha+\beta I\in A$. Define
$
\sph_x:=\alpha+\beta\sph_A=\{\alpha+\beta J : J\in \sph_A\}.
$
The set $\sph_A$ is the circularised of the singleton $\{\alpha+\iota\beta\}\subset \C$. Moreover, 
\begin{itemize}
\item if $x\in \R$, then $\sph_x=\{x\}$,
\item if $x\in A\setminus \R$, then the set $\sph_x$ is a sphere of dimension $d_A-1$ obtained by the sphere $\sph_A$ by a translation along the real axes and a dilatation.
\end{itemize}

We recall the following theorems on the zero sets of slice functions. For each $I\in\sph_A$ denote $\C_I^+:=\{\alpha+\beta I\in\C_I : \beta >0\}$ and $\Omega_I^+:=\Omega\cap \C_I^+$.

\begin{thm}[{\cite[Thm.3.2]{gps}}]\label{zeri1}
If $f\in\mathcal{S}_A(\Omega)$, then for each $x\in \Omega$ the sets $V(f)\cap \sph_x$ and $V(f^c)\cap \sph_x$ are both empty, both singletons or both equal to $\sph_x$. Moreover, 
$$
V(N(f))=\bigcup_{x\in V(f)}\sph_x=\bigcup_{x\in V(f^c)}\sph_x.
$$
\end{thm}

\begin{thm}[{\cite[Thm.3.5]{gps}}]\label{zeri2}
Let $\Omega\subset A$ be a circular domain and $f:\Omega \to A$ a non-zero slice regular function. Then
\begin{itemize}
\item[{\rm(i)}] The intersection $V(f)\cap \C_I^+$ is closed and discrete in $\Omega_I$ for each $I\in\sph_A$ with at most one exception $I_0\in\sph_A$, for which it holds $f|_{\Omega_{I_0}^+}\equiv 0$.
\item[{\rm(ii)}]  If, moreover, $N(f)\not\equiv 0$, then $V(f)$ is a union of isolated points and isolated spheres of the form $\sph_x$ for $x\in A\setminus \R$.
\end{itemize}
\end{thm}

Case (i) happens for instance for the function $f$ of Example \ref{fettazero}. By Proposition \ref{nonzeroint}, we deduce the following:

\begin{cor}
If $\Omega\subset A$ is a symmetric slice domain and $f:\Omega \to A$ a non-zero slice regular function, then $V(f)$ is a union of isolated points and isolated spheres of the form $\sph_x$ for $x\in A\setminus \R$.
\end{cor}

An isolated point of $V(f)$ is called an \textit{isolated zero of} $f$, while an isolated sphere $\sph_x\subset V(f)$, where $x\in A\setminus \R$, is called an \textit{isolated spherical zero of} $f$.

\subsection{Reciprocals of slice functions} Let $A$ be either the algebra of quaternions $\HH$ or the algebra of octonions $\O$. Let $\Omega\subset A$ be an open circular set. In particular, $\Omega$ is a disjoint union of symmetric slice domains and product domains. Let $f:\Omega\to A$ be a non-zero slice regular function and $\Omega'\subset \Omega$ a connected component of $\Omega$. By Proposition \ref{nonzeroint}, if $\Omega'$ is a symmetric slice domain, then $N(f)\not\equiv 0$, while it might happen that $N(f)\equiv 0$ when $\Omega'$ is a product domain. In particular, $\Omega\setminus V(N(f))$ may be contained only in some of the connected components of $\Omega$ or even be empty (see Example \ref{fettazero}). When $\Omega\setminus V(N(f))$ is non-empty, slice functions have a multiplicative inverse on $\Omega\setminus V(N(f))$, which is denoted by $f^{-\bullet}$.

\begin{prop}[{\cite[Prop.4.1]{gps}}]
If $\Omega\setminus V(N(f))$ is non-empty, then $f$ admits a multiplicative inverse $f^{-\bullet}$ in the algebra $\mathcal{S}_A(\Omega\setminus V(N(f)))$. Namely, the slice function defined as
$$
f^{-\bullet}(x):=(N(f)(x))^{-1}f^c(x)
$$
for each $x\in\Omega\setminus V(N(f))$ satisfies
$
f\cdot f^{-\bullet}=f^{-\bullet}\cdot f=1
$
on $\Omega\setminus V(N(f))$. Moreover, $f^{-\bullet}$ is slice regular on $\Omega\setminus V(N(f))$ if and only if $f$ is slice regular on $\Omega\setminus V(N(f))$.
\end{prop}

By Propositions \ref{algslice1}, \ref{nonzeroint} and the previous result, we deduce straightforwardly the following:

\begin{thm}\label{sliceregularstructure}
Let $A$ be either the algebra of quaternions $\HH$ or the algebra of octonions $\O$. Let $\Omega\subset A$ be an open circular set. The set $\mathcal{SR}_A(\Omega)$ of all slice regular functions on $\Omega$ is an alternative $*$-algebra if endowed with $+$, $\cdot$, $\cdot^c$. If $A=\HH$, then $\mathcal{SR}_\HH(\Omega)$ is associative. Moreover,
\begin{itemize}
\item if $\Omega$ is a symmetric slice domain, then $\mathcal{SR}_A(\Omega)$ is a division algebra,
\item if $\Omega$ is a product domain, then $\mathcal{SR}_A(\Omega)$ includes some element $f\not\equiv 0$ with $N(f)\equiv 0$. However, every element $f$ with $N(f)\not\equiv 0$ admits a multiplicative inverse in the algebra $\mathcal{SR}_A(\Omega\setminus V(N(f)))$.
\end{itemize}
\end{thm}

\subsection{Singularities of slice regular functions} Singularities of slice regular functions over the quaternions have been studied in \cite{gst2,s}. In \cite{gps1}, the authors study singularities of slice regular functions over general alternative $*$-algebras finding Laurent-type expansions and the related classification of singularities as removable, poles or essential. In \cite{gps}, the authors specialise the results of \cite{gps1} to the case of quaternions and octonions. In this section, we recall the results that we need in this article. 

We start by recalling the classification of the \textit{isolated singularities}. Let $A$ be the algebra of quaternions $\HH$ or the algebra of octonions $\O$. Let $\Omega\subset A$ be an open circular set.

\begin{thm}[Characterisation of isolated singularities {\cite[Thm.9.5]{gps1}}]\label{sing1}
Let $y\in \Omega\cap \R$ and let $f:\Omega\setminus\{y\}\to A$ be a slice regular function. Then one (and only one) of the following properties holds:
\begin{itemize}
\item[{\rm(i)}] {\rm(Removable singularity)} There exists a slice regular function $\Phi:\Omega\to A$ such that $f=\Phi|_{\Omega\setminus\{y\}}$.
\item[{\rm(ii)}] {\rm(Pole)} There exists an integer $k\geq 1$ and a slice regular function $\Phi:\Omega\to A$ such that 
$
\Phi(x)=(x-y)^kf(x)
$
for each $x\in \Omega\setminus\{y\}$.
\item[{\rm(iii)}] {\rm(Essential singularity)} For all open neighbourhoods $U$ of $y$ in $\Omega$ and all integers $k\geq 0$ it holds
$
\sup_{x\in U\setminus\{y\}}\|(x-y)^kf(x)\|=+\infty.
$
\end{itemize}
\end{thm}

Let $y\in A$. The function $f(x):=x-y$ is slice regular on $A$. The normal function 
$$
N(f)(x)=(x-y)\cdot(x-y)^c=(x-y)\cdot(x-y^c)=x^2-x(y+y^c)+yy^c
$$
coincides with the slice preserving quadratic polynomial function $\Delta_y(x):=x^2-xt(y)+n(y)$, whose zero set is $\sph_y$. In particular, if $y'\in A$, then $\Delta_y=\Delta_{y'}$ if and only if $\sph_y=\sph_{y'}$. 

Next, we recall the classification of \textit{spherical singularities}. Given that in this article we are not interested in the order of the singularity, we ignore questions related to spherical singularities with points with different multiplicities, and we state the result in the following form:

\begin{thm}[Characterisation of spherical singularities {\cite[Thm.9.4]{gps1}}]\label{sing2}
Let $y\in \Omega\setminus \R$ and let $f:\Omega\setminus\sph_y\to A$ be a slice regular function. Then, one (and only one) of the following properties holds:
\begin{itemize}
\item[{\rm(i)}] {\rm(Spherical removable singularity)} There exists a slice regular function $\Phi:\Omega\to A$ such that $f=\Phi|_{\Omega\setminus\sph_y}$. 
\item[{\rm(ii)}] {\rm(Spherical pole)} There exists an integer $k\geq 1$ and a slice regular function $\Phi:\Omega\to A$ such that $\Phi(x)=\Delta_y(x)^kf(x)$ for each $x\in \Omega\setminus\sph_y$.
\item[{\rm(iii)}] {\rm(Spherical essential singularity)} For all open neighbourhoods $U$ of $\sph_y$ in $\Omega$ and all integers $k\geq 0$ it holds $\sup_{x\in U\setminus\sph_y}\|\Delta_y(x)^kf(x)\|=+\infty. 
$
\end{itemize}
\end{thm}

Let $y\in\Omega$. Clearly, $\sph_y=\{y\}$ if $y\in \R$. For each $I\in\sph_A$ let $\{y_I,y_I^c\}:=\sph_y\cap \C_I$ (if $y\in \R$ it holds $y_I=y_I^c=y$ for each $I\in\sph_A$). Let $u_A$ be the integer defined in \eqref{dim}. Given an imaginary unit $I\in \sph_A$ let $\Bb_I:=\{1,I,I_1,II_1,\ldots,I_{u_A}\}$ be a splitting base of $A$ associated to $I$. Let $f:\Omega\setminus\sph_y\to A$ be a slice regular function. By the splitting lemma, there exist unique holomorphic functions $f^{\Bb_I}_0,\ldots,f^{\Bb_I}_{u_A}:\Omega_I\setminus\{y_I,y_I^c\}\to \C_I$ such that
$$
f_I=f^{\Bb_I}_0+f^{\Bb_I}_1I_1+\ldots f^{\Bb_I}_{u_A}I_{u_A}.
$$
By the results of \cite[\S9]{gps1}, we derive the following remark that characterise the singularities of slice regular functions in terms of the components given by the splitting lemma.

\begin{remark}\label{singchar}
The isolated or spherical singularity $\sph_y$ is:

(i) A removable singularity if and only if, for each imaginary unit $I\in\sph_A$ and each splitting base $\Bb_I$ of $A$ associated to $I$, both $y_I$ and $y_I^c$ are removable singularities for the functions $f^{\Bb_I}_0,\ldots,f^{\Bb_I}_{u_A}$.

(ii) A pole if and only if, for each imaginary unit $I\in\sph_A$ and each splitting base $\Bb_I$ of $A$ associated to $I$, the points $y_I$ and $y_I^c$ are poles (possibly of different order) or removable singularities for $f^{\Bb_I}_0,\ldots,f^{\Bb_I}_{u_A}$ and at least one between $y_I$ and $y_I^c$ is a pole for at least one of the functions $f^{\Bb_I}_0,\ldots,f^{\Bb_I}_{u_A}$.

(iii) An essential singularity if and only if, for each imaginary unit $I\in\sph_A$ and each splitting base $\Bb_I$ of $A$ associated to $I$, at least one among $y_I$ and $y_I^c$ is an essential singularity for at least one of the functions $f^{\Bb_I}_0,\ldots,f^{\Bb_I}_{u_A}$. \hfill$\sqbullet$
\end{remark}

\subsection{Semiregular slice functions}

Semiregular slice functions are the analogous of meromorphic functions in the hypercomplex setting. Let $A$ be either the algebra of quaternions $\HH$ or the algebra of octonions $\O$ and $\Omega\subset A$ a (non-empty) open circular set.

\begin{defn}[Semiregular slice function]\label{semiregularint}
 A function $f$ is called a \textit{semiregular slice function} on $\Omega$ if there exists a closed circular subset $\Omega_0\subset \Omega$, which is a union of isolated points on the real axes and isolated spheres of the form $\sph_y$ for $y\in\Omega\setminus \R$, such that 
\begin{itemize}
\item $f\in \mathcal{SR}_A(\Omega\setminus \Omega_0)$,
\item each isolated point $y\in\Omega_0\cap \R$ is an isolated pole or an isolated removable singularity for $f$,
\item each isolated sphere $\sph_y\subset \Omega_0$ for $y\in \Omega\setminus \R$ is an isolated spherical pole or an isolated spherical removable singularity for $f$. 
\end{itemize}
We denote by $\mathcal{SEM}_A(\Omega)$ the set of all semiregular slice functions on $\Omega$. \hfill$\sqbullet$
\end{defn}

The addition, slice multiplication and conjugation are well defined operations on $\mathcal{SEM}_A(\Omega)$. This follows by the fact that if two closed subsets $\Omega_0,\Omega_1\subset \Omega$ are unions of isolated points on the real axes and isolated spheres of the form $\sph_y$ for $y\in\Omega\setminus \R$, then $\Omega_0\cup \Omega_1$ is still closed set which is a union of isolated points on the real axes and isolated spheres of the form $\sph_y$ for $y\in\Omega\setminus \R$. In particular, the set $\mathcal{SEM}_A(\Omega)$ can be endowed (naturally) with a structure of $*$-algebra.

\begin{thm}[{\cite[Thm.6.6]{gps}}]\label{*semiregular}
The set $\mathcal{SEM}_A(\Omega)$ of all semiregular slice functions on $\Omega$ is an alternative $*$-algebra if endowed with $+$, $\cdot$, $\cdot^c$. If $A=\HH$, then $\mathcal{SEM}_\HH(\Omega)$ is associative. Moreover,
\begin{itemize}
\item if $\Omega$ is a symmetric slice domain, then $\mathcal{SEM}_A(\Omega)$ is a division algebra,
\item if $\Omega$ is a product domain, then $\mathcal{SEM}_A(\Omega)$ includes some element $f\not\equiv 0$ with $N(f)\equiv 0$. However, every element $f$ with $N(f)\not\equiv 0$ admits a multiplicative inverse within the algebra.
\end{itemize}
\end{thm}

\section{Properties of complex Nash functions in one complex variable}\label{compNashSec}

In this section, we study several (finiteness) properties for complex Nash functions defined on open subsets $D\subset \C$. We will use these results in \S\ref{SliceNashProp} to show their counterpart in the quaternionic and octonionic settings. Most of these results are probably well known by the experts. We include them here with full detailed proofs as we could not find precise references for them.

A subset $D\subset \C$ is called a \textit{domain} if it is open and connected. Given any open subset $D\subset \C$ and any function $f:D\to \C$, we denote by $V(f)$ the \textit{zero set of $f$} (\textit{in $D$}), that is the set $V(f):=\{z\in D : f(z)=0\}$. For each $z_0\in \C$ and each $R>0$ we denote by 
$$
B(z_0,R):=\{z\in\C : |z-z_0|<R\}
$$ 
the \textit{open ball with centre $z_0$ and radius $R$} and by $\overline{B}(z_0,R):=\{z\in\C : |z-z_0|\leq R\}$ the \textit{closed ball of centre $z_0$ and radius $R$}, which is the closure of $B(z_0,R)$ with respect to the Euclidean topology. 

\subsection{Real and complex Nash functions} We start by recalling the definition of real and complex Nash functions. Let $k$ be either the field of real numbers $\R$ or the field of complex numbers $\C$. Let $D\subset k^n$ be a (non-empty) open subset and $f:D\to k$ a $k$-valued function. 

\begin{defn}\label{defNash}
We say that $f$ is a \textit{$k$-Nash function at $x_0\in D$} if there exist an open neighbourhood $U$ of $x_0$ in $D$ and a non-zero polynomial $P\in k[\x,\t]:=k[\x_1,\ldots,\x_n,\t]$ such that 
\begin{itemize}
\item $f$ is $k$-analytic on $U$,
\item $P(x,f(x))=0$ for each $x\in U$. 
\end{itemize}
The function $f$ is said to be a \textit{$k$-Nash function on $D$} if it is a Nash function at every point of $D$. We denote by $\mathcal{N}_k(D)$ the set of all $k$-Nash functions on $D$ . \hfill$\sqbullet$
\end{defn}

We will also call $\R$-Nash functions \textit{real Nash functions} and $\C$-Nash functions \textit{complex Nash functions}. As being $k$-analytic is a local property, we have that, if a function $f$ is $k$-Nash on $D$, then $f$ is also $k$-analytic on $D$. By \cite[Cor.8.1.6]{bcr}, it follows that the set $\mathcal{N}_{\R}(D)$ endowed with the pointwise addition and multiplication of functions is a subring of the ring of real analytic functions on $D$. While, by \cite[Cor.1.11]{t}, the set $\mathcal{N}_{\C}(D)$ endowed with the pointwise addition and multiplication of functions is a subring of the ring of holomorphic functions on $D$. 

Let $f:D\to k$ be a $k$-analytic function and assume that $D$ is an open and connected subset of $k^n$. As for each polynomial $P\in k[\x,\t]$ the function $D\to k, \, x\mapsto P(x,f(x))$ is $k$-analytic, we deduce (straightforwardly) that the following are equivalent:
\begin{itemize}
\item[{\rm (i)}] $f\in\mathcal{N}_k(D)$,
\item[{\rm (ii)}] there exists $x_0\in D$ such that $f$ is $k$-Nash at $x_0$,
\item[{\rm (iii)}] there exists an irreducible polynomial $P\in k[\x,\t]$ such that $P(x,f(x))=0$ on $D$. 
\end{itemize}
In particular, if $D$ has finitely many connected components, then $f\in\mathcal{N}_k(D)$ if and only if there exists a non-zero polynomial $P\in k[\x,\t]$ such that $P(x,f(x))=0$ for each $x\in D$. If $D$ has infinitely many connected components, $f$ might not be algebraic over the ring $k[\x,\t]$. That is, in general, the existence of a non-zero polynomial $P\in k[\x,\t]$ such that $P(x,f(x))=0$ for each $x\in D$ is not guaranteed, see \cite[Ex.2.2]{ca}.

\subsection{Zero sets of complex Nash functions}

In this section, we show that non-zero complex Nash functions defined on domains of $\C$ have finitely many zeros. Namely, we show the following:

\begin{lem}[Finiteness of zeros of complex Nash functions]\label{lemmafiniti}
Let $D\subset \C$ be a domain and $f:D\to \C$ a non-zero complex Nash function. Then, $V(f)$ is a finite set.
\end{lem}
\begin{proof}
As $f$ is a complex Nash function on $D$ and $D$ is connected, then there exists a non-zero polynomial $P\in\C[\z_1,\z_2]$ such that $P(z,f(z))=0$ for each $z\in D$. If $f$ has infinite many zeros on $D$, then the polynomial $P(\z_1,0)$ has infinite many roots. This means that the polynomial $P(\z_1,0)$ is the zero polynomial. Thus, there exists an integer $k\geq 1$ such that $P(\z_1,\z_2)=\z_2^k Q(\z_1,\z_2)$ for some non-zero polynomial $Q\in\C[\z_1,\z_2]$ such that $Q(\z_1,0)$ is not the zero polynomial. In particular,
$
P(z,f(z))=f(z)^kQ(z,f(z))=0
$
for each $z\in D$. The holomorphic function $z\mapsto Q(z,f(z))$ vanishes on the non-empty open set $D\setminus V(f)$. Thus, $Q(z,f(z))=0$ for each $z\in D$. We conclude that $V(f)$ is finite, because otherwise the polynomial $Q(\z_1,0)$ would have infinitely many zeros, so $Q(\z_1,0)$ would be the zero polynomial.
\end{proof}

\subsection{Meromorphic complex Nash functions}\label{MeroCNash}

In this section, we study properties of meromorphic complex Nash functions. Let $D\subset \C$ be an open subset. Recall that a function $f$ is a \textit{meromorphic function on} $D$ if there exists a closed and discrete subset $D_0$ of $D$ such that $f:D\setminus D_0\to \C$ is a holomorphic function and each point $z\in D\setminus D_0$ is a removable singularity or a pole of $f$. We say that $f$ is a \textit{meromorphic complex Nash function on} $D$ if $f$ is a meromorphic function on $D$ and there exists a closed and discrete subset $D_0$ of $D$ such that $f\in\mathcal{N}_{\C}(D\setminus D_0)$.

We start by recalling the following well known bound for the modules of roots of a complex polynomial. We include its proof here for the sake of completeness. 

\begin{lem}\label{boundzero}
Let $P(\z):=\alpha_n\z^n+\ldots+\alpha_1\z+\alpha_0\in \C[\z]$ be a complex polynomial such that $\alpha_n\neq 0$ and define
$$
R:=1+\left|\frac{\alpha_{n-1}}{\alpha_n}\right|+\ldots+\left|\frac{\alpha_{0}}{\alpha_n}\right|.
$$
If $z\in \C$ is such that $P(z)=0$, then $|z|<R$.
\end{lem}
\begin{proof}
In order to conclude, it is enough to show that if $z\in \C$ is such that $|z|>R$, then $P(z)\neq 0$. Let $z\in\C$ be such that $|z|\geq R$. We have
\begin{equation}\label{diversozeroSec}
P(z)=\alpha_n z^n\Big(1+\frac{\alpha_{n-1}}{\alpha_n}z^{-1}+\ldots+\frac{\alpha_{0}}{\alpha_n}z^{-n}\Big).
\end{equation}
As $|z|\geq R\geq 1$, it holds
$$
\left|\frac{\alpha_{n-1}}{\alpha_n}z^{-1}+\ldots+\frac{\alpha_{0}}{\alpha_n}z^{-n}\right|\leq \frac{1}{R}\Big(\left|\frac{\alpha_{n-1}}{\alpha_n}\right|+\ldots+\left|\frac{\alpha_{0}}{\alpha_n}\right|\Big)< 1,
$$
so, in particular,
$$
1+\frac{\alpha_{n-1}}{\alpha_n}z^{-1}+\ldots+\frac{\alpha_{0}}{\alpha_n}z^{-n}\neq 0
$$
By \eqref{diversozeroSec}, we conclude that $P(z)\neq 0$, as required.
\end{proof}

Next, we study the singularities of (germs of) complex Nash functions.

\begin{lem}[Singularities of complex Nash functions]\label{onlypoles}
Let $D\subset \C$ be an open subset and $z_0\in D$. If $f:D\setminus\{z_0\}\to \C$ is a complex Nash function, then $f$ is a meromorphic complex Nash function on $D$.
\end{lem}
\begin{proof}
Up to an affine change of coordinates, we may assume that $z_0=0$. Moreover, as we are interested in the behaviour of $f$ locally around 0, we may also assume that $D$ is connected. By the Riemann extension theorem \cite[Thm.1.5.2]{Na}, it is enough to show: \textit{There exists an integer $k\geq 0$ such that the function $z^kf(z)$ is bounded locally around 0.} 

As $f$ is a Nash function on $D\setminus\{0\}$ and $D\setminus\{0\}$ is connected, then there exists a non-zero polynomial 
$$
P(\z_1,\z_2):=\alpha_n(\z_1)\z_2^n+\ldots+\alpha_1(\z_1)\z_2+\alpha_0(\z_1)\in \C[\z_1,\z_2]=\C[\z_1][\z_2]
$$ 
such that $P(z,f(z))=0$ for each $z\in D\setminus\{0\}$. We may assume that $\alpha_n$ is not the zero polynomial. For each $z\in D\setminus\{0\}$, we have that $f(z)$ is a zero for the polynomial $P(z,\z_2)\in\C[\z_2]$. Thus, by Lemma \ref{boundzero}, we have
\begin{equation}\label{stimapoli}
|f(z)|\leq 1+\left|\frac{\alpha_{n-1}(z)}{\alpha_n(z)}\right|+\ldots+\left|\frac{\alpha_{0}(z)}{\alpha_n(z)}\right| 
\end{equation}
for each $z\in D\setminus\{0\}$ such that $\alpha_n(z)\neq 0$. As $D$ is open and $\alpha_n$ is not the zero polynomial, then there exists $R>0$ such that the closed ball $\overline{B}(0,R)$ is contained in $D$ and $\alpha_n(z)\neq 0$ for each $z\in \overline{B}(0,R)\setminus\{0\}$. Let 
$$
M:=\max_{\ell=0,\ldots,n-1}\max_{z\in \overline{B}(0,R)} |\alpha_{\ell}(z)|.
$$
As $\alpha_n$ is not the zero polynomial, then there exist an integer $k\geq 0$ and a polynomial $Q\in \C[\z]$, with $Q(0)\neq 0$, such that $\alpha_n(\z)=\z^k Q(\z)$. Observe that $Q(z)\neq 0$ for each $z\in \overline{B}(0,R)$. Let $L:=\min_{z\in \overline{B}(0,R)} |Q(z)|>0$. It holds
$
|\alpha_n(z)|=|z^kQ(z)|\geq L|z^k|
$
for each $z\in \overline{B}(0,R)$. By \eqref{stimapoli}, we have
$$
|f(z)|\leq 1+\frac{(n-1) M}{|\alpha_n(z)|}\leq 1+\frac{(n-1) M}{L|z^k|}
$$
for each $z\in \overline{B}(0,R)\setminus\{0\}$. We deduce
$$
|z^kf(z)|\leq |z^k|+\frac{(n-1) M}{L}\leq \max_{z\in \overline{B}(0,R)}|z^k|+\frac{(n-1) M}{L},
$$
for each $z\in \overline{B}(0,R)\setminus\{0\}$, so the function $z^kf(z)$ is bounded locally around zero, as desired.
\end{proof}

By Lemma \ref{onlypoles}, we deduce that meromorphic complex Nash functions have finitely many poles on domains of $\C$.

\begin{lem}[Finiteness of poles of meromorphic complex Nash functions]\label{finitepoles}
Let $D\subset \C$ be a domain and $f$ a meromorphic complex Nash function on $D$. Let $D_0\subset D$ be a closed and discrete subset such that $f\in\mathcal{N}_{\C}(D\setminus D_0)$. Then there exist 
\begin{itemize}
\item finitely many (possibly none) points $z_1,\ldots,z_k \in D_0$,
\item a complex Nash function $\Phi: D\setminus \{z_1,\ldots,z_k\}\to \C$,
\end{itemize}
such that $\Phi$ is meromorphic on $D$ and $f=\Phi|_{D\setminus D_0}$.
\end{lem}
\begin{proof}
As $f$ is a complex Nash function on $D\setminus D_0$ and $D\setminus D_0$ is connected, then there exists a non-zero polynomial 
$$
P(\z_1,\z_2):=\alpha_n(\z_1)\z_2^n+\ldots+\alpha_1(\z_1)\z_2+\alpha_0(\z_1)\in \C[\z_1,\z_2]=\C[\z_1][\z_2]
$$ 
such that $P(z,f(z))=0$ for each $z\in D\setminus D_0$. We may assume that $\alpha_n$ is not the zero polynomial. In particular, $V(\alpha_n)$ is a finite subset of $\C$. By Lemma \ref{onlypoles}, we only need to show that if $z_0\in D_0\setminus V(\alpha_n)$, then $z_0$ is a removable singularity for $f$. Let $z_0\in D_0\setminus V(\alpha_n)$. As $\alpha_n(z_0)\neq 0$, by \eqref{stimapoli}, we have that $f$ is locally bounded around $z_0$. By the Riemann extension theorem \cite[Thm.1.5.2]{Na}, we conclude that $z_0$ is a removable singularity for $f$, as required.
\end{proof}

By the previous results, we deduce straightforwardly that complex Nash functions on $D\setminus D_0$ are meromorphic on $D$ and have finitely many poles on each connected component of $D$. 

\begin{cor}\label{mero}
Let $D\subset \C$ be an open subset and $D_0\subset D$ a closed and discrete subset. If $f$ is a complex Nash function on $D\setminus D_0$, then $f$ is a meromorphic function on $D$ with finitely many poles on each connected component of $D$.
\end{cor}

We end this section by showing that global meromorphic complex Nash functions are actually rational functions. Compare this result with \cite[Thm.1.3]{t}.

\begin{thm}[Global meromorphic complex Nash functions]\label{globalmer}
Let $D_0\subset \C$ be a closed and discrete subset and $f:\C\setminus D_0\to \C$ a complex Nash function. Then there exist two polynomials $P,Q\in\C[\z]$ such that $f(z)=\tfrac{P(z)}{Q(z)}$ for each $z\in \C\setminus D_0$.
\end{thm}
\begin{proof}
By Lemma \ref{finitepoles}, there exist finitely many (possibly none) points $z_1,\ldots,z_k\in D_0$ and a complex Nash function $\Phi:\C\setminus\{z_1,\ldots,z_k\}\to \C$ which is meromorphic on $\C$ and such that $f=\Phi|_{\C\setminus D_0}$. For each $j=1,\ldots,k$ denote by $\ord_{z_j}(\Phi)\geq 0$ the order of the meromorphic function $\Phi$ at $z_j$. Consider the polynomial
$$
Q(\z):=\prod_{j=1}^k(\z-z_j)^{\ord_{z_j}(\Phi)}\in\C[\z]
$$
As both $\Phi$ and $Q$ are complex Nash functions on $\C\setminus\{z_1,\ldots,z_k\}$, by \cite[Cor.1.11]{t}, we deduce that also the function $G:=Q\cdot \Phi$ is a complex Nash function on $\C\setminus\{z_1,\ldots,z_k\}$. As $\C\setminus \{z_1,\ldots,z_k\}$ is connected and $G$ is a complex Nash function on $\C\setminus\{z_1,\ldots,z_k\}$, then there exists a non-zero polynomial $R\in \C[\z,\t]$ such that
\begin{equation}\label{nasheqzero}
R(z,G(z))=0
\end{equation}
for each $z\in\C\setminus\{z_1,\ldots,z_k\}$. 

Let $j\in\{1,\ldots,k\}$. Using the Laurent expansion of $\Phi$ around $z_j$, we deduce that the function $z\mapsto (z-z_j)^{\ord_{z_j}(\Phi)}\Phi(z)$ is holomorphic in $z_j$. In particular, the function $G=Q\cdot \Phi$ is holomorphic on $\C$. Thus, the function $z\mapsto R(z,G(z))$ is also holomorphic on $\C$. By \eqref{nasheqzero}, we deduce that 
$
R(z,G(z))=0
$
for each $z\in \C$. In particular, $G=Q\Phi$ is a complex Nash function on $\C$. By \cite[Thm.1.3]{t}, there exists a polynomial $P\in \C[\z]$ such that $Q(z)\Phi(z)=P(z)$ for each $z\in \C$, so
$$
\Phi(z)=\frac{P(z)}{Q(z)}
$$
for each $z\in\C\setminus\{z_1,\ldots,z_k\}$. As $f=\Phi|_{\C\setminus D_0}$ and $\{z_1,\ldots,z_k\}\subset D_0$, we conclude that 
$$
f(z)=\frac{P(z)}{Q(z)}
$$
for each $z\in \C\setminus D_0$, as required.
\end{proof}

\subsection{Polynomial bounds at infinity for complex Nash functions}\label{BoundsSec}

In this section we show that the module of a complex Nash function is polynomially bounded at infinity. We start with the following elementary lemma that we include here for the sake of completeness.

\begin{lem}\label{polineq}
Let $P\in\C[\z]$ be a polynomial of degree $n\geq 0$. Then there exists a constant $C\geq 0$ such that 
$
|P(z)|\leq C(1+|z|^n)
$
for each $z\in \C$.
\end{lem}
\begin{proof}
Assume that $P(\z):=\alpha_n\z^n+\ldots+\alpha_1\z+\alpha_0$ and define
$
C:=|\alpha_0|+\ldots+|\alpha_n|.
$
It holds $|P(z)|\leq C$ if $|z|\leq 1$ and $|P(z)|\leq C|z|^n$ if $|z|>1$. Thus, $|P(z)|\leq C(1+|z|^n)$ for each $z\in \C$, as required.
\end{proof}

Next, we show the following:

\begin{prop}[Polynomial bounds for complex Nash functions]\label{polboundC}
Let $D\subset \C$ be an unbounded domain and $f\in \mathcal{N}_\C(D)$. Then there exist an integer $m\geq 0$ and constants $C,R\geq1$ such that 
$$
|f(z)|\leq C(1+|z|^m)
$$
for each $z\in D$ such that $|z|\geq R$.
\end{prop}
\begin{proof}
As $f$ is a complex Nash function on $D$ and $D$ is connected, then there exists a non-zero polynomial 
$$
P(\z_1,\z_2):=\alpha_n(\z_1)\z_2^n+\ldots+\alpha_1(\z_1)\z_2+\alpha_0(\z_1)\in \C[\z_1,\z_2]=\C[\z_1][\z_2]
$$ 
such that $P(z,f(z))=0$ for each $z\in D$. We may assume that $\alpha_n$ is not the zero polynomial. For each $z\in D$, we have that $f(z)$ is a zero for the polynomial $P(z,\z_2)\in\C[\z_2]$. Thus, by Lemma \ref{boundzero}, we have
\begin{equation}\label{stimapoli2}
|f(z)|\leq 1+\left|\frac{\alpha_{n-1}(z)}{\alpha_n(z)}\right|+\ldots+\left|\frac{\alpha_{0}(z)}{\alpha_n(z)}\right| 
\end{equation}
for each $z\in D$ such that $\alpha_n(z)\neq 0$. As $\alpha_n$ is not the zero polynomial, then there exists $R\geq1$ such that $V(\alpha_n)\subset  B(0,R)$. As 
\begin{itemize}
\item $\C\setminus B(0,R)$ is a closed subset of $\C$, 
\item $|\alpha_n(z)|$ diverges on complements of compact sets 
\item $\alpha_n(z)\neq 0$ for each $z\in  \C\setminus B(0,R)$,
\end{itemize} 
then there exists
$$
C_1:=\min_{z\in  \C\setminus B(0,R)}|\alpha_n(z)|>0.
$$
By Lemma \ref{polineq}, for each $k=0,\ldots,n-1$ there 
exist an integer $m_k\geq 0$ and a constant $C_{2,k}\geq 0$ such that
$
|\alpha_k(z)|\leq C_{2,k}(1+|z|^{m_k})
$
for each $z\in \C\setminus B(0,R)$. Let
$$
C_2:=\max_{k=0,\ldots,n-1}\{ C_{2,k}\} \quad \text{and}\quad m:=\max_{k=0,\ldots,n-1} \{m_k\}
$$
Observe that $|z|^{m_k}\leq |z|^m$ for each $z\in \C\setminus B(0,R)$ and each $k=0,\ldots,n-1$, because $R\geq 1$. By \eqref{stimapoli2}, we deduce
\begin{align*}
|f(z)|&\leq 1+\left|\frac{C_{2,n-1}(1+|z|^{m_{n-1}})}{C_1}\right|+\ldots+\left|\frac{C_{2,0}(1+|z|^{m_{0}})}{C_1}\right| \\
&\leq 1+(n-1)\frac{C_2}{C_1}(1+|z|^m)\leq  \Big(1+(n-1)\frac{C_2}{C_1}\Big)(1+|z|^m)
\end{align*}
for each $z\in D\setminus B(0,R)$. Define $C:=1+(n-1)\tfrac{C_2}{C_1}$ and observe that $C\geq 1$. We conclude that $|f(z)|\leq C(1+|z|^m)$ for each $z\in D$ such that $|z|\geq R$, as required.
\end{proof}

The following example shows that the converse of the previous result is false.

\begin{example}\label{irrationalpower}
Let $D:=\C\setminus\{z\in \R, \, z\leq 0\}$  and let $\gamma\in \R\setminus \Q$ be an irrational number such that $\gamma>0$ (here $\Q$ denotes, as usual, the field of rational numbers). Let $z^{\gamma}$ be the holomorphic branch of $\z^{\gamma}:=e^{\gamma\log\z}$ such that $1^{\gamma}=1$. Let $m\geq 1$ be an integer such that $\gamma<m$. We have 
$$
|z^{\gamma}|\leq 1+|z|^m,
$$
for each $z\in D$. Let
$
P(\z_1,\z_2):=\sum_{r,s=0}^n\alpha_{r,s}\z_1^{r}\z_2^s\in\C[\z_1,\z_2]
$
be a polynomial such that $P(z,z^{\gamma})=0$ for each $z\in D$. This means,
$$
P(z,z^{\gamma})=\sum_{r,s=0}^n\alpha_{r,s}z^{r+s\gamma}=0
$$
for each $z\in D$. As $\gamma$ is irrational, then $r_1+s_1\gamma\neq r_2+s_2\gamma$ for each $(r_1,s_1),(r_2,s_2)\in\mathbb{N}^2$ such that $(r_1,s_1)\neq (r_2,s_2)$ (here $\mathbb{N}$ denotes, as usual, the set of non-negative integers). We claim: \textit{$P$ is the zero polynomial.}

In order to show that $P$ is the zero polynomial, it is enough to show that the family of functions
$$
\Ff:=\{z^\delta:D\to \C :\, \delta\in \R \textit{ and } \delta>0\}
$$ 
is linearly independent over $\C$. Let $\delta_1,\ldots,\delta_k>0$ be such that $\delta_i\neq \delta_j$ if $i\neq j$. We may assume that $\delta_1<\delta_i$ for each $i\neq 1$. In particular, $z^{\delta_i-\delta_1}\in\Ff$ for each $i\neq 1$. Assume that there exist $\beta_1,\ldots,\beta_k\in \C$ such that
$
\beta_1z^{\delta_1}+\ldots+\beta_kz^{\delta_k}=z^{\delta_1}(\beta_1+\beta_2z^{\delta_2-\delta_1}+\ldots+\beta_kz^{\delta_k-\delta_1})=0
$
for each $z\in D$. In particular, 
$
\beta_1+\beta_2z^{\delta_2-\delta_1}+\ldots+\beta_kz^{\delta_k-\delta_1}=0
$
for each $z\in D$. Letting $z$ approach to zero, we derive that $\beta_1=0$, that is 
$
\beta_2z^{\delta_2-\delta_1}+\ldots+\beta_kz^{\delta_k-\delta_1}=0
$
for each $z\in D$. Using an inductive argument, we deduce that the family $\mathcal{F}$ is linearly independent over $\C$. Thus, $P$ is the zero polynomial, as required. \hfill$\sqbullet$
\end{example}

Let $D\subset \C$ be an unbounded domain and $f$ a meromorphic complex Nash function on $D$. Let $D_0\subset D$ be a closed and discrete subset such that $f\in\mathcal{N}_{\C}(D\setminus D_0)$. By Lemma \ref{finitepoles}, there exist  
\begin{itemize}
\item finitely many (possibly none) points $z_1,\ldots,z_k \in D_0$,
\item a complex Nash function $\Phi: D\setminus \{z_1,\ldots,z_k\}\to \C$,
\end{itemize}
such that $\Phi$ is meromorphic on $D$ and $f=\Phi|_{D\setminus D_0}$. In particular, There exists $R'>1$ such that $\Phi$ is a complex Nash function on $D\setminus B(0,R')$. Thus, by Proposition \ref{polboundC}, there exist an integer $m\geq 0$ and constants $C,R''\geq1$ such that 
$$
|\Phi(z)|\leq C(1+|z|^m)
$$
for each $z\in D$ such that $|z|\geq \max\{R',R''\}$. In particular, we deduce the following:

\begin{cor}
Let $D\subset \C$ be an unbounded domain and $f$ a meromorphic complex Nash function on $D$. Let $D_0\subset D$ be a closed and discrete subset such that $f\in\mathcal{N}_{\C}(D\setminus D_0)$. Then, there exist an integer $m\geq 0$ and constants $C,R\geq1$ such that 
$$
|f(z)|\leq C(1+|z|^m)
$$
for each $z\in D\setminus D_0$ such that $|z|\geq R$.
\end{cor}

\section{What is the `correct' definition of slice-Nash functions?}\label{buonadef}

Real and complex Nash functions are defined as those analytic functions that are algebraic over the ring of polynomials (see Definition \ref{defNash}). Thus, it seems natural to extend the definition of Nash functions to quaternions and octonions by considering those functions that are slice regular and algebraic over the slice polynomials. Let us focus on $\HH$ for the moment. As $\HH$ is not commutative, what does it mean for a slice function to be \textit{`algebraic over the slice polynomials'} in this context? There are several difficulties to be addressed in order to give the `correct' definition of Nash functions over $\HH$ that we point out in this section. We start by explaining why a straightforward generalisation of Definition \ref{defNash} is not the `correct' notion for Nash functions on $\HH$. 

Let $\Omega\subset \HH$ be an open and connected circular set and $f:\Omega\to \HH$ a slice function. We may try to introduce the following:

\begin{defn}\label{algH}
We say that the function $f$ is \textit{right algebraic (over $\HH[\q_1,\q_2]$)} if there exists a non-zero polynomial $P\in \HH[\q_1,\q_2]$ such that $P(q,f(q))=0$ for each $q\in \Omega$. \hfill$\sqbullet$
\end{defn}

The following example shows that there are `too few' right algebraic slice regular functions. Namely, some slice regular functions that should be considered Nash functions, according to the classical theory of real and complex Nash functions, are not right algebraic. It is worthwhile to notice that the function $f(z):=\sqrt{z}\iota$ (where $\sqrt{z}$ is the branch of $\sqrt{\z}$ such that $\sqrt{1}=1$) is a complex Nash function on $D:=\C\setminus \{z\in \R : z\leq 0\}$, because it is holomorphic on $D$ and satisfies $P(z,f(z))=0$ for each $z\in D$, where $P(\z_1,\z_2):=\z_1+\z_2^2$.

\begin{example}\label{poche}
\textit{Let $\Omega:=\HH\setminus \{q\in \R: q\leq0\}$ and let $J\in\sph_\HH$ be an imaginary unit. Let $f:\Omega\to \HH$ be the slice regular branch of $\sqrt{\q}J$ such that $f(1)=J$. If $P\in\HH[\q_1,\q_2]$ is a slice polynomial such that $P(q,f(q))=0$ for each $q\in\Omega$, then $P$ is the zero polynomial.}

Let $I\in\sph_\HH$ be an imaginary unit such that $IJ=-JI$ and $P\in\HH[\q_1,\q_2]$ a slice polynomial such that $P(q,f(q))=0$ for each $q\in\Omega$. As $\HH\equiv \C_I\oplus \C_IJ$ (as $\C_I$-vector spaces), we may write
$$
P(\q_1,\q_2)=\sum_{r,\ell=0}^n \q_1^{r}\q_2^{\ell}(\alpha_{r,\ell}+\beta_{r,\ell}J)
$$
for some integer $n\geq 1$ and $\alpha_{r,\ell},\beta_{r,\ell}\in \C_I$. Up to add some zero coefficients if necessary, we may assume that $n$ is odd, that is $n=2m+1$ for some integer $m\geq 0$. This will help to simplify the notation in what follows. 

Let $\Omega_I:=\Omega\cap \C_I$. Clearly, we have $P(z_I,f|_{\Omega_I}(z_I))=0$ for each $z_I\in \Omega_I$. In order to simplify the explicit calculation, define the function $g:=f|_{\Omega_I}J^{-1}:\Omega_I\to \HH$. For each $z_I\in\Omega$, we have $f|_{\Omega_I}(z_I)=g(z_I)J $ and $g(z_I)=\sqrt{z_I}\in \C_I$, where $\sqrt{z_I}$ is the holomorphic branch of $\sqrt{\z_I}$ such that $\sqrt{1}=1$. Observe that 
$$
g(z_I)J=J\overline{g(z_I)}
$$
for each $z_I\in \Omega_I$, because $IJ=-JI$. Thus, for each integer $\ell\geq 0$, we have
$$
(g(z_I)J)^{\ell}=
\begin{cases}
(-1)^{\frac{\ell}{2}}g(z_I)^{\frac{\ell}{2}}\overline{g(z_I)}^{\frac{\ell}{2}}, & \text{ if } \ell \text{ is even},\\
 (-1)^{\frac{\ell-1}{2}} g(z_I)^{\frac{\ell+1}{2}}\overline{g(z_I)}^{\frac{\ell-1}{2}}J, & \text{ if } \ell \text{ is odd},
\end{cases}
$$
for each $z_I\in \Omega_I$. As $P(z_I,f|_{\Omega_I}(z_I))=0$ for each $z_I\in \Omega_I$, we deduce
\begin{align*}
P(z_I,&f|_{\Omega_I}(z_I))=P(z_I,g(z_I)J)=\sum_{r,\ell=0}^n z_I^{r}(g(z_I)J)^{\ell}(\alpha_{r,\ell}+\beta_{r,\ell}J)\\
&=\!\begin{multlined}[t][10.5cm]\sum_{r=0}^n z_I^{r}\sum_{\ell \text{\, even}} (-1)^{\frac{\ell}{2}}g(z_I)^{\frac{\ell}{2}}\overline{g(z_I)}^{\frac{\ell}{2}}(\alpha_{r,\ell}+\beta_{r,\ell}J)\\
+\sum_{r=0}^n z_I^{r}\sum_{\ell \text{\, odd}} (-1)^{\frac{\ell-1}{2}} g(z_I)^{\frac{\ell+1}{2}}\overline{g(z_I)}^{\frac{\ell-1}{2}}J(\alpha_{r,\ell}+\beta_{r,\ell}J)\end{multlined}\\
&=\!\begin{multlined}[t][10.5cm]\sum_{r=0}^n z_I^{r}\Big(\sum_{\ell \text{\, even}} (-1)^{\frac{\ell}{2}}g(z_I)^{\frac{\ell}{2}}\overline{g(z_I)}^{\frac{\ell}{2}}\alpha_{r,\ell}-\sum_{\ell \text{\, odd}} (-1)^{\frac{\ell-1}{2}} g(z_I)^{\frac{\ell+1}{2}}(\overline{g(z_I)}^{\frac{\ell-1}{2}})\overline{\beta_{r,\ell}}\Big)\\
+\sum_{r=0}^n z_I^{r}\Big(\sum_{\ell \text{\, even}} (-1)^{\frac{\ell}{2}}g(z_I)^{\frac{\ell}{2}}\overline{g(z_I)}^{\frac{\ell}{2}}\beta_{r,\ell}+\sum_{\ell \text{\, odd}} (-1)^{\frac{\ell-1}{2}} g(z_I)^{\frac{\ell+1}{2}}(\overline{g(z_I)}^{\frac{\ell-1}{2}})\overline{\alpha_{r,\ell}}\Big)J\end{multlined}\\
&=0
\end{align*}
for each $z_I\in\Omega_I$. In particular, as $\{1,J\}$ is a base of $\HH$ as a $\C_I$-vector space, we have
\begin{align}
\begin{split}\label{eqesemfew}
&\sum_{r=0}^n z_I^{r}\Big(\sum_{\ell \text{\, even}} (-1)^{\frac{\ell}{2}}g(z_I)^{\frac{\ell}{2}}\overline{g(z_I)}^{\frac{\ell}{2}}\alpha_{r,\ell}-\sum_{\ell \text{\, odd}} (-1)^{\frac{\ell-1}{2}} g(z_I)^{\frac{\ell+1}{2}}(\overline{g(z_I)}^{\frac{\ell-1}{2}})\overline{\beta_{r,\ell}}\Big)=0\\
&\sum_{r=0}^n z_I^{r}\Big(\sum_{\ell \text{\, even}} (-1)^{\frac{\ell}{2}}g(z_I)^{\frac{\ell}{2}}\overline{g(z_I)}^{\frac{\ell}{2}}\beta_{r,\ell}+\sum_{\ell \text{\, odd}} (-1)^{\frac{\ell-1}{2}} g(z_I)^{\frac{\ell+1}{2}}(\overline{g(z_I)}^{\frac{\ell-1}{2}})\overline{\alpha_{r,\ell}}\Big)=0
\end{split}
\end{align}
for each $z_I\in\Omega_I$. 

We want to apply \cite[Lem.3.B.13]{gr} to \eqref{eqesemfew}. In order to do this, we reorder \eqref{eqesemfew} with respect to the powers of $\overline{g(z_I)}$. Recall that $n=2m+1$. For each $k=0,\ldots,m$ define the polynomials $Q^{(1)}_k,R^{(1)}_k\in\C_I[\z_1,\z_2]$ as 
$$
Q_k^{(1)}(\z_1,\z_2):=(-1)^{k}\z_2^{k}\Big(\sum_{r=0}^n \z_1^{r}\alpha_{r,2k}\Big) \quad\text{and}\quad R_k^{(1)}(\z_1,\z_2):=(-1)^{k} \z_2^{k+1}\Big(\sum_{r=0}^n \z_1^{r}\overline{\beta_{r,2k+1}}\Big).
$$
Moreover, for each $k=0,\ldots,m$, define the polynomials $Q^{(2)}_k,R^{(2)}_k\in\C_I[\z_1,\z_2]$ as 
$$
Q_k^{(2)}(\z_1,\z_2):=(-1)^k\z_2^{k}\Big(\sum_{r=0}^n \z_1^{r}\beta_{r,2k}\Big) \quad\text{and}\quad R_k^{(2)}(\z_1,\z_2):=(-1)^{k} \z_2^{k+1}\Big(\sum_{r=0}^n \z_1^{r}\overline{\alpha_{r,2k+1}}\Big).
$$
By \eqref{eqesemfew}, we deduce that the polynomials $P_1,P_2\in\C_I[\z_1,\z_2,\z_3]$ defined as
\begin{align*}
P_1(\z_1,\z_2,\z_3)&:=\sum_{k=0}^m(Q_k^{(1)}(\z_1,\z_2)-R^{(1)}_k(\z_1,\z_2))\z_3^k\\
P_2(\z_1,\z_2,\z_3)&:=\sum_{k=0}^m(Q_k^{(2)}(\z_1,\z_2)+R^{(2)}_k(\z_1,\z_2))\z_3^k
\end{align*}
satisfy
$
P_1(z_I,g(z_I),\overline{g(z_I)})=P_2(z_I,g(z_I),\overline{g(z_I)})=0
$
for each $z_I\in\Omega_I$. Observe that the polynomials $P_1(z_I,g(z_I),\z)$ and $P_2(z_I,g(z_I),\z)$ are polynomials whose coefficients are holomorphic functions on $\Omega_I$. As $\overline{g}:\Omega_I\to \C_I$ is an anti-holomorphic function and satisfies identically the polynomials $P_1(z_I,g(z_I),\z)$ and $P_2(z_I,g(z_I),\z)$ on the open set $\Omega_I$, then, by \cite[Lem.3.B.13]{gr}, we have that both $P_1(z_I,g(z_I),\z)$ and $P_2(z_I,g(z_I),\z)$ are zero as polynomials with holomorphic coefficients. In particular,
\begin{align*}
& Q_k^{(1)}(z_I,g(z_I))-R^{(1)}_k(z_I,g(z_I))=(-1)^kg(z_I)^k\Big(\sum_{r=0}^n z_I^{r}\alpha_{r,2k}-g(z_I)\sum_{r=0}^n z_I^{r}\overline{\beta_{r,2k+1}}\Big)=0  \\
& Q_k^{(2)}(z_I,g(z_I))+R^{(2)}_k(z_I,g(z_I))=(-1)^kg(z_I)^k\Big(\sum_{r=0}^n z_I^{r}\beta_{r,2k}+g(z_I)\sum_{r=0}^n z_I^{r}\overline{\alpha_{r,2k+1}}\Big)=0
\end{align*}
for each $z_I\in\Omega_I$ and $k=0,\ldots,m$. Recall that $g(z_I)=\sqrt{z_I}$, where $\sqrt{z_I}$ is the holomorphic branch of $\sqrt{\z_I}$ such that $\sqrt{1}=1$. We deduce that
\begin{align*}
&\sum_{r=0}^n z_I^{r}\alpha_{r,2k}-\sqrt{z_I}\sum_{r=0}^n z_I^{r}\overline{\beta_{r,2k+1}}=\sum_{r=0}^n \alpha_{r,2k
}z_I^r-\sum_{r=0}^n\overline{\beta_{r,2k+1}}z_I^{r+\frac{1}{2}}=0\\
&\sum_{r=0}^n z_I^{r}\beta_{r,2k}+\sqrt{z_I}\sum_{r=0}^n z_I^{r}\overline{\alpha_{r,2k+1}}=\sum_{r=0}^n \beta_{r,2k
}z_I^r+\sum_{r=0}^n\overline{\alpha_{r,2k+1}}z_I^{r+\frac{1}{2}}=0
\end{align*}
for each $z_I\in\Omega_I$ and $k=0,\ldots,m$. We conclude that all the coefficients $\alpha_{r,\ell}$ and $\beta_{r,\ell}$ are zero for each $r,\ell=0,\ldots,n$, so $P$ is the zero polynomial, as desired. \hfill$\sqbullet$
\end{example}

Observe that, as $\HH$ is non-commutative, we may also give the left counterpart of Definition \ref{algH}. Namely, we may define slice functions \textit{left algebraic} as those slice functions $f:\Omega\to \HH$ such that there exists a non-zero polynomial $P\in \HH[\q_1,\q_2]$ with $P(f(q),q)=0$ for each $q\in\Omega$. There is no evident reason why right algebraic slice functions should be related to left algebraic slice functions. Nevertheless, it can be shown that the function $\sqrt{q}J$ defined in Example \ref{poche} is not even left algebraic (the argument is different from the one of Example \ref{poche}, because, \cite[Lem.3.B.13]{gr} cannot be applied in this case and one has to check the order of certain functions around zero). In both cases, the class of right or left algebraic slice regular functions is not the `correct' class for generalising Nash functions to the quaternionic setting.

Unlike the real and complex cases, the set of slice quaternionic polynomials $\HH[\q_1,\q_2]$ is not a ring if endowed with the sum and the quaternionic product of polynomials. But $\HH[\q_1,\q_2]$ is a ring if we endow it with the sum and the slice product. This might suggests a need to replace Definition \ref{algH} with a definition that involves the slice product instead of the quaternionic product, so to take into account the ring structure of $\HH[\q_1,\q_2]$ in analogy with the real and complex cases. 

Let $f:\Omega\to A$ be a slice function. We define inductively $f^{\bullet n}:=f^{\bullet (n-1)}\cdot f$ for each integer $n\geq 1$, where $f^{\bullet 0}=1$. Let $P\in\HH[\q_1,\q_2]$ be a slice polynomial of the form
$
P(\q_1,\q_2):=\sum_{k_1,k_2=0}^n\q_1^{k_1}\q_2^{k_2}\alpha_{k_1,k_2}.
$
Given $q\in\HH$, with the notation $P(q,f(q))^\bullet$ we mean
$$
P(q,f(q))^{\bullet}:=\sum_{k_1,k_2=0}^nq^{k_1}f^{\bullet k_2}(q)\alpha_{k_1,k_2}\Big(=\sum_{k_1,k_2=0}^nq^{\bullet k_1}\cdot f^{\bullet k_2}(q)\cdot\alpha_{k_1,k_2}\Big)\in\HH
$$
Observe that, if $f\in \mathcal{SR}_{\HH}(\Omega)$, then by \cite[Prop.11]{gp}, the function $\Omega\to \HH,\, q\mapsto P(q,f(q))^{\bullet}$ is slice regular on $\Omega$. Taking into account the slice product instead of the quaternionic product, we may introduce the following definition which is the counterpart of Definition \ref{algH} in this setting.

\begin{defn}\label{staralgH}
We say that the function $f$ is \textit{right slice-algebraic} if there exists a non-zero polynomial $P\in \HH[\q_1,\q_2]$ such that $P(q,f(q))^{\bullet}=0$ for each $q\in \Omega$. \hfill$\sqbullet$
\end{defn}

Slice functions \textit{left slice-algebraic} may be defined in the same way. Unlike the class of (right or left) algebraic slice regular functions that are `too few', the class of (right or left) slice-algebraic slice regular functions are `too many' to be the `correct' class of quaternionic slice-Nash functions. We point this out in the following example:

\begin{example}\label{troppe}
\textit{Let $I,J\in\sph_\HH$ be imaginary units such that $IJ=-JI$. Let $f:\HH\to \HH$ be the slice regular function defined as $f(q):=(\cos q)I+(\sin q)J$ for each $q\in\HH$. Then $f$ is both right and left slice-algebraic.}

In fact, consider the slice polynomials $P_s(\q_1,\q_2):=\q_s^2+1\in \HH[\q_1,\q_2]$ for $s=1,2$. We have
\begin{align*}
P_1(f(q),q)^{\bullet}=P_2(q,f(q))^{\bullet}\stackrel{(*)}{=}-\cos^2 q-\sin^2 q+\cos q\sin q(IJ+JI)+1=0
\end{align*}
for each $q\in \HH$, where the equality $(*)$ is justified by the fact that both the slice functions $\cos q$ and $\sin q$ are slice preserving. \hfill$\sqbullet$
\end{example}

As already mentioned several times, real and complex Nash functions are intrinsically of `algebraic nature' (whatever that means). The ring of $k$-Nash functions, where $k$ is either $\R$ or $\C$, is an intermediate ring `of algebraic nature' between the ring of polynomial functions, which are `too few' and `too rigid', and the ring of $k$-analytic functions, which are `too many' (for instance, the ring of global $k$-analytic functions is not Noetherian). In our opinion, the `correct' notion of slice-Nash functions on $\HH$ should maintain this nature, namely, to be an intermediate class of `algebraic nature' (whatever that means) between slice polynomials and slice regular functions. In particular, slice regular functions like the one of Example \ref{troppe} should not be Nash. 

The last notion of `algebraicity' one may consider, in order to define quaternionic Nash functions, is the notion of algebraicity over the ring of mixed polynomial, namely, the polynomials in the variables $\q_1$ and $\q_2$ expressible in the form
$$
P(\q_1,\q_2)=\sum_{s=1}^n \alpha_{s,1}\q_{\mu_{s,1}}\alpha_{s,2}\q_{\mu_{s,2}}\ldots\alpha_{s,\ell}\q_{\mu_{s,\ell}}\alpha_{s,\ell+1},
$$
where $r\geq 0$, $s_\ell\geq 0$, $\mu_{s,p}\in \{1,2\}$ and $\alpha_{s,p}\in \HH$. The ring operations are simply the usual addition and multiplication of polynomials. The following example shows that any function is algebraic over the ring of mixed polynomials (because there exist non-zero mixed polynomials $P$ that vanish identically on $\HH^2$).

\begin{example}\label{mixed}
Let $\{1,i,j,ij\}$ be the standard Hamilton base of $\HH$. Let 
$$
Q(\q):=\frac{1}{4}(\q-i\q i-j\q j- ij\q ij),
$$
which is a non-zero mixed polynomial. A straightforward computation shows that $Q(q)=\Re(q)$ for each $q\in \HH$. In particular, the non-zero mixed polynomial $P(\q_1,\q_2):=\q_1Q(\q_1)-Q(\q_1)\q_1$ satisfies $P(q,f(q))\equiv0$ for any function $f:\Omega\to \HH$ defined on any subset $\Omega\subset \HH$. \hfill$\sqbullet$
\end{example}

Examples \ref{poche}, \ref{troppe} and \ref{mixed} suggest that the `correct' definition of Nash functions on $\HH$ cannot be of `global nature'. That is, the `algebraic nature' (whatever that means) of Nash functions on $\HH$ cannot be detected globally using polynomials. Even if it may sound disappointing (at least for classical real and complex geometers), this phenomenon is not surprising. The same difficulties appeared with attempts to generalise the notion of holomorphic functions to the quaternionic setting. Both the approach of Gentili and Struppa \cite{gs,gs2} and the one of Ghiloni and Perotti \cite{gp0,gp} for generalising the notion of holomorphic functions to quaternions are intrinsically related to the `book decomposition' of $\HH$ into complex slices.

For the octonions, the situation is even more involved. If $P$ is an `ordered polynomial' in the variables $\x_1,\x_2$, then in order to have a well defined polynomial function $x\mapsto P(x,f(x))$ we have to fix the order of the parentheses on the monomials of $P$ (for instance as introduced in \S\ref{prelipol}) or considering only polynomials with real coefficients. The situation is analogous if we consider the slice product instead of the octonionic product - the quantity $x^{ s}f^{\bullet r}(x)\cdot\alpha$ is not well defined in general. 

The above discussion in this section motivated us to give a definition for slice-Nash functions over $\HH$ and $\O$ that takes into account all these phenomena that appear in the non-commutative and non-associative setting (see Definition \ref{defslicenash}). In particular, the notion of Nash functions we propose, and that we strongly believe to be the natural generalisation of the classical theory of real and complex Nash functions to this context, is related to the `book decomposition' of $\HH$ and $\O$ into complex slices (see Theorems \ref{char} and \ref{char2}) and of `algebraic nature' (see \S\ref{SliceNashProp}).

\section{Quaternionic and octonionic slice-Nash functions}\label{Maindef}

In this section, we introduce the main definition of this article. Namely, we define slice-Nash functions on open circular subsets of $\HH$ and $\O$. First, we define stem-Nash functions, then define slice-Nash functions as those slice regular functions induced by a stem-Nash function. We also provide several equivalent characterisations for slice-Nash functions and we show that slice derivatives of slice-Nash functions are still slice-Nash functions. Moreover, we show that the set of all slice-Nash functions has a (natural) structure of alternative $*$-algebra.

\subsection{Stem-Nash functions}

Let $A$ be either the algebra of quaternions $\HH$ or the algebra of octonions $\O$. Let $D\subset \C$ be a (non-empty) open subset and $F:=F_1+\iota F_2:D\to A\otimes_{\R}\C$ a stem function. Let $d_A$ and $u_A$ be the integers introduced in \eqref{dim}. Let $I\in\sph_A$ be an imaginary unit and $\Bb_I:=\{I_0:=1,I, I_1,II_1,\ldots,I_{u_A},II_{u_A}\}$ a splitting base of $A$ associated to $I$. As $\Bb_I$ is a base of $A$ as a $\R$-vector space, then there exist unique real valued functions $F_{s,\ell}^{I,I_k}:D\to  \R$ for $s,\ell=1,2$ and $k=0,\ldots,u_A$ such that
\begin{equation}\label{components}
F_s=\sum_{k=0}^{u_A}(F_{s,1}^{I,I_k}+F_{s,2}^{I,I_k}I)I_k
\end{equation}
for $s=1,2$. Define
$$
F_{I,I_k}^{\ell}:=F_{1,\ell}^{I,I_k}+\iota F_{2,\ell}^{I,I_k}:D\to \C
$$
for each $\ell=1,2$ and $k=0,\ldots,u_A$. We have
$$
F=F_1+\iota F_2=\sum_{k=0}^{u_A}(F_{I,I_k}^{1}+F_{I,I_k}^{2}I)I_k
$$
Observe that $F$ is a holomorphic stem function on $D$ if and only if the functions $F_{I,I_k}^{\ell}:D\to \C$ are holomorphic functions on $D$ for each $\ell=1,2$ and each $k=0,\ldots,u_A$ \cite[Rmk.3(2)]{gp}. 

This notation will allow us to write the statements of our results uniformly both for $\HH$ and $\O$. In order to lighten the exposition, we will use also the notation of the following example:

\begin{example}\label{notation}
(i) Let $F:=F_1+\iota F_2:D\to \HH\otimes_\R\C$ be a stem function. Let $I,J\in\sph_{\HH}$ be imaginary units such that $J\neq\pm I$. Then $\Bb_I:=\{1,I,J,IJ\}$ is a splitting base of $\HH$ associated to $I$. In particular, there exist unique real valued functions $F_{s,\ell}^{I,J}:D\to \R$ such that 
$$
F_s=F_{s,0}^{I,J}+F_{s,1}^{I,J}I+F_{s,2}^{I,J}J+F_{s,3}^{I,J}IJ
$$
for $s=1,2$. Let
$
F_{I,J}^{\ell}:=F_{1,\ell}^{I,J}+\iota F_{2,\ell}^{I,J}
$
for each $\ell=0,1,2,3$. We have
$$
F=F_1+\iota F_2=F_{I,J}^{0}+F_{I,J}^{1}I+F_{I,J}^{2}J+F_{I,J}^{3}IJ.
$$
The functions $F_{s,\ell}^{I,J}$ for $\ell=0,1,2,3(=d_{\HH})$ correspond to the functions $F_{s,1}^{I,I_k}$ and $F_{s,2}^{I,I_k}$ for $k=0,1(=u_{\HH})$ of \eqref{components}.

(ii) Let $F:=F_1+\iota F_2:D\to \O\otimes_\R\C$ be a stem function. Let $I,J,L\in\sph_{\O}$ be imaginary units such that $\Bb_I:=\{1,I,J,IJ,L,IL,JL,I(JL)\}$ is a splitting base of $\HH$ associated to $I$. In particular, there exist unique real valued functions $F_{s,\ell}^{I,J,L}:D\to \R$ such that 
$$
F_s=F_{s,0}^{I,J,L}+F_{s,1}^{I,J,L}I+F_{s,2}^{I,J,L}J+F_{s,3}^{I,J,L}IJ+F_{s,4}^{I,J,L}L+F_{s,5}^{I,J,L}IL+F_{s,6}^{I,J,L}JL+F_{s,7}^{I,J,L}I(JL)
$$
for $s=1,2$. Let
$
F_{I,J,L}^{\ell}:=F_{1,\ell}^{I,J,L}+\iota F_{2,\ell}^{I,J,L}
$
for each $\ell=0,\ldots,7$. We have
\begin{multline*}
F=F_1+\iota F_2\\
=F_{I,J,L}^{0}+F_{I,J,L}^{1}I+F_{I,J,L}^{2}J+F_{I,J,L}^{3}IJ+F_{I,J,L}^{4}L+F_{I,J,L}^{5}IL+F_{I,J,L}^{6}JL+F_{I,J,L}^{7}I(JL).
\end{multline*}
The functions $F_{s,\ell}^{I,J,L}$ for $\ell=0,\ldots,7(=d_{\O})$ correspond to the functions $F_{s,1}^{I,I_k}$ and $F_{s,2}^{I,I_k}$ for $k=0,1,2,3(=u_{\HH})$ of \eqref{components}. \hfill$\sqbullet$
\end{example}

We introduce the following definition:

\begin{defn}[Stem-Nash function]\label{defstemnash}
We say that a stem function $F:D\to A\otimes_{\R}\C$ is a \textit{stem-Nash function on $D$} if there exist $I\in \sph_A$ and a splitting base $\{I_0:=1,I, I_1,II_1,\ldots,I_{u_A},II_{u_A}\}$ of $A$ associated to $I$ such that the functions $F_{I,I_k}^{\ell}:D\to \C$ are complex Nash functions on $D$ for each $\ell=1,2$ and each $k=0,\ldots,u_A$. \hfill$\sqbullet$
\end{defn}

As complex Nash functions are in particular holomorphic (see Definition \ref{defNash}), we have that, if $F$ is a stem-Nash function on $D$, then the functions $F_{I,I_k}^{\ell}:D\to \C$ are holomorphic functions on $D$ for each $\ell=1,2$ and $k=0,\ldots,u_A$. We deduce that (see also \cite[Rmk.3(2)]{gp}): \textit{If $F$ is a stem-Nash function on $D$, then $F$ is a holomorphic stem function on $D$.}

By \cite[Prop.1.6]{t}, we deduce straightforwardly the following important result, which makes our definition consistent.

\begin{prop}[Independence from base choice I]\label{indipendente}
A stem function $F:D\to A\otimes_{\R}\C$ is a stem-Nash function on $D$ if and only if for each imaginary unit $I\in\sph_A$ and each splitting base $\{I_0:=1,I, I_1,II_1,\ldots,I_{u_A},II_{u_A}  \}$ of $A$ associated to $I$ the functions $F_{I,I_k}^{\ell}:D\to \C$ are complex Nash functions on $D$ for each $\ell=1,2$ and each $k=0,\ldots,u_A$. 
\end{prop}

Instead of introducing the notion of stem-Nash function as in Definition \ref{defstemnash}, another possibility is to focus on the structure of $A\otimes_\R \C$ as a $\R$-vector space and define stem-Nash functions as those stem functions $F$ whose real components with respect to a real base $\Bb$ of $A$ are real Nash functions. We will see that, under the hypothesis that the involved stem function $F$ is a holomorphic stem function, these two approaches are actually equivalent.

Let $F=F_1+\iota F_2:D\to A\otimes_\R\C$ be a stem function. We identify $\C$ with $\R^2$ in the standard way. Let $D_{\R}:=\{(\alpha,\beta)\in\R^2 : \alpha+\iota \beta\in D\}$. The stem function $F$ induces the (real) map
$$
F^{\R}=(F_1^{\R},F_2^{\R}):D_{\R}\to A\oplus A, \quad (\alpha,\beta)\mapsto (F_1(\alpha+\iota \beta), F_2(\alpha+\iota \beta)).
$$
In order to lighten the exposition, up to a light abuse of notation, we will write $F=(F_1,F_2)$ instead of $F^{\R}=(F_1^{\R},F_2^{\R})$. This will create no confusion, as the situation will be always clear by the context. Let $\Bb:=\{\e_0,\ldots,\e_{d_A}\}$ be any base of $A$ as a $\R$-vector space (here we are not requiring that $1\in \Bb$). Then, there exist unique real valued functions $F_{s,r}^{\Bb}:D_{\R}\to \R$ for $s=1,2$ and $r=0,\ldots,d_A$ such that 
$$
F_s=F_{s,0}^{\Bb}\e_0+\ldots+F_{s,d_A}^{\Bb}\e_{d_A}
$$
for $s=1,2$. 

We start by showing that the fact that the real components $F_{s,r}^{\Bb}:D_{\R}\to \R$ of $F=(F_1,F_2)$ are real Nash functions does not depend by the choice of the real base $\Bb$ of $A$. This is analogous to Proposition \ref{indipendente}.

\begin{prop}[Independence from base choice II]\label{indipendente2}
The following are equivalent: 
\begin{itemize}
\item[\rm{(i)}] There exists a base $\Bb$ of $A$ as a $\R$-vector space such that $F_{s,r}^{\Bb}:D_{\R}\to \R$ is a real Nash function on $D_{\R}$ for each $s=1,2$ and each $r=0,\ldots,d_A$.
\item[\rm{(ii)}] For each base $\Bb$ of $A$ as a $\R$-vector space $F_{s,r}^{\Bb}:D_{\R}\to \R$ is a real Nash function on $D_{\R}$ for each $s=1,2$ and each $r=0,\ldots,d_A$.
\end{itemize}
\end{prop}
\begin{proof}
We may assume $A=\HH$, as the case $A=\O$ is completely analogous. As (ii) implies (i) trivially, we only have to show the other implication. Assume that there exists a base $\Bb'$ such that the real components $F_{s,r}^{\Bb'}:D_{\R}\to \R$ of the map 
$
F=(F_1,F_2):D_\R\to \HH\oplus\HH
$
with respect to the base $\Bb'$ are real Nash functions on $D_{\R}$ for each $s=1,2$ and each $r=0,1,2,3$. Let $\Bb$ any base of $\HH$ as a $\R$-vector space. Then there exist (unique) $\alpha_{\ell,r}^s\in \R$ for $s=1,2$ and $\ell,r=0,1,2,3$ such that
$$
F_{s,\ell}^{\Bb}=\alpha_{\ell,0}^sF_{s,0}^{\Bb'}+\alpha_{\ell,1}^sF_{s,1}^{\Bb'}+\alpha_{\ell,2}^sF_{s,2}^{\Bb'}+\alpha_{\ell,3}^sF_{s,3}^{\Bb'}
$$
for each $s=1,2$ and $\ell=0,1,2,3$. As real Nash functions form a ring \cite[Cor.8.1.6]{bcr}, we conclude that the functions $F_{s,\ell}^{\Bb}:D_{\R}\to \R$ are real Nash functions on $D_{\R}$ for each $s=1,2$ and each $\ell=0,1,2,3$, as required.
\end{proof}

Let $I\in\sph_A$ be any imaginary unit and $\Bb_I:=\{I_0:=1,I, I_1,II_1,\ldots,I_{u_A},II_{u_A}\}$ any splitting base of $A$ associated to $I$. Let $\Bb:=\{\e_0,\ldots,\e_{d_A}\}$ be any base of $A$ as a $\R$-vector space. By \cite[Prop.2.4]{ca}, we derive that the complex components $F_{I,I_k}^{\ell}$ of a holomorphic stem function $F$ with respect to the splitting base $\Bb_I$ are complex Nash functions if and only if the real components $(F_{1,r}^{\Bb},F_{2,r}^{\Bb})$ of $F=(F_1,F_2)$ with respect to the base $\Bb$ are real Nash functions.

\begin{prop}\label{equiv}
The following are equivalent: 
\begin{itemize}
\item[\rm{(i)}] The functions
$
F_{I,I_k}^{\ell}:=F_{1,\ell}^{I,I_k}+\iota F_{2,\ell}^{I,I_k}:D\to \C
$
are complex Nash functions on $D$ for each $\ell=1,2$ and each $k=0,\ldots,u_A$.
\item[\rm{(ii)}] The stem function $F$ is holomorphic and the functions $F_{s,r}^{\Bb}:D_{\R}\to \R$ are real Nash functions on $D_{\R}$ for each $s=1,2$ and each $r=0,\ldots,d_A$.
\end{itemize}
\end{prop}
\begin{proof}
We may assume $A=\HH$, as the case $A=\O$ is completely analogous. Let $I,J\in\sph_\HH$ be imaginary units such that $J\neq \pm I$. In particular, $\Bb_I:=\{1,I,J,IJ\}$ is a splitting base of $\HH$ associated to $I$. Let $F^{I,J}_{s,\ell}:D\to \R$ for $s=1,2$ and $\ell=0,1,2,3$ be the real components of $F=F_1+\iota F_2$ with respect to the base $\Bb_I$. Define
$$
F^{\Bb_I}_{s,\ell}: D_\R\to \R, \quad (\alpha,\beta)\mapsto F^{I,J}_{s,\ell}(\alpha+\iota \beta)
$$
for $s=1,2$ and $\ell=0,1,2,3$. By \cite[Prop.2.4]{ca}, if the functions $F_{I,J}^{\ell}:=F^{I,J}_{1,\ell}+\iota F^{I,J}_{2,\ell}:D\to \C$ are complex Nash functions on $D$ for each $\ell=0,1,2,3$, then the functions $F_{s,\ell}^{I,J}:D_{\R}\to \R$ are real Nash functions on $D$ for each $s=1,2$ and $\ell=0,1,2,3$. Thus, by Proposition \ref{indipendente2}, $F_{s,r}^{\Bb}$ are real Nash functions on $D_{\R}$ for each $s=1,2$ and each $r=0,\ldots,d_A$. Moreover, as complex Nash functions are in particular holomorphic, we also have that $F$ is a holomorphic stem function on $D$, so $(i)$ implies $(ii)$.

Vice versa, if (ii) holds, then by Proposition \ref{indipendente2} we may assume that the base $\Bb$ is of the form $\{1,I,J,IJ\}$ for some $I,J\in\sph_\HH$ such that $J\neq \pm I$. As $F$ is a holomorphic stem function on $D$, then the functions $F_{I,J}^{\ell}:=F_{1,\ell}^{I,J}+\iota F_{2,\ell}^{I,J}:D\to \C$ are holomorphic functions on $D$ for each $\ell=0,1,2,3$. Thus, using again \cite[Prop.2.4]{ca}, we conclude that (ii) implies (i), as desired. 
\end{proof}

The following example shows that the hypothesis that $F$ is a holomorphic stem function in point (ii) cannot be dropped. 

\begin{example}\label{hol2}
Consider the stem function 
$
F:D\to \HH\otimes_{\R} \C, \ z\mapsto \overline{z}:=1\otimes \overline{z}.
$
Let $\Bb:=\{1,I,J,IJ\}$, where $I,J\in\sph_\HH$ are imaginary units such that $J\neq \pm I$. Consider the real components $F_{s,\ell}^{\Bb}:D_\R\to \R$ of $F=(F_1,F_2)$ with respect to the base $\Bb$. We have
$$
F_{1,0}^{\Bb}(\alpha,\beta)=\alpha, \quad F_{2,0}^{\Bb}(\alpha,\beta)=-\beta \quad \text{and} \quad F_{s,1}^{\Bb}(\alpha,\beta)=F_{s,2}^{\Bb}(\alpha,\beta)=F_{s,3}^{\Bb}(\alpha,\beta)=0
$$
for each $(\alpha,\beta)\in D_\R$ and each $s=1,2$. In particular, $F_{s,\ell}^{\Bb}$ are real Nash functions on $D_\R$ for each $s=1,2$ and $\ell=0,1,2,3$. The complex components $F^{\ell}_{I,J}:D\to \C$ of $F$ with respect to the base $\Bb$ are the functions defined as
$$
F^{0}_{I,J}(z)=\overline{z} \quad \text{and} \quad F^{1}_{I,J}(z)=F^{2}_{I,J}(z)=F^{3}_{I,J}(z)=0
$$
for each $z\in D$. Observe that the function $F^{0}_{I,J}:D\to \C$ is an anti-holomorphic function. As being holomorphic does not depend on the choice of complex coordinates on $\HH\otimes_{\R}\C$ \cite[Rmk.3(2)]{gp}, then $F$ is not a holomorphic stem function, so not a stem-Nash function. \hfill$\sqbullet$
\end{example}

\subsection{Slice-Nash functions}

Let $A$ be either the algebra of quaternions $\HH$ or the algebra of octonions $\O$. Let $D\subset \C$ be a (non-empty) open subset and $F:D\to A\otimes_{\R}\C$ a stem function. Let $\Omega:=\Omega_D$ the circularised of $D$. We are ready to introduce our main definition. 

\begin{defn}[Slice-Nash function]\label{defslicenash}
We say that the slice function $f:=\mathcal{I}(F):\Omega\to A$ is a \textit{slice-Nash function on} $\Omega$ if the following hold:
\begin{itemize}
\item[\rm{(i)}] $f\in\mathcal{SR}_A(\Omega)$,
\item[\rm{(ii)}] the stem function $F:D\to A\otimes_{\R}\C$ is a stem-Nash function on $D$.
\end{itemize}
We denote by $\mathcal{SN}_A(\Omega)$ the set of all slice-Nash functions on $\Omega$.\hfill$\sqbullet$
\end{defn}

Observe that in the previous definition point (ii) implies point (i). In fact, if the stem function $F:D\to A\otimes_{\R}\C$ is a stem-Nash function, then, in particular, $F$ is a holomorphic stem function on $D$. Thus, the slice function $f:=\mathcal{I}(F):\Omega_D\to A$ is slice regular. We explicitly included point (i) in the definition, though not necessary, in order to underline the fact that $\mathcal{SN}_A(\Omega_D)\subset \mathcal{SR}_A(\Omega_D)$.

\begin{example}\label{esempislice}
(i) \textit{Let $f\in\HH[\q]$ be a slice polynomial. Then for any open circular set $\Omega\subset \HH$ the slice polynomial function $f:\Omega\to\HH, \, q\mapsto f(q)$ is a slice-Nash function on $\Omega$.}

We may write 
$$
f(\q):=\q^n\alpha_n+\ldots+\q\alpha_1+\alpha_0
$$
for some integer $n\geq 0$ and $\alpha_0,\ldots,\alpha_n\in\HH$. Let $D\subset \C$ be an open subset such that $\Omega=\Omega_D$. Then the stem function 
$$
F:D\to \HH\otimes_{\R}\C, \quad z\mapsto z^n\alpha_n+\ldots+z\alpha_1+\alpha_0:=\alpha_n\otimes z^n+\ldots+\alpha_1\otimes z+\alpha_0\otimes 1.
$$
is the stem function that induces $f$. Let $I,J\in\sph_\HH$ be imaginary units such that $J\neq \pm I$. As $\{1,I,J,IJ\}$ is a base of $\HH$ as a $\R$-vector space, then there exist unique $\alpha_k^0,\alpha_k^1,\alpha_k^2,\alpha_k^3\in\R$ for $k=0,\ldots,n$ such that
$
\alpha_k=\alpha_k^0+\alpha_k^1 I+\alpha_k^2J+\alpha_k^3IJ
$
for each $k=0,\ldots,n$. The stem function $F$ may be written, with respect to the base $\{1,I,J,IJ\}$, as
$$
F(z)=(z^n\alpha^0_n+\ldots+\alpha^0_0)+(z^n\alpha^1_n+\ldots+\alpha^1_0)I+(z^n\alpha^2_n+\ldots+\alpha^2_0)J+(z^n\alpha^3_n+\ldots+\alpha^3_0)IJ
$$
for each $z\in D$. As $\{1,I,J,IJ\}$ is a splitting base of $\HH$ associated to $I$, we deduce that $F$ is a stem-Nash function on $D$, so $f$ is a slice-Nash function on $\Omega=\Omega_D$. 

(ii) \textit{Analogously, if $f\in\O[\x]$ is a slice polynomial, then for any open circular set $\Omega\subset \O$ the slice polynomial function $\Omega\to\O, \, x\mapsto f(x)$ is a slice-Nash function on $\Omega$.}

(iii) \textit{Let $\Omega:= \HH\setminus\{q\in \R: q\leq0\}$ and let $J\in\sph_\HH$ be an imaginary unit. Let $f:\Omega\to \HH$ be the regular branch of $\sqrt{\q}J$ such that $f(1)=J$. Then $f\in\mathcal{SN}_\HH(\Omega)$.} 

Let $D:=\C\setminus\{z\in\R : z\leq 0\}$. A straightforward computation shows that the stem function $F$ that induces $f$ is the stem function
$$
F:D\to \HH\otimes_{\R}\C, \quad z\mapsto \sqrt{z}J:=J\otimes \sqrt{z},
$$
where $\sqrt{z}$ is the holomorphic branch of $\sqrt{\z}$ such that $\sqrt{1}=1$. Let $I\in\sph_\HH$ be any imaginary unit such that $J\neq\pm I$. As $\{1,I,J,IJ\}$ is a base of $\HH$ as a $\R$-vector space, the stem function $F$ may be written, with respect to this base, as
$
F(z)=\sqrt{z}J
$
for each $z\in D$. As $\{1,I,J,IJ\}$ is a splitting base of $\HH$ associated to $I$, we deduce that $F$ is a stem-Nash function on $D$, so $f$ is a slice-Nash function on $\Omega=\Omega_D$. 

(iv) \textit{Analogously, if $J\in\sph_\HH$ is an imaginary unit and $f:\Omega\to \HH$ the regular branch of $\sqrt{\x}J$ such that $f(1)=J$, where $\Omega:= \O\setminus\{x\in \R: x\leq0\}$, then $f\in\mathcal{SN}_\O(\Omega)$.}

(v) \textit{Let $\Omega\subset \HH$ be any open circular set and $I,J\in\sph_\HH$ any imaginary units such that $J\neq \pm I$. The slice regular function $f:\Omega\to \HH$ defined as $f(q)=(\cos q)I+(\sin q) J$ is not a slice-Nash function on $\Omega$.}

Let $D\subset \C$ be an open subset such that $\Omega=\Omega_D$ and $F:D\to \HH\otimes_\R\C$ the stem function such that $f=\mathcal{I}(F)$. With respect to the base $\{1,I,J,IJ\}$ of $\HH$ as a $\R$-vector space, we may write
$$
F(z)=(\cos z)I+(\sin z)J
$$
for each $z\in D$. As both $\cos z$ and $\sin z$ have infinitely many zeros on $\C$, then by Lemma \ref{lemmafiniti}, we have that $\cos z$ and $\sin z$ are not complex Nash functions on $\C$. In particular, as the restriction of a complex Nash function is still a complex Nash function, then $\cos z$ and $\sin z$ are not complex Nash functions on $D$. By Proposition \ref{indipendente}, we deduce that $F$ is not a stem-Nash function on $D$, so $f$ is not a slice-Nash function on $\Omega$. 

(vi) \textit{Analogously, for each pair of imaginary units $I,J\in\sph_\O$ such that $IJ=-JI$ and each open circular subset $\Omega\subset \O$ the slice regular function $f:\Omega\to \O$ defined as $f(x)=(\cos x)I+(\sin x) J$ is not a slice-Nash function on $\Omega$.} \hfill$\sqbullet$
\end{example}

\subsection{Characterisation of slice-Nash functions}\label{sectionchar}

In this section, we provide characterisations of slice-Nash functions in terms of their components given by the splitting lemma and in terms of their real components. These characterisations will allow us to show that slice-Nash functions share many of the properties of the real and complex Nash functions. 

Let $A$ be either the algebra of quaternions $\HH$ or the algebra of octonions $\O$. Let $d_A$ and $u_A$ be the integers introduced in \eqref{dim}. Let $\Omega\subset A$ be an open circular set and $f:\Omega\to A$ a slice regular function. For each $I\in\sph_A$, denote by $f_I$ the restriction $f|_{\Omega_I}$ of $f$ to $\Omega_I:=\Omega\cap \C_I$. Let $\Bb_I:=\{I_0:=1,I, I_1,II_1,\ldots,I_{u_A},II_{u_A}\}$ be a splitting base of $A$ associated to $I$. By the splitting lemma, there exist unique holomorphic functions  $f_1^{I,I_k},f^{I,I_k}_2:\Omega_I\to \C_I$ for $k=0,\ldots,u_A$ such that 
$$
f_I(z_I)=\sum_{k=0}^{u_A}(f_1^{I,I_k}(z_I)+f^{I,I_k}_2(z_I)I)I_k
$$
for each $z_I\in\Omega_I$. We start with the following characterisation for slice-Nash functions that relates a slice-Nash function $f$ to the components $f_1^{I,I_k},f^{I,I_k}_2$ of $f_I$ given by the splitting lemma.

\begin{thm}[Characterisation for slice-Nash functions I]\label{char}
The following are equivalent:
\begin{itemize}
\item[{\rm(i)}] $f\in \mathcal{SN}_A(\Omega)$.
\item[{\rm(ii})] There exist $I\in\sph_A$ and a splitting base $\Bb_I:=\{I_0:=1,I, I_1,II_1,\ldots,I_{u_A},II_{u_A}\}$ of $A$ associated to $I$ such that the functions  $f_1^{I,I_k},f^{I,I_k}_2:\Omega_I\to \C_I$ are $\C_I$-Nash functions on $\Omega_I$ for each $k=0,\ldots,u_A$. 
\item[{\rm(iii)}] For each $I\in\sph_A$ and each splitting base $\Bb_I:=\{I_0:=1,I, I_1,II_1,\ldots,I_{u_A},II_{u_A}\}$ of $A$ associated to $I$, the functions $f_1^{I,I_k},f^{I,I_k}_2:\Omega_I\to \C_I$ are $\C_I$-Nash functions on $\Omega_I$ for each $k=0,\ldots,u_A$. 
\end{itemize}
\end{thm}
\begin{proof}
We show the statement for $A=\HH$, as the case $A=\O$ is analogous. Let $D\subset \C$ be an open subset such that $\Omega=\Omega_D$ and $F:=F_1+\iota F_2:D\to \HH\otimes_{\R}\C$ the stem function such that $f=\mathcal{I}(F)$. Let $I,J\in\sph_\HH$ be imaginary units such that $J\neq\pm I$. As $\{1,I,J,IJ\}$ is a base of $\HH$ as a $\R$-vector space, then there exist (unique) real valued functions $F_{s,\ell}^{I,J}:D\to  \R$ for $s=1,2$ and $\ell=0,1,2,3$ such that 
$$
F_s=F_{s,0}^{I,J}+F_{s,1}^{I,J}I+F_{s,2}^{I,J}J+F_{s,3}^{I,J}IJ
$$
for $s=1,2$. For $\ell=0,1,2,3$ define
$$
F_{I,J}^{\ell}:=F_{1,\ell}^{I,J}+\iota F_{2,\ell}^{I,J}
$$
Let $\phi_I:\C\to \C_I$ be the $*$-isomorphism introduced in \eqref{*iso}. We have
\begin{align*}
f_I(z_I)&=F_1(z)+I F_2(z)\\
&=F_{1,0}^{I,J}(z)+F_{1,1}^{I,J}(z)I+F_{1,2}^{I,J}(z)J+F_{1,3}^{I,J}(z)IJ\\
&\hspace{9mm}+I(F_{2,0}^{I,J}(z)+F_{2,1}^{I,J}(z)I+F_{2,2}^{I,J}(z)J+F_{2,3}^{I,J}(z)IJ)\\
&=((F_{1,0}^{I,J}(z)+IF_{2,0}^{I,J}(z))+(F_{1,1}^{I,J}(z)+IF_{2,1}^{I,J}(z))I)\\
&\hspace{9mm}+((F_{1,2}^{I,J}(z)+IF_{2,2}^{I,J}(z))+(F_{1,3}^{I,J}(z)+IF_{2,3}^{I,J}(z))I)J
\end{align*}
for each $z_I\in\Omega_I$, where $z=\phi_I^{-1}(z_I)\in D$. We deduce that
\begin{align}\label{split}
f_1^{I,J}(z_I)&=(F_{1,0}^{I,J}(z)+IF_{2,0}^{I,J}(z))+(F_{1,1}^{I,J}(z)+IF_{2,1}^{I,J}(z))I\\\label{split2}
f_2^{I,J}(z_I)&=(F_{1,2}^{I,J}(z)+IF_{2,2}^{I,J}(z))+(F_{1,3}^{I,J}(z)+IF_{2,3}^{I,J}(z))I
\end{align}
for each $z_I\in\Omega_I$, where $z=\phi_I^{-1}(z_I)\in D$. By \eqref{split}, we deduce 
\begin{align*}
f_1^{I,J}(z_I)&=(F_{1,0}^{I,J}(z)+IF_{2,0}^{I,J}(z))+(F_{1,1}^{I,J}(z)+IF_{2,1}^{I,J}(z))I\\
&=\phi_I(F_{I,J}^0(z))+\phi_I( F_{I,J}^1(z))I=\phi_I(F_{I,J}^0(\phi_I^{-1}(\phi_I(z))))+\phi_I( F_{I,J}^1(\phi_I^{-1}(\phi_I(z))))I\\
&=\phi_I(F_{I,J}^0(\phi_I^{-1}(z_I)))+\phi_I( F_{I,J}^1(\phi_I^{-1}(z_I)))I
\end{align*}
for each $z_I\in \Omega_I$. Analogously, by \eqref{split2}, we deduce that 
$$
f_2^{I,J}(z_I)=\phi_I(F_{I,J}^2(\phi_I^{-1}(z_I)))+\phi_I( F_{I,J}^3(\phi_I^{-1}(z_I)))I
$$
for each $z_I\in\Omega_I$. 

As $\phi_I$ is an isomorphism that preserves the complex multiplication and the complex structures, we deduce that $G\in \mathcal{N}_{\C}(D)$ if and only if $\phi_I\circ G\circ \phi_I^{-1}\in\mathcal{N}_{\C_I}(\Omega_I)$. By Proposition \ref{indipendente}, if $f\in \mathcal{SN}_\HH(\Omega)$, then $F_{I,J}^{\ell}\in \mathcal{N}_{\C}(D)$ for each $I,J\in\sph_\HH$ such that $J\neq \pm I$ and each $\ell=0,1,2,3$. As $\mathcal{N}_{\C_I}(\Omega_I)$ is a ring with respect to the pointwise addition and multiplication of functions \cite[Cor.1.11]{t} and by the fact that $\phi_I\mathcal{N}_{\C}(D)\phi_I^{-1}=\mathcal{N}_{\C_I}(\Omega_I)$ for each $I\in\sph_\HH$, it follows that if $F_{I,J}^{\ell}\in \mathcal{N}_{\C}(D)$ for each $\ell=0,1,2,3$, then $f_1^{I,J},f^{I,J}_2\in\mathcal{N}_{\C_I}(\Omega_I)$. In particular, as $\{1,I,J,IJ\}$ is a splitting base of $\HH$ associated to $I$, then we have that (i) implies (iii). As (iii) trivially implies (ii), we only need to show that (ii) implies (i) to conclude the proof.

Let $I,J\in\sph_\HH$ be imaginary units such that $J\neq\pm I$. Assume that $f_1^{I,J},f^{I,J}_2:\Omega_I\to \C_I$ are $\C_I$-Nash functions on $\Omega_I$. Recall that $F_1(\overline{z})=F_1(z)$ and $F_2(\overline{z})=-F_2(z)$ for each $z\in D$. We may write \eqref{split} and \eqref{split2} as
\begin{align}\label{eq11}
\begin{split}
f_1^{I,J}(z_I)&=(F_{1,0}^{I,J}(z)-F_{2,1}^{I,J}(z))+(F_{1,1}^{I,J}(z)+F_{2,0}^{I,J}(z))I\\
f_2^{I,J}(z_I)&=(F_{1,2}^{I,J}(z)-F_{2,3}^{I,J}(z))+(F_{1,3}^{I,J}(z)+F_{2,2}^{I,J}(z))I
\end{split}
\end{align}
for each $z_I\in\C_I$, where $z=\phi_I^{-1}(z_I)\in D$. As $\Omega_I$ is symmetric with respect to the real axis, we deduce
\begin{align}\label{eq22}
\begin{split}
(f_1^{I,J})^c(z_I)&:=\overline{f_1^{I,J}(\overline{z_I})}=(F_{1,0}^{I,J}(z)+F_{2,1}^{I,J}(z))+(-F_{1,1}^{I,J}(z)+F_{2,0}^{I,J}(z))I\\
(f_2^{I,J})^c(z_I)&:=\overline{f_2^{I,J}(\overline{z_I})}=(F_{1,2}^{I,J}(z)+F_{2,3}^{I,J}(z))+(-F_{1,3}^{I,J}(z)+F_{2,2}^{I,J}(z))I
\end{split}
\end{align}
for each $z_I\in\Omega_I$, where $z=\phi_I^{-1}(z_I)\in D$. A straightforward calculation shows that the functions $(f_1^{I,J})^c$ and $(f_2^{I,J})^c$ are $\C_I$-Nash functions on $\Omega_I$ (see also \cite[Rmk.2.3]{ca}). Using again the $*$-isomorphism $\phi_I$ and the fact that $\phi_I^{-1}\mathcal{N}_{\C_I}(\Omega_I)\phi_I=\mathcal{N}_{\C}(D)$, by \eqref{eq11} and \eqref{eq22}, we have that the following functions
\begin{align}\label{sumdiff}
\begin{split}
(F_{1,0}^{I,J}-F_{2,1}^{I,J})+\iota (F_{1,1}^{I,J}+F_{2,0}^{I,J}),& \quad (F_{1,2}^{I,J}-F_{2,3}^{I,J})+\iota (F_{1,3}^{I,J}+F_{2,2}^{I,J})\\
(F_{1,0}^{I,J}+F_{2,1}^{I,J})+\iota (-F_{1,1}^{I,J}+F_{2,0}^{I,J}),& \quad (F_{1,2}^{I,J}+F_{2,3}^{I,J})+\iota (-F_{1,3}^{I,J}+F_{2,2}^{I,J})
\end{split}
\end{align}
are complex Nash functions on $D$. As $\mathcal{N}_\C(D)$ is a ring \cite[Cor.1.11]{t}, taking suitable linear combinations of the functions in \eqref{sumdiff}, we deduce that the functions $F_{I,J}^{\ell}=F_{1,\ell}^{I,J}+\iota F_{2,\ell}^{I,J}$ are complex Nash functions on $D$ for each $\ell=0,1,2,3$. We conclude that $f$ is a slice-Nash function on $\Omega$, as required.
\end{proof}

Let $\Bb:=\{\e_0,\ldots,\e_{d_A}\}$ be any base of $A$ as a $\R$-vector space (here we are not requiring that $1\in\Bb$). Define
\begin{equation}\label{omegaB}
\Omega_{\Bb}:=\{(x_0,\ldots,x_{d_A})\in \R^{d_A+1} : x_0\e_0+\ldots+x_{d_A}\e_{d_A}\in \Omega\}\subset \R^{d_A+1}.
\end{equation}
Let $f:\Omega\to A$ be a slice function. Then, there exist unique real valued functions $f_\ell^{\Bb}:\Omega\to \R$ for $\ell=0,\ldots,d_A$ such that 
\begin{equation}\label{effereale}
f(x)=f_0^{\Bb}(x)\e_0+\ldots+f_{d_A}^{\Bb}(x)\e_{d_A}
\end{equation}
for each $x\in \Omega$. After the identification $A\equiv\R^{d_A+1}$ induced by the base $\Bb$, the functions $f_\ell^{\Bb}$ can be regarded as functions of real variables
$$
f^{\Bb,\R}_{\ell}:\Omega_{\Bb}\to \R, \quad (x_0,\ldots,x_{d_A})\mapsto f_\ell^{\Bb}(x_0\e_0+\ldots+x_{d_A}\e_{d_A}).
$$
In order to lighten the exposition, up to a light abuse of notation, we will write $f_\ell^{\Bb}$ instead of $f_\ell^{\Bb,\R}$. This will create no confusion, as the situation will be always clear by the context. The next characterisation relates a slice-Nash function $f$ to its real components $f_{\ell}^{\Bb}$.

\begin{thm}[Characterisation of slice-Nash functions II]\label{char2}
The followings are equivalent:
\begin{itemize}
\item[{\rm(i)}] $f\in\mathcal{SN}_A(\Omega)$.
\item[{\rm(ii})] $f$ is slice regular and there exists a base $\Bb$ of $A$ as a $\R$-vector space such that the functions $f_\ell^{\Bb}:\Omega_{\Bb}\to \R$ are real Nash functions for each $\ell=0,\ldots,d_A$.
\item[{\rm(iii)}] $f$ is slice regular and, for each base $\Bb$ of $A$ as a $\R$-vector space, the functions $f_\ell^{\Bb}:\Omega_{\Bb}\to \R$ are real Nash functions for each $\ell=0,\ldots,d_A$.
\end{itemize}
\end{thm}
\begin{proof}
Observe that (iii) trivially implies (ii). Moreover, we omit the proof of the implication (ii)$\rightarrow$(iii), as the argument is essentially the same as the one of the proof of Proposition \ref{indipendente2}. In particular, we only need to show that (i) is equivalent to (iii). We show the statement for $A=\HH$, as the case $A=\O$ is analogous. Let $D\subset \C$ be an open subset such that $\Omega=\Omega_D$ and $F:=F_1+\iota F_2:D\to \HH\otimes_{\R}\C$ a stem function such that $f=\mathcal{I}(F)$. 

\noindent{\sc (i) implies (iii).} Each $q\in \HH\setminus\R$ can be written uniquely as $q=\alpha+\beta I$, where $\alpha,\beta\in \R$ with $\beta>0$ and $I\in\sph_\HH$. In particular, $\alpha=\Re(q):=\tfrac{q+q^c}{2}$ and $\beta I=\Im(q):=\tfrac{q-q^c}{2}$. A straightforward computation shows that
$$
\beta=\sqrt{-\Im^2(q)} \quad \text{and} \quad I=\frac{\Im(q)}{\sqrt{-\Im^2(q)}}.
$$
Observe that, as $q\not\in \R$, then $\Im^2(q)\in\R$ and $\Im^2(q)<0$. Thus, each $q\in \HH\setminus\R$ can be written (uniquely) as
\begin{equation}\label{J}
q=\Re(q)+\sqrt{-\Im^2(q)}\frac{\Im(q)}{\sqrt{-\Im^2(q)}}.
\end{equation}

Let $D_\R:=\{(\alpha,\beta)\in\R^2 : \alpha+\iota \beta\in D\}$. We regard the stem function $F$ as the map of real variables $F=(F_1,F_2):D_{\R}\to \HH\oplus \HH$. By \eqref{J}, we deduce that 
\begin{equation}\label{eqreal}
f(q)=F_{1}\Big(\Re(q),\sqrt{-\Im^2(q)}\Big)+\frac{\Im(q)}{\sqrt{-\Im^2(q)}}F_{2}\Big(\Re(q),\sqrt{-\Im^2(q)}\Big)
\end{equation}
for each $q\in \Omega\setminus\R$. Let $\Bb:=\{\e_0,\e_1,\e_2,\e_3\}$ be any base of $\HH$ as a $\R$-vector space. Consider the identification $\HH\equiv \R^4$ induced by the (real) base $\Bb$. With this identification, the product of $\HH$ is a (real) bilinear map and the conjugation $q\mapsto q^c$ a (real) linear map. In particular, the real components $(\alpha_0,\alpha_1,\alpha_2,\alpha_3)$ of the map
$$
(\Omega\setminus\R)_{\Bb}\to \R^4, \quad (q_0,q_1,q_2,q_3)\mapsto \frac{\Im(q_0\e_0+q_1\e_1+q_2\e_2+q_3\e_3)}{\sqrt{-\Im^2(q_0\e_0+q_1\e_1+q_2\e_2+q_3\e_3)}}
$$
with respect to the base $\Bb$ are real Nash functions on $(\Omega\setminus\R)_{\Bb}$. 

Let $F_{1,\ell}^{\Bb},F_{2,\ell}^{\Bb}:D_\R\to \R$ for $\ell=0,1,2,3$ be the real components of $F_1$ and $F_2$ with respect to the base $\Bb$. We may regard the functions $F_{s,\ell}^{\Bb}$ for $s=1,2$ and $\ell=0,1,2,3$, as functions
$$
\widetilde{F}_{s,\ell}^{\Bb}:\Omega_{\Bb}\to \R, \quad (q_0,q_1,q_2,q_3)\mapsto \widetilde{F}_{s,\ell}^{\Bb}(q_0,q_1,q_2,q_3):=F_{s,\ell}^{\Bb}\Big(\Re(q),\sqrt{-\Im^2(q)}\Big),
$$
where $q:=q_0\e_0+q_1\e_1+q_2\e_2+q_3\e_3\in\Omega_{\Bb}$. Clearly, if $F_{s,\ell}^{\Bb}$ is a real Nash function of $D_{\R}\setminus\R$, then $\widetilde{F}_{s,\ell}^{\Bb}$ is a real Nash function on $\Omega_\Bb\setminus\R$, because both $\Re(q)$ and $\sqrt{-\Im^2(q)}$ are Nash functions in the variables $q_0,q_1,q_2$ and $q_3$.

As the product of $\HH$ is a (real) bilinear map, by \eqref{eqreal}, we deduce that for each $\ell=0,1,2,3$ there exist unique $\beta^\ell_{r,k}\in\R$ for $r,k\in\{0,1,2,3\}$ such that
\begin{equation}\label{realchareq}
f_\ell^{\Bb}(q_0,q_1,q_2,q_3)=\widetilde{F}_{1,\ell}^{\Bb}(q_0,q_1,q_2,q_3)+\sum_{r,k=0}^3\beta^\ell_{r,k}\alpha_r(q_0,q_1,q_2,q_3) \widetilde{F}_{2,\ell}^{\Bb}(q_0,q_1,q_2,q_3)
\end{equation}
for each $(q_0,q_1,q_2,q_3)\in(\Omega\setminus\R)_\Bb$. The numbers $\beta^\ell_{r,k}\in\R$ are the only real numbers that satisfy the equality
$
\e_r\cdot\e_k=\beta^0_{r,k}\e_0+\beta^1_{r,k}\e_1+\beta^2_{r,k}\e_2+\beta^3_{r,k}\e_3
$
for $r,k\in\{0,1,2,3\}$. By Proposition \ref{indipendente2}, $F_{1,\ell}^{\Bb}$ and $F_{2,\ell}^{\Bb}$ are real Nash functions on $D_{\R}$ for each $\ell=0,1,2,3$, so $\widetilde{F}_{1,\ell}^{\Bb}$ and $\widetilde{F}_{2,\ell}^{\Bb}$ are real Nash functions on $(\Omega\setminus\R)_{\Bb}$ for each $\ell=0,1,2,3$. By \eqref{realchareq} and \cite[Cor.8.1.6]{bcr}, we deduce that $f_\ell^{\Bb}$ is a real Nash function on $(\Omega\setminus\R)_{\Bb}$ for each $\ell=0,1,2,3$. 

Observe that, as $F_1$ and $F_2$ are real analytic maps on $D_{\R}$, then, by \cite[Prop.7(iii)]{gp}, $f$ is a real analytic map on $\Omega_{\Bb}$. Let $\Omega'$ be a connected component of $\Omega_{\Bb}$ such that $\Omega'\cap \R\neq \varnothing$ and $\ell\in\{0,1,2,3\}$. As $f_\ell^{\Bb}$ is a real Nash function on $\Omega'\setminus\R$, then there exists a non-zero polynomial $P_\ell\in\R[\x,\t]:=\R[\x_1,\ldots,\x_d,\t]$ such that
$
P_\ell(q,f_\ell^{\Bb}(q))=0
$
for each $q\in\Omega'\setminus\R$. We deduce that $P_\ell(q,f_\ell^{\Bb}(q))=0$ for each $q\in\Omega'$. In particular, $f_\ell^{\Bb}$ is a real Nash function on $\Omega'$. As being a Nash function is a local property on the connected components of $\Omega_{\Bb}$, we conclude that $f_\ell^{\Bb}$ is a real Nash function on $\Omega_{\Bb}$, as required. 

\noindent{\sc (iii) implies (i).} Let $I,J\in\sph_\HH$ be imaginary units such that $J\neq \pm I$ and consider the splitting base $\Bb:=\{1,I,J,IJ\}$. Let $f_{\ell}^{\Bb}:\Omega\to \R$ for $\ell=0,1,2,3$ be the unique real valued functions such that
$$
f(q)=f_{0}^{\Bb}(q)+f_{1}^{\Bb}(q)I+f_{2}^{\Bb}(q)J+f_{3}^{\Bb}(q)IJ
$$
for each $q\in\Omega$. In particular,
$$
f_I(z_I)=f_{0}^{\Bb}(z_I)+f_{1}^{\Bb}(z_I)I+f_{2}^{\Bb}(z_I)J+f_{3}^{\Bb}(z_I)IJ
$$
for each $z_I\in\Omega_I:=\Omega\cap \C_I$, where $f_I:=f|_{\Omega_I}$. As $f$ is slice regular, by the splitting lemma, the functions 
$$
f_1^{I,J}:=f_{0}^{\Bb}|_{\Omega_I}+f_{1}^{\Bb}|_{\Omega_I}I:\Omega_I\to \C_I, \quad 
f_2^{I,J}:=f_{2}^{\Bb}|_{\Omega_I}+f_{3}^{\Bb}|_{\Omega_I}I:\Omega_I\to \C_I
$$
are holomorphic. As $f_{\ell}^{\Bb}$, seen as a function $f_{\ell}^{\Bb}:\Omega_{\Bb}\to \R$, is a real Nash function on $\Omega_{\Bb}$ for each $\ell=0,1,2,3$, then its restriction $f_{\ell}^{\Bb}|_{(\Omega_I)_{\Bb}}$ is a real Nash function on $(\Omega_{I})_\Bb=\Omega_{\Bb}\cap \C_I$ for each $\ell=0,1,2,3$. Moreover, as
\begin{align*}
&(\Omega_I)_{\Bb}=\{(q_0,q_1,q_2,q_3)\in \R^4: q_2=q_3=0 \text{ and }q_0+q_1I\in \Omega_I\}\\
&(\Omega_I)_{\R}=\{(\alpha,\beta)\in\R^2 : \alpha+\beta I\in \Omega_I\},
\end{align*}
then the function $f_{\ell}^{\Bb}|_{\Omega_I}$, regarded as a function of real variables $(\Omega_I)_\R\to \R$, is a real Nash function on $(\Omega_I)_\R$ for each $\ell=0,1,2,3$. In particular, as the functions $f_1^{I,J}$ and $f_2^{I,J}$ are holomorphic on $\Omega_I$, by \cite[Prop.2.4]{ca}, we deduce that the functions $f_1^{I,J}$ and $f_2^{I,J}$ are $\C_I$-Nash functions on $\Omega_I$. By Theorem \ref{char}, we conclude that $f\in\mathcal{SN}_\HH(\Omega)$, as desired.
\end{proof} 

In points (ii) and (iii) of the previous theorem, the assumption that $f$ is slice regular cannot be dropped (compare this with Example \ref{hol2}). For instance, the function 
$$
f:A\to A, \quad x\mapsto x^c
$$ 
is a real Nash function (when expressed in real components with respect to any base $\Bb$ of $A$ as a $\R$-vector space) but clearly $f$ is not slice regular, so it is not a slice-Nash function.

\subsection{Slice derivatives}

In this section, we show that if a slice function $f$ is slice-Nash, then its slice derivatives are also slice-Nash. 

\begin{prop}\label{slicederProp}
Let $A$ be either the algebra of quaternions $\HH$ or the algebra of octonions $\O$. Let $\Omega\subset A$ be an open circular set and $f\in\mathcal{SN}_A(\Omega)$. Then, for each integer $n\geq 1$, the $n$\textsuperscript{th} slice derivative $\tfrac{\partial^n f}{\partial \x^n}:\Omega\to A$ of $f$ is a slice-Nash function on $\Omega$. 
\end{prop}
\begin{proof}
We show the statement for $A=\HH$, as the case $A=\O$ is similar. As $\tfrac{\partial^n f}{\partial \x^n}=\tfrac{\partial }{\partial \x}\big(\tfrac{\partial^{n-1} f}{\partial \x^{n-1}}\big)$ for each integer $n\geq 2$, it is enough to show the statement for $n=1$. Let $D\subset \C$ be an open subset such that $\Omega=\Omega_D$ and $F:D\to \HH\otimes_{\R}\C$ the stem function such that $f=\mathcal{I}(F)$. Let $I,J\in\sph_\HH$ be imaginary units such that $J\neq \pm I$. Let $F_{I,J}^{\ell}:D\to \C$ for $\ell=0,1,2,3$ be the (unique) complex functions such that 
$
F=F_{I,J}^{0}+F_{I,J}^{1}I+F_{I,J}^{2}J+F_{I,J}^{3}IJ.
$
By Proposition \ref{indipendente}, $F_{I,J}^{\ell}\in\mathcal{N}_\C(D)$ for each $\ell=0,1,2,3$. By \cite[Cor.1.12]{t}, we deduce that $\tfrac{\partial F_{I,J}^{\ell}}{\partial \z}\in\mathcal{N}_\C(D)$ for each $\ell=0,1,2,3$. 
Thus, $\tfrac{\partial F}{\partial \z}$ is a stem-Nash function on $D$. We conclude that 
$$
\frac{\partial f}{\partial \x}=\mathcal{I}\Big(\frac{\partial F}{\partial \z}\Big)\in \mathcal{SN}_\HH(\Omega),
$$
as required.
\end{proof}

\subsection{Algebraic structure on $\mathcal{SN}_A(\Omega)$}

Let $A$ be either the algebra of quaternions $\HH$ or the algebra of octonions $\O$. Let $\Omega\subset A$ be an open circular set. In this section, we show that $\mathcal{SN}_A(\Omega)$ is a $*$-subalgebra of the $*$-algebra $\mathcal{SR}_A(\Omega)$ of slice-regular functions on $\Omega$ endowed with the sum, the slice product and the conjugation $f\mapsto f^c$.

Let $D\subset \C$ be an open subset such that $\Omega=\Omega_D$. We start by showing that the complex $*$-algebra structure of $A\otimes_{\R}\C$ induces naturally a complex $*$-algebra structure on the set of all stem-Nash functions on $D$.

\begin{prop}\label{Calgebra}
The set of all stem-Nash functions on $D$ is a complex $*$-algebra (with respect to the operations induced by the complex $*$-algebra structure of $A\otimes_{\R}\C$).
\end{prop}
\begin{proof}
We show the statement for $A=\O$, as the case $A=\HH$ is similar but easier. Let $D\subset \C$ be an open subset and $F,G:D\to \O\otimes_\R\C$ stem-Nash functions. Let $I\in\sph_\O$ be an imaginary unit and $\Bb_I:=\{1,I,I_1,II_1,I_2,II_2,I_3,II_3\}$ a splitting base of $\O$ associated to $I$. In order to lighten the exposition, we denote the elements of the splitting base $\Bb_I$ as $J_0:=1,J_1,\ldots,J_7$. Let $F^{\ell},G^{\ell}:D\to \C$ be the unique complex valued functions such that
$$
F=F^0J_0+\ldots+F^7J_7\quad \text{and} \quad G=G^0J_0+\ldots+G^7J_7.
$$
By Proposition \ref{indipendente}, we may assume that $F^{\ell},G^{\ell}\in \mathcal{N}_{\C}(D)$ for each $\ell=0,\ldots,7$.

We may write the stem function $F+G$ as
$$
F+G=(F^0+G^0)J_0+\ldots+(F^7+G^7)J_7.
$$
As $\{J_0,\ldots,J_7\}$ is a base of $\O$ as a $\R$-vector space and the product of $\O$ a real bilinear map, then for each $r,s,k\in\{0,\ldots,7\}$ there exist unique $\alpha^k_{r,s}\in\R$ such that
$
J_rJ_s=\alpha_{r,s}^0J_0+\ldots+\alpha_{r,s}^7J_7.
$
In particular, we may write $F\cdot G$ as
$$
F\cdot G=\Big(\sum_{r=0}^7 F^r J_r\Big)\Big(\sum_{s=0}^7 G^s J_s\Big)=\sum_{r,s,k=0}^7\alpha_{r,s}^k F^rG^s J_k.
$$
By \cite[Cor.1.11]{t}, we deduce that all the components of $F+G$ and $G\cdot F$ with respect to the base $\{J_0,\ldots,J_7\}$ are complex Nash functions on $D$. Thus, $F+G$ and $F\cdot G$ are stem-Nash functions on $D$. As $\alpha F=\alpha\cdot F$ for each $\alpha\in \C$, we deduce that $\alpha F$ is a stem-Nash function for each $\alpha\in\C$. In particular, the set of all stem-Nash functions on $D$ is a (complex) algebra. Moreover, as
$$
F^c=F_0-F_1J_1-\ldots-F_7J_7,
$$
then, if $F$ is a stem-Nash function on $D$, $F^c$ is also a stem-Nash function on $D$. In particular, the algebra of all stem-Nash functions on $D$ endowed with the $*$-involution $F\mapsto F^c$ is a $*$-algebra, as required.
\end{proof}

By Proposition \ref{Calgebra}, we deduce straightforwardly the following:

\begin{prop}\label{ringNash}
If $f,g\in\mathcal{SN}_A(\Omega)$, then $f+g,f\cdot g\in\mathcal{SN}_A(\Omega)$.
\end{prop}
\begin{proof}
Let $D\subset \C$ be an open subset such that $\Omega=\Omega_D$ and $F,G:D\to A\otimes_\R\C$ stem functions such that $f=\mathcal{I}(F)$ and $g=\mathcal{I}(G)$. As $f+g=\mathcal{I}(F+G)$ and $f\cdot g=\mathcal{I}(F\cdot G)$, by Proposition \ref{Calgebra}, we deduce that $f+g,f\cdot g\in\mathcal{SN}_A(\Omega)$, as required.
\end{proof}

In the following remark, we show that the conjugate, the normal function, and the reciprocal of a slice-Nash function are also slice-Nash functions.

\begin{remark}\label{siminv}
\textit{Assume that $f\in\mathcal{SN}_A(\Omega)$. Then
\begin{itemize}
\item[{\rm(i)}] $f^c\in \mathcal{SN}_A(\Omega)$.
\item[{\rm(ii)}] $N(f)\in\mathcal{SN}_A(\Omega)$.
\item[{\rm(iii)}] If $\Omega\setminus V(N(f))\neq\varnothing$, then $f^{-\bullet}\in \mathcal{SN}_A(\Omega\setminus V(N(f)))$.
\end{itemize}}
\end{remark}
Let $D\subset \C$ be an open subset such that $\Omega=\Omega_D$ and $F:=F_1+\iota F_2:D\to A\otimes_\R \C$ the stem function such that $f=\mathcal{I}(F)$. 

(i) It follows straightforwardly from the fact that $f^c:=\mathcal{I}(F^c)$. 

(ii) By point (i), $f^c\in \mathcal{SN}_A(\Omega)$. By Proposition \ref{ringNash}, we conclude that $N(f)=f\cdot f^c\in \mathcal{SN}_A(\Omega)$. 

(iii) Recall that $f^{-\bullet}(x):=(N(f)(x))^{-1}f^c(x)$ for each $x\in\Omega\setminus V(N(f))$. By point (i) and Proposition \ref{ringNash}, we are reduced to show $(N(f))^{-1}\in \mathcal{SN}_{A}(\Omega\setminus V(N(f)))$. By definition, $N(f)=\mathcal{I}(F\cdot F^c)$. As $F\cdot F^c$ is real valued, then the function $(F\cdot F^c)^{-1}:=\tfrac{1}{F\cdot F^c}$ is a well defined  stem-Nash function on $D\setminus V(F\cdot F^c)$. By the fact that $(F\cdot F^c)\cdot (F\cdot F^c)^{-1}=1$ on $D\setminus V(F\cdot F^c)$ and as $\mathcal{I}$ is an isomorphism of $*$-algebras, we deduce that $(N(f))^{-1}=\mathcal{I}((F\cdot F^c)^{-1})$. Observe that the circularised of $D\setminus V(F\cdot F^c)$ is $\Omega\setminus V(N(f))$. We conclude that $(N(f))^{-1}$ is a slice-Nash function on $\Omega\setminus V(N(f))$, as required. \hfill$\sqbullet$

We deduce the following:

\begin{example}\label{remarkprodotto}
(i) Let $g,h\in A[\x]$ be slice polynomials such that $h\neq 0$. Then the `slice rational' function $h^{-\bullet}\cdot g:A\setminus V(N(h))\to A$ is a slice-Nash function on $A\setminus V(N(h))$.

(ii) More generally, if $g,h:\Omega\to A$ are slice-Nash functions defined on an open circular set $\Omega\subset A$ such that $\Omega\setminus V(N(h))\neq\varnothing$, then the function $h^{-\bullet}\cdot g:\Omega\setminus V(N(h))\to A$ is a slice-Nash function on $\Omega\setminus V(N(h))$. \hfill$\sqbullet$
\end{example} 

By Proposition \ref{ringNash} and Remark \ref{siminv}, we deduce the following result. Compare this result with Theorem \ref{sliceregularstructure}.

\begin{thm}\label{*Nash}
The set $\mathcal{SN}_A(\Omega)$ of all slice-Nash functions on $\Omega$ is an alternative $*$-subalgebra of the alternative $*$-algebra $\mathcal{SR}_A(\Omega)$ of slice regular functions on $\Omega$ endowed with $+$, $\cdot$, $\cdot^c$. If $A=\HH$, then $\mathcal{SN}_\HH(\Omega)$ is associative. Moreover,
\begin{itemize}
\item[\rm{(i)}] if $\Omega$ is a symmetric slice domain, then $\mathcal{SN}_A(\Omega)$ is a division algebra,
\item[\rm{(ii)}] if $\Omega$ is a product domain, then $\mathcal{SN}_A(\Omega)$ includes some element $f\not\equiv 0$ with $N(f)\equiv 0$. However, every element $f$ with $N(f)\not\equiv 0$ admits a multiplicative inverse in the algebra $\mathcal{SN}_A(\Omega\setminus V(N(f)))$.
\end{itemize}
\end{thm}
\begin{proof}
The fact that $\mathcal{SN}_A(\Omega)$ is a $*$-subalgebra of $\mathcal{SR}_A(\Omega)$ follows straightforwardly by Proposition \ref{ringNash}. Point (i) and the second part of point (ii) follow by Theorem \ref{sliceregularstructure} and Remark \ref{siminv}(iii), while the fist part of point (ii) follows by Example \ref{fettazero} after noticing that the involved function $f$ is slice-Nash.
\end{proof}

\section{Finiteness properties of slice-Nash functions}\label{SliceNashProp}

In this section, we show several finiteness properties of slice-Nash functions, finding many analogies with the properties of the classical real and complex Nash functions. 

\subsection{Global slice-Nash functions}

It is known that there are `very few' entire complex Nash functions $F\in\mathcal{N}_{\C}(\C^n)$. Namely, if $F:\C^n\to\C$ is a complex Nash function, then $F$ is a polynomial function \cite[Thm.1.3]{t}. Let $A$ be either the algebra of quaternions $\HH$ or the algebra of octonions $\O$. We make use of Theorem \ref{char} to derive from this result its slice counterpart, namely, to show that $\mathcal{SN}_A(A)=A[x]$, where $A[x]$ denotes the set of all slice polynomial functions on $A$. 

\begin{thm}[Global slice-Nash functions]\label{global}
If $f\in\mathcal{SN}_A(A)$, then $f\in A[x]$.
\end{thm}
\begin{proof}
We show the statement for $A=\O$, as the case $A=\HH$ is similar but easier. Let $I\in\sph_{\O}$ be an imaginary unit and $\{I_0:=1,I,I_1,II_1,I_2,II_2,I_3,II_3\}$ a splitting base of $\O$ associated to $I$. By the splitting lemma, there exist unique holomorphic functions $f_k:\C_I\to \C_I$ for $k=0,1,2,3$ such that
$
f_I(z_I)=f_0(z_I)I_0+f_1(z_I)I_1+f_2(z_I)I_2+f_3(z_I)I_3
$
for each $z_I\in\C_I$. By Theorem \ref{char}, as $f$ is a slice-Nash function on $\O$, then the functions $f_{k}$ are complex Nash functions on $\C_I$ for each $k=0,1,2,3$. By \cite[Thm.1.3]{t}, we have that $f_k$ is a polynomial function on $\C_I$ for each $k=0,1,2,3$. That is, for each $k=0,1,2,3$, there exist an integer $n_k\geq 0$ and $\alpha_{k,0},\ldots,\alpha_{k,n_k}\in \C_I$ such that
$
f_k(z_I)=z_I^{n_k}\alpha_{k,n_k}+\ldots+z_I\alpha_{k,1}+\alpha_{k,0}
$
for each $z_I\in \C_I$. Define $N:=\max\{n_0,\ldots,n_3\}$. By Artin's theorem, we have
$$
(z_I^r\alpha_{k,r})I_k=z_I^r(\alpha_{k,r}I_k)
$$
for each $k=0,1,2,3$, each $r=0,\ldots,n_k$ and each $z_I\in \C_I$. For each $k=0,1,2,3$ and for each $n_k<r\leq N$, set $\alpha_{k,r}:=0$. We deduce
$$
f_I(z_I)=\sum_{k=0}^3 f_k(z_I)I_k=\sum_{k=0}^3 \Big(\sum_{r=0}^{n_k}z_I^r\alpha_{k,r}\Big)I_k=\sum_{k=0}^3 \sum_{r=0}^{n_k}z_I^r(\alpha_{k,r}I_k)=\sum_{r=0}^{N}z_I^r\Big(\sum_{k=0}^3 \alpha_{k,r}I_k\Big)
$$
for each $z_I\in \C_I$. Consider the slice polynomial
$$
P(\x):=\sum_{r=0}^{N}\x^r\Big(\sum_{k=0}^3 \alpha_{k,r}I_k\Big)\in \O[\x]
$$
Let $P:\O\to \O$ be the polynomial function associated to $P(\x)$. As the $f=P$ on the slice $\C_I$, using the representation formula, we deduce that $f(x)=P(x)$ for each $x\in\O$. We conclude that $f\in\O[x]$ is a slice polynomial function, as required.
\end{proof}

In the following example, we show that the previous theorem is no longer true if we consider slice-Nash functions defined on proper open circular subsets $\Omega\subsetneq A$. The example is for $A=\O$, but it can easily be adapted to the case $A=\HH$ as $\HH$ is a $*$-subalgebra of $\O$.

\begin{example}
Let $I,J\in\sph_\O$ be imaginary units such that $J\neq\pm I$. Consider the symmetric slice domain $\Omega:=\O\setminus\{x\in \R : x\leq 0\}$ and the slice regular function 
$$
f:\Omega\to \O, \quad x\mapsto \sqrt{x}I+x^2+J,
$$
where $\sqrt{x}$ is the regular branch of $\sqrt{\x}$ such that $\sqrt{1}=1$. Let $D:= \C \setminus \{z\in \R : z\leq 0\}$. Clearly $\Omega=\Omega_D$. A straightforward computation shows that the stem function 
$$
F:D\to \O\otimes_{\R}\C, \quad z\mapsto \sqrt{z}I+z^2+J:=I\otimes \sqrt{z}+1\otimes z^2+J\otimes 1,
$$
satisfies $f=\mathcal{I}(F)$, where $\sqrt{z}$ is the holomorphic branch of $\sqrt{\z}$ such that $\sqrt{1}=1$. Let $L\in \sph_\O$ be an imaginary unit such that $\Bb_I:=\{1,I,J,IJ,L,IL,JL,I(JL)\}$ is a splitting base of $\O$ associated to $I$. The complex coordinates $(F^0_{I,J,L},\ldots,F^7_{I,J,L})$ of $F$ with respect to the base $\Bb_I$ are the functions defined as $F^0_{I,J,L}(z)=z^2,  F^1_{I,J,L}(z)=\sqrt{z}, \quad F^2_{I,J,L}(z)=1$ and $F^3_{I,J,L}(z)=\ldots=F^7_{I,J,L}(z)=0$ for $z\in D$. We have that $f\in\mathcal{SN}_\O(\Omega)$, but $f$ is not a polynomial function. \hfill$\sqbullet$
\end{example}

\subsection{Zero sets of slice-Nash functions}

In this section, we show that if a slice function $f$ is slice-Nash on a circular domain $\Omega$ and $N(f)$ is not identically zero, then its zero set $V(f)$ is a \textit{finite} union of isolated zeros and isolated spherical zeros. Let $A$ be either the algebra of quaternions $\HH$ or the algebra of octonions $\O$.

\begin{prop}[Finiteness of zeros for slice-Nash functions]\label{finiti}
Let $\Omega\subset A$ be a circular domain and $f:\Omega\to A$ a slice-Nash function such that $N(f)$ is not identically zero. Then its zero set $V(f)$ is a finite union of isolated zeros and isolated spherical zeros.
\end{prop}
\begin{proof}
We show the statement for $A=\HH$, as the case $A=\O$ is similar. As $N(f)$ is not identically zero, then by Theorem \ref{zeri2}, the zero set $V(f)$ of $f$ consists only of isolated zeros and isolated spherical zeros. Moreover, as $N(N(f))=N(f)^2$ is not identically zero, using again Theorem \ref{zeri2}, it follows that the zero set $V(N(f))$ of $N(f)$ also only consists of isolated zeros and isolated spherical zeros. By Theorem \ref{zeri1}, the zero set $V(N(f))$ of $N(f)$ is the circularised of the zero set $V(f)$ of $f$, so, in particular, $V(f)\subset V(N(f))$. Moreover, $V(N(f))$ consists only of isolated zeros on the real axes and isolated spherical zeros. In particular, for each $I\in \sph_\HH$ we have that $V(N(f))$ is the circularised of the zero set of the restriction of $N(f)$ to $\C_I$ and the restriction of $N(f)$ to $\C_I$ has only isolated zeros. Let $I\in\sph_\HH$ be an imaginary unit, $\Omega_I:=\Omega\cap \C_I$ and $f_I:=f|_{\Omega_I}$. As $V(f)\subset V(N(f))$ and for each spherical zero $\sph_x$ of $N(f)$ either $\sph_x$ is a spherical zero of $f$ or $\sph_x$ contains at most one isolated zero of $f$, then, up to replace $f$ with $N(f)$ if necessary, we may assume that 
\begin{itemize}
\item $V(f)$ is the circularised of $V(f_I)$,
\item $V(f_I)$ is a discrete subset of $\Omega_I$,
\end{itemize}
and we may reduce to show: \textit{$f_I$ has only finitely many isolated zeros on $\Omega_I$.}

Let $J\in\sph_\HH$ be an imaginary unit such that $J\neq \pm I$. By the splitting lemma there exist holomorphic functions $f_1^{I,J},f_2^{I,J}:\Omega_I\to \C_I$ such that
$
f_I(z_I)=f_1^{I,J}(z_I)+f_2^{I,J}(z_I)J
$
for each $z_I\in\Omega_I$. As $\{1,J\}$ is a base of $\HH$ as a $\C_I$ vector space, then
$
V(f_I)=V(f_1^{I,J})\cap V(f_2^{I,J}).
$
In particular, it is enough to show that at least one between the zero sets $V(f_1^{I,J})$ and $ V(f_2^{I,J})$ is finite. As $\Omega$ is connected, then $\Omega_I$ has one or two connected components. In both cases, as $V(f_I)$ is discrete, then on each of the connected components of $\Omega_I$, at least one between $f_1^{I,J}$ and $f_2^{I,J}$ is not identically zero. By Lemma \ref{lemmafiniti}, we conclude that on each connected component of $\Omega_I$ at least one between $f_1^{I,J}$ or $f_2^{I,J}$ have finitely many zeros, so at least one between $V(f_1^{I,J})$ and $V(f_2^{I,J})$ is finite, as desired. 
\end{proof}

By the previous result and Proposition \ref{nonzeroint}, we deduce straightforwardly the following:

\begin{cor}\label{slicedomainzero}
Let $\Omega\subset A$ be a symmetric slice domain and $f:\Omega\to A$ a non-zero slice-Nash function. Then, its zero set $V(f)$ is a finite union of isolated zeros and isolated spherical zeros.
\end{cor}

\subsection{Semiregular slice-Nash functions}\label{semiregularSec}

Let $A$ be either the algebra of quaternions $\HH$ or the algebra of octonions $\O$. Let $d_A=\dim_{\R}(A)-1$ be the integer introduced in \eqref{dim}. Let $\Omega\subset A$ be an open circular set and $\Omega_0\subset \Omega$ any subset. In order to lighten the exposition we introduce the following definition. We say that $\Omega_0$ is a \textit{discrete set of singularities} if it is closed in $\Omega$ and it is a union of isolated points on the real axis and of isolated $(d_A-1)$-dimensional spheres with centre on the real axes. Clearly a discrete set of singularities $\Omega_0$ is also a circular set whose intersection with the slice $\C_I$, for each $I\in\sph_A$, is a closed and discrete subset of $\Omega_I:=\Omega\cap \C_I$ consisting of isolated points on the real axes and isolated pairs of complex conjugated points. Recall that $\mathcal{SEM}_A(\Omega)$ denotes the set of all semiregular slice functions on $\Omega$. 

We introduce the following definition: 

\begin{defn}[Semiregular slice-Nash function]
We say that a function $f$ is a \textit{semiregular slice-Nash function on} $\Omega$ if $f\in \mathcal{SEM}_A(\Omega)$ and there exists a discrete set of singularities $\Omega_0\subset \Omega$ such that $f\in\mathcal{SN}_A(\Omega\setminus \Omega_0)$. We denote by $\mathcal{SEN}_A(\Omega)$ the set of all semiregular slice-Nash functions on $\Omega$.
\end{defn}

We make use of the results of \S\ref{MeroCNash} to investigate finiteness properties of semiregular slice-Nash functions. We start by showing that slice-Nash functions do not have essential singularities. 

\begin{lem}[Singularities of slice-Nash functions]\label{Sliceonlypoles}
Let $\Omega_0$ be a subset of $\Omega$ consisting either of a single point on the real axis or a $(d_A-1)$-dimensional sphere with centre on the real axis. If $f\in\mathcal{SN}_A(\Omega\setminus\Omega_0)$, then $f\in \mathcal{SEM}_A(\Omega)$.
\end{lem}
\begin{proof}
We show the statement for $A=\HH$, as the case $A=\O$ is analogous. Let $q\in \Omega_0$. Then, there exist $I\in\sph_\HH$ and $\alpha,\beta\in \R$ such that $q=\alpha+\beta I$ (possibly with $\beta=0$ if $q$ is real). Let $J\in \sph_\HH$ be an imaginary unit such that $J\neq \pm I$. Let $\Omega_I:=\Omega\cap \C_I$ and 
$$
(\Omega\setminus \Omega_0)_I:=(\Omega\setminus \Omega_0)\cap \C_I=\Omega_I\setminus\{q,\overline{q}\}.
$$
Clearly, $\{q,\overline{q}\}=\{q\}$ if $q\in\R$. Denote by $f_I$ the restriction of $f$ to $\Omega_I\setminus\{q,\overline{q}\}$. By the splitting lemma, there exist holomorphic functions $f_1^{I,J},f_2^{I,J}:\Omega_I\setminus \{q,\overline{q}\}\to \C_I$ such that
$$
f_I(z_I)=f_1^{I,J}(z_I)+f_2^{I,J}(z_I)J
$$
for each $z_I\in\Omega_I\setminus\{q,\overline{q}\}$. By Theorem \ref{char}, $f_1^{I,J},f_2^{I,J}\in\mathcal{N}_{\C_I}(\Omega_I\setminus\{q,\overline{q}\})$. By Lemma \ref{onlypoles}, $q$ and $\overline{q}$ are not essential singularities for $f_1^{I,J}$ nor for $f_2^{I,J}$, so by Remark \ref{singchar}, $q$ is not an essential singularity for $f$, as required.
\end{proof}

We show next, that semiregular slice-Nash functions defined on a circular domain have only finitely many isolated poles and finitely many isolated spherical poles.

\begin{lem}[Finiteness of poles of semiregular slice-Nash functions]\label{Slicefinitepoles}
Assume that $\Omega$ is a circular domain. Let $f\in \mathcal{SEN}_A(\Omega)$ and let $\Omega_0\subset \Omega$ be a discrete set of singularities such that $f\in\mathcal{SN}_A(\Omega\setminus \Omega_0)$. Then there exist 
\begin{itemize}
\item finitely many isolated points $q_1,\ldots,q_k \in \Omega_0\cap \R$,
\item finitely many isolated $(d_A-1)$-dimensional spheres $\sph_{x_1},\ldots,\sph_{x_r}\subset \Omega_0$ with centre on the real axes,
\item a slice-Nash function $\Phi: \Omega\setminus (\{q_1,\ldots,q_k\}\cup\sph_{x_1}\cup\ldots\cup\sph_{x_r})\to A$ 
\end{itemize}
such that $\Phi\in\mathcal{SEM}_A(\Omega)$ and $f=\Phi|_{\Omega\setminus(\{q_1,\ldots,q_k\}\cup\sph_{x_1}\cup\ldots\cup\sph_{x_r})}$.
\end{lem}
\begin{proof}
We show the statement for $A=\HH$, as the case $A=\O$ is analogous. Let $I,J\in\sph_\HH$ be imaginary units such that $J\neq \pm I$. Observe that $(\Omega_0)_I:=\Omega_0 \cap \C_I$ is a closed and discrete subset of $\C_I$ consisting of isolated points on the real axes and isolated pairs of complex conjugated points. In particular, $(\Omega_0)_I$ is symmetric with respect to the real axes. Denote by $f_I$ the restriction of $f$ to $(\Omega\setminus \Omega_0)_I:=(\Omega\setminus\Omega_0)\cap \C_I$. By the splitting lemma, there exist holomorphic functions $f_1^{I,J},f_2^{I,J}:(\Omega\setminus \Omega_0)_I\to \C_I$ such that
$
f_I(z_I)=f_1^{I,J}(z_I)+f_2^{I,J}(z_I)J
$
for each $z_I\in (\Omega\setminus \Omega_0)_I$. By Theorem \ref{char}, $f_1^{I,J},f_2^{I,J}\in\mathcal{N}_{\C_I}((\Omega\setminus \Omega_0)_I)$. As $\Omega$ is connected, then $\Omega_I$ consists of one or two connected components. Thus, by Lemma \ref{finitepoles}, there exist finitely many points $z_I^1,\ldots,z_I^k\in(\Omega_0)_I$ and $\C_I$-Nash functions 
$
\Phi_1^{I,J},\Phi_2^{I,J}:\Omega_I\setminus\{z_I^1,\ldots,z_I^k\}\to \C_I
$
such that $\Phi_1^{I,J}$ and $\Phi_2^{I,J}$ are meromorphic on $\Omega_I$ and
$$
f_1^{I,J}(z_I)=\Phi_1^{I,J}(z_I), \quad f_2^{I,J}(z_I)=\Phi_2^{I,J}(z_I)
$$
for each $z_I\in\Omega_I\setminus\{z_I^1,\ldots,z_I^k\}$. Observe that, as $(\Omega_0)_I$ is symmetric with respect to the real axes, then $\overline{z_I}^1,\ldots,\overline{z_I}^k\in (\Omega_0)_I$. Let $\Omega_1\subset \HH$ be the circularised of $\{z_I^1,\ldots,z_I^k,\overline{z_I}^1,\ldots,\overline{z_I}^k\}$. As $\Omega_0$ is a circular set and the set $\{z_I^1,\ldots,z_I^k,\overline{z_I}^1,\ldots,\overline{z_I}^k\}$ is symmetric with respect to the real axes and contained in $(\Omega_0)_I$, then $\Omega_1\subset \Omega_0$. Moreover, $\Omega_1$ is a finite union of isolated points on the real axes (corresponding to the $z_k$ such that $z_k=\overline{z}_k$) and isolated 2-dimensional spheres with centre on the real axes (corresponding to the $z_k$ such that $z_k\neq\overline{z}_k$). 

By the representation formula, there exists a unique slice regular function $\Phi:\Omega\setminus\Omega_1\to \HH$ such that
$
\Phi(z_I)=\Phi_1^{I,J}(z_I)+\Phi_2^{I,J}(z_I)J
$
for each $z_I\in \Omega_I\setminus\{z_I^1,\ldots,z_I^k,\overline{z_I}^1,\ldots,\overline{z_I}^k\}$. Clearly $f=\Phi|_{\Omega\setminus\Omega_1}$, because $f_1^{I,J}=\Phi_1^{I,J}$ and $f_2^{I,J}=\Phi_2^{I,J}$ on $\Omega_I\setminus\{z_I^1,\ldots,z_I^k\}$. By Theorem \ref{char}, $\Phi$ is a slice-Nash function on $\Omega\setminus \Omega_1$. By Lemma \ref{Sliceonlypoles}, we conclude that $\Phi$ is semiregular, as required. 
\end{proof}

By the previous results, we deduce straightforwardly that slice-Nash functions on $\Omega\setminus \Omega_0$, where $\Omega_0\subset \Omega$ is a discrete set of singularities, are semiregular on $\Omega$ with finitely many isolated poles and finitely many isolated spherical poles on each connected component of $\Omega$.

\begin{cor}
Let $\Omega_0\subset \Omega$ be a discrete set of singularities and $f\in\mathcal{SN}_A(\Omega\setminus\Omega_0)$. Then $f$ is a semiregular slice function on $\Omega$. Moreover, if $\Omega$ is connected, then the set of poles of $f$ consists of finitely many isolated points on the real axes and finitely many isolated $(d_A-1)$-dimensional spheres with centre on the real axes.
\end{cor}

Let $\Omega_0\subset \Omega$ be a discrete set of singularities. As $\Omega\setminus \Omega_0$ is axially symmetric with respect to the real axes, it follows by Remark \ref{siminv}(i) that, if $f\in \mathcal{SN}_A(\Omega\setminus \Omega_0)$, then $f^c\in \mathcal{SN}_A(\Omega\setminus \Omega_0)$. In particular, we deduce straightforwardly by Theorems \ref{*semiregular} and \ref{*Nash} the following result, which is the analogous to the $*$-algebra $\mathcal{SEN}_A(\Omega)$ of Theorem \ref{*semiregular} for the $*$-algebra $\mathcal{SEM}_A(\Omega)$. 

\begin{thm}\label{semiregularAlgThm}
The set $\mathcal{SEN}_A(\Omega)$ of all semiregular slice-Nash functions on $\Omega$ is an alternative $*$-subalgebra of the alternative $*$-algebra $\mathcal{SEM}_A(\Omega)$ of semiregular slice functions on $\Omega$ endowed with $+$, $\cdot$, $\cdot^c$. If $A=\HH$, then $\mathcal{SEN}_\HH(\Omega)$ is associative. Moreover,
\begin{itemize}
\item if $\Omega$ is a symmetric slice domain, then $\mathcal{SEN}_A(\Omega)$ is a division algebra,
\item if $\Omega$ is a product domain, then $\mathcal{SEN}_A(\Omega)$ includes some element $f\not\equiv 0$ with $N(f)\equiv 0$. However, every element $f$ with $N(f)\not\equiv 0$ admits a multiplicative inverse within the algebra.
\end{itemize}
\end{thm}

We end this section by showing that global semiregular slice-Nash functions are actually `slice rational' functions (compare this result with Theorem \ref{global}).

\begin{thm}[Global semiregular slice-Nash functions]\label{globalsemiregular}
Let $\Omega_0\subset A$ be a discrete set of singularities and $f\in\mathcal{SN}_A(A\setminus \Omega_0)$. Then there exist two slice polynomials $P,Q\in A[\x]$ such that 
$
f(x)=(Q^{-\bullet}\cdot P)(x)
$
for each $x\in A\setminus \Omega_0$.
\end{thm}
\begin{proof}
We show the statement for $A=\O$, as the case $A=\HH$ is similar but easier. Let $\Omega_0\subset \O$ be a discrete set of singularities such that $f\in\mathcal{SN}_\O(\O\setminus \Omega_0)$. Let $I,J,L\in\sph_\O$ be imaginary units such that $\{1,I,J,IJ,L,IL,JL,I(JL)\}$ is a splitting base of $\O$ associated to $I$. The set $(\Omega_0)_I:= \Omega_0\cap \C_I$ is a closed and discrete subset of $\C_I$ consisting of isolated points on the real axes and isolated pairs of complex conjugated points. By the splitting lemma, there exist unique holomorphic functions $f_0,f_1,f_2,f_3:\C_I\setminus(\Omega_0)_I\to \C_I$ such that
\begin{equation}\label{sr1}
f_I(z_I)=f_0(z_I)+f_1(z_I)J+f_2(z_I)L+f_3(z_I)(JL)
\end{equation}
for each $z_I\in \C_I\setminus(\Omega_0)_I$. By Theorem \ref{char}, $f_0,f_1,f_2,f_3\in\mathcal{N}_{\C_I}(\C_I\setminus(\Omega_0)_I)$. By Theorem \ref{globalmer}, for each $k=0,1,2,3$ there exist polynomials $P_k,Q_k\in\C_I[\z_I]$ such that 
$$
f_k(z_I)=\frac{P_k(z_I)}{Q_k(z_I)}
$$
for each $z_I\in \C_I\setminus(\Omega_0)_I$. Up to multiply the numerators and the denominators for suitable non-zero polynomials and substitute the $P_k$'s with the corresponding products, we may assume that there exists a polynomial $Q\in \C_I[\z_I]$ such that
\begin{equation}\label{sr2}
f_k(z_I)=\frac{P_k(z_I)}{Q(z_I)}
\end{equation}
for each $k=0,1,2,3$ and each $z_I\in \C_I\setminus(\Omega_0)_I$. By \eqref{sr1} and  \eqref{sr2}, we deduce that
$$
f_I(z_I)=\frac{1}{Q(z_I)}(P_0(z_I)+P_1(z_I)J+P_2(z_I)L+P_3(z_I)(JL))
$$
for each $z_I\in\C_I\setminus(\Omega_0)_I$. Consider next the slice polynomial function $P\in\O[x]$ defined as
$$
P(x):=P_0(x)+(P_1\cdot J)(x)+(P_2\cdot L)(x)+(P_3\cdot(JL))(x)
$$
for each $x\in \O$, where $\cdot$ denotes, as usual, the slice product (see also Remark \ref{sliceprodoct}). We have
$$
(Q^{-\bullet}\cdot P)(z_I)=\frac{1}{Q(z_I)}(P_0(z_I)+(P_1J)(z_I)+(P_2L)(z_I)+(P_3(JL))(z_I))
$$
for each $z_I\in\C_I\setminus(\Omega_0)_I$. In particular, the slice regular functions $f$ and $Q^{-\bullet}\cdot P$ coincide on $\C_I\setminus(\Omega_0)_I$. As $\O\setminus\Omega_0$ is the circularised of $\C_I\setminus(\Omega_0)_I$, then by the representation formula, we conclude that
$
f(x)=(Q^{-\bullet}\cdot P)(x)
$
for each $x\in \O\setminus\Omega_0$, as required.
\end{proof}

In the following example, we show that the previous result is no longer true if we consider slice-Nash functions defined on proper open circular subsets $\Omega\subsetneq A$. The example is for $A=\HH$, but it can easily be adapted to the case $A=\O$.

\begin{example}
Let $\Omega:=\HH\setminus \{q\in \R: q\leq0\}$ and let $J\in\sph_\HH$ be an imaginary unit. Let $g:\Omega\to \HH$ be the slice regular branch of $\sqrt{\q}J$ such that $g(1)=J$. Consider the slice regular function $f:\Omega\setminus\{1\}\to \HH$ defined as $f(q):=(q-1)^{-\bullet}\cdot g(q)$ for each $q\in \Omega\setminus\{1\}$. It is clear that $f$ is semiregular on $\Omega$. By Example \ref{esempislice}(iii), $g$ is a slice-Nash function on $\Omega$. Thus, by Remark \ref{remarkprodotto}(ii), $f$ is a slice-Nash function on $\Omega\setminus\{1\}$. In particular, $f$ is a semiregular slice-Nash function on $\Omega$, but $f$ is not of the form $Q^{-\bullet}\cdot P$, where $P,Q\in\HH[\q]$. \hfill$\sqbullet$
\end{example}

\subsection{Polynomial bounds at infinity for slice-Nash functions}

In this section we make use of Proposition \ref{polboundC} to show that the norm of a slice-Nash function is polynomially bounded at infinity. Let $A$ be either the algebra of quaternions $\HH$ or the algebra of octonions $\O$. Recall that the norm of $x\in A$ is defined as $\|x\|:=\sqrt{n(x)}=\sqrt{xx^c}$ and coincide with the Euclidean norm of $\HH\simeq\R^4$ and $\O\simeq \R^8$ respectively. 

\begin{prop}[Polynomial bounds for slice-Nash functions]\label{slicepolbound}
Let $\Omega\subset A$ be an unbounded circular domain and $f\in \mathcal{SN}_A(\Omega)$. Then, there exist an integer $m\geq 0$ and constants $C,R\geq 1$ such that
$$
\|f(x)\|\leq C(1+\|x\|^m)
$$
for each $x\in \Omega$ such that $\|x\|\geq R$. 
\end{prop}
\begin{proof}
We show the statement for $A=\O$, as the case $A=\HH$ is similar but easier. Let $I\in\sph_\O$ be an imaginary unit. As $\Omega$ is a circular domain, then $\Omega_I:=\Omega\cap \C_I$ is an open subset of $\C_I$ with one or two connected components. We may assume that $\Omega_I$ has two connected components $\Omega_I^1$ and $\Omega_I^2$, otherwise the proof is similar but easier. Denote $f_I:=f|_{\Omega_I}$. Let $J,L\in\sph_\O$ be any imaginary units such that $\{1,I,J,IJ,L,IL,JL,I(JL)\}$ is a splitting base of $\O$ associated to $I$. By the splitting lemma, there exist unique holomorphic functions $f_0,f_1,f_2,f_3:\Omega_I\to \C_I$ such that 
\begin{equation}\label{splitstima}
f_I(z_I)=f_0(z_I)+f_1(z_I)J+f_2(z_I)L+f_3(z_I)(JL)
\end{equation}
for each $z_I\in\Omega_I$. By Theorem \ref{char}, $f_0,f_1,f_2,f_3\in \mathcal{N}_{\C_I}(\Omega_I)$. By Proposition \ref{polboundC}, for each $s=1,2$ and each $r=0,1,2,3$ there exist an integer $m_{s,r}\geq 0$ and constants $C_{s,r},R_{s,r}\geq 1$ such that
\begin{equation}\label{stimabound}
|f_r(z_I)|\leq C_{s,r}(1+|z_I|^{m_{s,r}}) 
\end{equation}
for each $z_I\in \Omega_I^s$ such that $|z_I|\geq R_{s,r}$. Define
$$
m:= \max_{r=0,1,2,3}\max_{s=1,2} \{m_{s,r}\}, \quad C:=8\max_{r=0,1,2,3}\max_{s=1,2} \{C_{s,r}\}\quad \text{\and} \quad R:=\max_{r=0,1,2,3}\max_{s=1,2}\{R_{r,s}\}\geq 1.
$$
We have
\begin{equation}\label{eqbound3}
 C_{s,r}(1+|z_I|^{m_{s,r}})\leq \frac{1}{8}C(1+|z_I|^m)
\end{equation}
for each $z_I\in \Omega^s_I$ such that $|z_I|\geq R\geq 1$ and each $s=1,2$ and $r=0,1,2,3$.

Observe that for each $x\in \O$ and each $K\in\sph_\O$, it holds $\|xK\|=\|x\|$. This follows by the fact that the norm $\|\cdot\|$ on $\O$ is multiplicative (due to the absence of zero divisors) and $\|K\|=1$. In particular, by \eqref{splitstima}, \eqref{stimabound} and \eqref{eqbound3}, we deduce that
\begin{align}\label{stimabound2}
\begin{split}
\|f_I(z_I)\|&\leq \|f_0(z_I)\|+\|f_1(z_I)J\|+\|f_2(z_I)L\|+\|f_3(z_I)(JL)\|\\
&=\|f_0(z_I)\|+\|f_1(z_I)\|+\|f_2(z_I)\|+\|f_3(z_I)\|\\
&\leq \frac{1}{8}C(1+|z_I|^m)+\frac{1}{8}C(1+|z_I|^m)+\frac{1}{8}C(1+|z_I|^m)+\frac{1}{8}C(1+|z_I|^m)\\
&=\frac{1}{2}C(1+|z_I|^m)
\end{split}
\end{align}
for each $z_I\in\Omega_I$ such that $|z_I|\geq R$, where $|z_I|:=(z_I\overline{z_I})^{\tfrac{1}{2}}$ denotes the usual complex module of $z_I\in\C_I$. 

For each $J\in\sph_\O$ and for each $z_J=\alpha+\beta J\in\Omega_J:=\Omega\cap \C_J$, denote $z_I:=\alpha+\beta I\in\Omega_I$. That is, $z_I=\phi_I(\phi_J^{-1}(z_J))$, where for each $I\in\sph_\O$, $\phi_I:\C\to \C_I$ is the $*$-isomorphism introduced in \ref{*iso}. By the representation formula, for each $J\in \sph_\O$  we have 
$$
f_J(z_J)=\frac{1}{2}(f_I(z_I)+f_I(\overline{z_I}))-\frac{J}{2}(I(f_I(z_I)-f_I(\overline{z_I})))
$$
for each $z_J\in \Omega_J:=\Omega\cap \C_J$. Recall that $\Omega_I$ is symmetric with respect to the real axes, so by \eqref{stimabound2}, we have
$
|f_I(\overline{z_I})|\leq \tfrac{1}{2}C(1+|\overline{z_I}|^m)
$
for each $z_I\in\Omega_I$. By the representation formula and \eqref{stimabound2}, we deduce 
\begin{align*}
\|f(z_J)\|&=\|f_J(z_J)\|\leq \frac{1}{2}(\|f_I(z_I)\|+\|f_I(\overline{z_I}))\|+\|J(I f_I(z_I))\|+\|J(If_I(\overline{z_I}))\|)\\
&=\frac{1}{2}(2\|f_I(z_I)\|+2\|f_I(\overline{z_I}))\|)= \|f_I(z_I)\|+\|f_I(\overline{z_I}))\|\\
&\leq \frac{1}{2}C(1+|z_I|^m)+ \frac{1}{2}C(1+|\overline{z_I}|^m)=C(1+|z_I|^m).
\end{align*}
for each $z_I\in\Omega_I$ such that $|z_I|\geq  R$. As for each $x\in \Omega$, there exists $J\in \sph_\O$ such that $x\in\Omega_J$, then 
$$
\|f(x)\|\leq C(1+\|x\|^m)
$$
for each $x\in \Omega$ such that $\|x\|\geq R$, as required. 
\end{proof}

The following example shows that the converse of the previous result is false. The example is for $\O$, but it can be easily adapted to $\HH$.

\begin{example}
Let $D:=\C\setminus\{z\in \R, \, z\leq 0\}$  and let $\gamma\in \R\setminus \Q$ be an irrational number such that $\gamma>0$ (here $\Q$ denotes, as usual, the field of rational numbers). Let $z^{\gamma}$ be the holomorphic branch of $\z^{\gamma}:=e^{\gamma\log\z}$ such that $1^{\gamma}=1$. Consider the holomorphic stem function 
$$
F:D\to \O\otimes_\R\C, \quad z\mapsto z^{\gamma}:=1\otimes z^{\gamma}.
$$
Let $\Omega:=\Omega_D\subset \O$ be the circularised of $D$ and $f:=\mathcal{I}(F):\Omega\to \O$ the slice regular function induced by $F$. By Example \ref{irrationalpower}, $F$ is not a stem-Nash function on $D$, so $f$ is not a slice-Nash function on $\Omega$. Let $n\geq 1$ be an integer such that $\gamma<n$ and $I\in \sph_\O$ an imaginary unit. Then
$$
\|f(z_I)\|=\|\Re(z^{\gamma})+I\Im(z^{\gamma})\|=|z^{\gamma}|\leq1+|z|^n=1+\|z_I\|^n.
$$
for each $z_I:=\alpha+\beta I\in \C_I$, where $\alpha,\beta\in \R$ and $z:=\alpha+\iota \beta\in \C$. We deduce that
$
\|f(x)\|\leq 1+\|x\|^n
$
for each $x\in \Omega$. \hfill$\sqbullet$
\end{example}

Let $\Omega\subset A$ be an unbounded circular domain and $f\in \mathcal{SEN}_A(\Omega)$. Let $\Omega_0\subset \Omega$ be a discrete set of singularities such that $f\in\mathcal{SN}_A(\Omega\setminus \Omega_0)$. By Lemma \ref{Slicefinitepoles}, there exist 
\begin{itemize}
\item finitely many isolated points $q_1,\ldots,q_k \in \Omega_0\cap \R$,
\item finitely many isolated $(d_A-1)$-dimensional spheres $\sph_{x_1},\ldots,\sph_{x_r}\subset \Omega_0$ with centre on the real axes,
\item a slice-Nash function $\Phi: \Omega\setminus (\{q_1,\ldots,q_k\}\cup\sph_{x_1}\cup\ldots\cup\sph_{x_r})\to A$ 
\end{itemize}
such that $\Phi\in\mathcal{SEM}_A(\Omega)$ and $f=\Phi|_{\Omega\setminus(\{q_1,\ldots,q_k\}\cup\sph_{x_1}\cup\ldots\cup\sph_{x_r})}$. In particular, there exists $R'>1$ such that $\Phi$ is slice-Nash on $D\setminus B(0,R')$. Thus, by Proposition \ref{slicepolbound}, there exist an integer $m\geq 0$ and constants $C,R''\geq1$ such that 
$$
\|\Phi(x)\|\leq C(1+\|x\|^m)
$$
for each $x\in \Omega$ such that $\|x\|\geq \max\{R',R''\}$. In particular, we deduce the following:

\begin{cor}
Let $\Omega\subset A$ be an unbounded circular domain and $f\in \mathcal{SEN}_A(\Omega)$. Let $\Omega_0\subset \Omega$ be a discrete set of singularities such that $f\in\mathcal{SN}_A(\Omega\setminus \Omega_0)$. Then, there exist an integer $m\geq 0$ and constants $C,R\geq1$ such that 
$$
\|f(x)\|\leq C(1+\|x\|^m)
$$
for each $x\in \Omega\setminus\Omega_0$ such that $\|x\|\geq R$.
\end{cor}

\bibliographystyle{amsalpha}

\begin{thebibliography}{FFQU}

\bibitem{ap} G. Alon, E. Paran: A quaternionic Nullstellensatz. \textit{J. Pure Appl. Algebra} \textbf{225} (2021), no. 4, Paper No. 106572, 9 pp.

\bibitem{ap2} G. Alon, E. Paran: A central quaternionic Nullstellensatz. \textit{J. Algebra} \textbf{574} (2021), 252--261.

\bibitem{alb} A. Altavilla, C. Bisi: Log-biharmonicity and a Jensen formula in the space of quaternions. \textit{Ann. Acad. Sci. Fenn. Math.} \textbf{44} (2019), no. 2, 805–839.

\bibitem{ab} D. Angella, C. Bisi: Slice-quaternionic Hopf surfaces. \textit{J. Geom. Anal.} \textbf{29} (2019), no. 3, 1837–1858.

\bibitem{ar} M. Artin: On the solutions of analytic equations. \textit{Invent. Math.} \textbf{5} (1968), 277--291.

\bibitem{ba} J. C. Baez: The octonions. \textit{Bull. Amer. Math. Soc. (N.S.)} \textbf{39} (2002), no. 2, 145--205.

\bibitem{bd2} C. Bisi, A. De Martino: On the quadratic cone of $\R_3$. \textit{Preprint} (2021) {\tt arXiv:2109.14582}

\bibitem{bd} C. Bisi, A. De Martino: On Brolin's theorem over the quaternions. \textit{Indiana Univ. Math. J.} \textbf{71} (2022), no. 4, 1675--1705.

\bibitem{bg} C. Bisi, G. Gentili: On quaternionic tori and their moduli space. \textit{J. Noncommut. Geom.} \textbf{12} (2018), no. 2, 473--510.

\bibitem{bs} C. Bisi, C. Stoppato: The Schwarz–Pick lemma for slice regular function. \textit{Indiana Univ. Math. J.} \textbf{61} (2012), no. 1, 297--317.

\bibitem{bs2} C. Bisi, C. Stoppato: Landau’s theorem for slice regular functions on the quaternionic unit ball. \textit{Internat. J. Math.} \textbf{28}(3) (2017).

\bibitem{bw} C. Bisi, J. Winkelmann: On a quaternionic Picard theorem. \textit{Proc. Amer. Math. Soc.} Ser. B \textbf{7} (2020), no. 3, 1750017, 21 pp.

\bibitem{bw2} C. Bisi, J. Winkelmann: On Runge pairs and topology of axially symmetric domains.\textit{J. Noncommut. Geom.} \textbf{15} (2021), no. 2, 713--734.

\bibitem{bw3} C. Bisi, J. Winkelmann: The harmonicity of slice regular functions. \textit{J. Geom. Anal.} \textbf{31} (2021), no. 8, 7773–7811.

\bibitem{bw6} C. Bisi, J. Winkelmann: Invariants and automorphisms for slice regular functions. \textit{J. Noncommut. Geom.} (2025). Published online first. {\tt DOI 10.4171/JNCG/615}

\bibitem{bcr} J. Bochnak, M. Coste, M.-F. Roy: {Real algebraic geometry.} {\em Ergeb. Math. Grenzgeb.} (3), {\bf36} Springer-Verlag, Berlin, (1998).


\bibitem{ca} A. Carbone: Holomorphic functions with Nash real part. \textit{Preprint} (2025) {\tt arXiv:2507.17387}

\bibitem{csast} F. Colombo, I. Sabadini, D.C. Struppa: Slice monogenic functions. 	\textit{Israel J. Math.} \textbf{171} (2009), 385--403.

\bibitem{cosm} J.H. Conway, D.A. Smith: \textit{On quaternions and octonions: their geometry, arithmetic, and symmetry. A K Peters, Ltd.,} Natick, MA (2003). 

\bibitem{crs1} M. Coste, J. M. Ruiz, M. Shiota: Approximation in compact Nash manifolds. \textit{Amer. J. Math.} \textbf{117} (1995), no. 4, 905--927.


\bibitem{Cox} D.A. Cox, J. Little, D. O’Shea - \textit{Ideals, varieties, and algorithms.} An introduction to computational algebraic geometry and commutative algebra. \textit{Undergrad. Texts Math.}, Springer, Cham, (2015).

\bibitem{cu} C.G. Cullen: An integral theorem for analytic intrinsic functions on quaternions. \textit{Duke Math. J.} \textbf{32} (1965), 139--148.


\bibitem{fu} R. Fueter: Die Funktionentheorie der Differentialgleichungen $\Delta u=0$ und $\Delta\Delta u=0$ mit vier reellen Variablen. \textit{Comment. Math. Helv.} \textbf{7} (1934), no. 1, 307--330.


\bibitem{gss} G. Gentili, A. Gori, G. Sarfatti: A direct approach to quaternionic manifolds. \textit{Math. Nachr.} \textbf{290} (2017), no. 2-3, 321--331.

\bibitem{gss2} G. Gentili, A. Gori, G. Sarfatti: Quaternionic toric manifolds. \textit{J. Symplectic Geom.} \textbf{17} (2019), no. 1, 267--300.


\bibitem{gssv} G. Gentili, C. Stoppato, D.C. Struppa, F. Vlacci: Recent developments for regular functions of a hypercomplex variable.  \textit{Hypercomplex Analysis}, Trends Math., Birkhäuser, Basel (2009), 165--185.

\bibitem{gst} G. Gentili, C. Stoppato: Zeros of regular functions and polynomials of a quaternionic variable. \textit{Michigan Math. J.} \textbf{56} (2008), no. 3, 655--667.

\bibitem{gst3} G. Gentili, C. Stoppato: The open mapping theorem for regular quaternionic functions. \textit{Ann. Sc. Norm. Super. Pisa Cl. Sci.} (5)\textbf{8} (2009), no. 4, 805--815.

\bibitem{gst2}G. Gentili, C. Stoppato: The zero sets of slice regular functions and the open mapping theorem. In \textit{Hypercomplex Analysis and Applications.} Trends Math. Birkh\"user/Springer Basel AG, Basel (2011), 95--107.

\bibitem{gsst} G. Gentili, C. Stoppato, D. C. Struppa: \textit{Regular functions of a quaternionic variable. Springer Monogr. Math.} Springer, Cham (2022).

\bibitem{gs} G. Gentili, D.C. Struppa: A new approach to Cullen-regular functions of a quaternionic variable. \textit{C. R. Math. Acad. Sci. Paris} \textbf{342} (2006), no. 10, 741--744.

\bibitem{gs2} G. Gentili, D.C. Struppa: A new theory of regular functions of a quaternionic variable. \textit{Adv. Math.} \textbf{216} (2007), no. 1, 279--301.

\bibitem{gs3} G. Gentili, D.C. Struppa: Regular functions on the space of Cayley numbers. \textit{Rocky Mountain J.Math.} \textbf{40} (2010), no. 1, 225--241.

\bibitem{gv} G. Gentili, I. Vignozzi: The Weierstrass factorization theorem for slice regular functions over the quaternions. \textit{Ann. Glob. Anal. Geom.} \textbf{40} (2011)., no. 4, 435--466.

\bibitem{gh} R. Ghiloni: Algebraic obstructions and a complete solution of a rational retraction problem. \textit{Proc. Amer. Math. Soc.} \textbf{130} (2002), no. 12, 3525--3535.

\bibitem{gh1} R. Ghiloni: Rigidity and moduli space in real algebraic geometry.\textit{Math. Ann.} \textbf{335} (2006), no. 4, 751--766.

\bibitem{gh2} R. Ghiloni: On the space of morphisms into generic real algebraic varieties. \textit{Ann. Scuola Norm. Sup. Pisa Cl. Sci.} \textbf{5}(5) (2006), no. 3, 419--438.

\bibitem{gp0} R. Ghiloni, A. Perotti: A New Approach to Slice Regularity on Real Algebras. In \textit{Hypercomplex Analysis and Applications.} Trends Math. Birkh\"user/Springer Basel AG, Basel (2011), 109--123.

\bibitem{gp} R. Ghiloni, A. Perotti: Slice regular functions on real alternative algebras. \textit{Adv. Math.} {\bf226} (2011), no. 2, 1662--1691.

\bibitem{gp3} R. Ghiloni, A. Perotti: Zeros of regular functions of quaternionic and octonionic variable: a division lemma and the camshaft effect. \textit{Ann. Mat. Pura Appl.} (4)\textbf{190} (2011), no. 3, 539--551.

\bibitem{gp2} R. Ghiloni, A. Perotti: Power and spherical series over real alternative $*$-algebras. \textit{Indiana Univ. Math. J.} {\bf63} (2014), no. 2, 495--532.

\bibitem{gp5} R. Ghiloni, A. Perotti: Global differential equations for slice regular functions. \textit{Math. Nachr.} \textbf{287} (2014), no. 5--6, 561--573.

\bibitem{gp6} R. Ghiloni, A. Perotti: Slice regular functions in several variables. \textit{Math. Z.} \textbf{302} (2022), no. 1, 295--351. 

\bibitem{gps1} R. Ghiloni, A. Perotti, C. Stoppato: Singularities of slice regular functions over real alternative $*$-algebras. \textit{Adv. Math.} \textbf{305} (2017), 1085--1130.

\bibitem{gps3} R. Ghiloni, A. Perotti, C. Stoppato: The algebra of slice functions. \textit{Trans. Amer. Math. Soc.} \textbf{369} (2017), no. 7, 4725--4762.

\bibitem{gps}R. Ghiloni, A. Perotti, C. Stoppato: Division algebras of slice functions. \textit{Proc. Roy. Soc. Edinburgh Sect. A} \textbf{150} (2020), no. 4, 2055--2082.

\bibitem{gps2}R. Ghiloni, A. Perotti, C. Stoppato: Slice regular functions and orthogonal complex structures over $\R^8$. \textit{J. Noncommut. Geom.} \textbf{16} (2022), no. 2, 637–676.

\bibitem{gsv} A. Gori, G. Sarfatti, F. Vlacci: Zero sets and Nullstellensatz type theorems for slice regular quaternionic polynomials. \textit{Linear Algebra Appl.} \textbf{685} (2024), 162--181.

\bibitem{gr} R. Gunning, H. Rossi: \textit{Analytic functions of several complex varieties.} Prentice-Hall, Englewood Cliffs, NJ, (1965).

\bibitem{lam} T.Y. Lam: \textit{A First Course in Noncommutative Rings. Grad. Texts in Math.} \textbf{131}, Springer-Verlag, New York (1991).

\bibitem{le} L. Lempert: Algebraic approximations in analytic geometry. \textit{Invent. Math.} \textbf{121} (1995), no. 2, 335--353.


\bibitem{Na} R. Narasimhan, Y. Nievergelt: \textit{Complex analysis in one variable.} \textit{Birkh\"auser Boston Inc}, Boston, MA (2001).

\bibitem{nas} J. Nash: Real algebraic manifolds. \textit{Ann. of Math.} (2) \textbf{56} (1952), 405--421.

\bibitem{ku} A.G. Kurosh: \textit{Lectures on general algebra}. Translated by K. A. Hirsch. Chelsea Publishing Co., New York(1965).


\bibitem{sh} M. Shiota: \textit{Nash manifolds.} {\em Lecture Notes in Math.}, {\bf1269} Springer-Verlag, Berlin, (1987).

\bibitem{sa} E. Savi: On the first-order theories of quaternions and octonions. \textit{Preprint} (2024) {\tt arXiv:2404.04976}

\bibitem{s} C. Stoppato: Poles of regular quaternionic functions. \textit{Complex Var. Elliptic Equ.} \textbf{54} (2009), no. 11, 1001--1018.

\bibitem{st} E.L. Stout: Algebraic domains on Stein manifolds. \textit{Proceedings of the conference on Banach algebras and several complex variables (New Haven, Conn., 1983)} 259--266. Contemp. Math., \textbf{32} American Mathematical Society, Providence, RI (1984). 

\bibitem{su} A. Sudbery: Quaternionic analysis. \textit{Math. Proc. Cambridge Philos. Soc.} \textbf{85} (1979), no. 2, 199--224.


\bibitem{to} A. Tognoli: Su una congettura di Nash. \textit{Ann. Scuola Norm. Sup. Pisa Cl. Sci.} (3)\textbf{27} (1973), 167--185.

\bibitem{t} P. Tworzewski: Intersections of analytic sets with linear subspaces. \textit{Ann. Scuola Norm. Sup. Pisa Cl. Sci.} (4){\bf17} (1990), no. 2, 227--271.

\bibitem{w} J.P. Ward: \textit{Quaternions and Cayley Numbers: Algebra and Applications}. \textit{ Math. Appl.} \textbf{403}, Kluwer Academic Publishers Group, Dordrecht (1997).

\end{thebibliography}

\end{document}